\documentclass[12pt,letterpaper]{amsart}

\newcommand{\indentalign}{\hspace{0.3in}&\hspace{-0.3in}}
\newcommand{\la}{\langle}
\newcommand{\ra}{\rangle}

\newcommand{\ds}{\displaystyle}
\newcommand{\sech}{\operatorname{sech}}
\newcommand{\defeq}{\stackrel{\rm{def}}{=}}
\newcommand{\supp}{\operatorname{supp}}

\newcommand{\spn}{\operatorname{span}}

\usepackage{amssymb}
\usepackage{amsthm}
\usepackage{amsxtra}
\usepackage{graphicx}
\usepackage{color}

\usepackage{float}

\newtheorem{theorem}{Theorem}
\newtheorem{definition}[theorem]{Definition}
\newtheorem{proposition}[theorem]{Proposition}
\newtheorem{lemma}[theorem]{Lemma}

\theoremstyle{remark}
\newtheorem{remark}[theorem]{Remark}

\numberwithin{equation}{section}

\numberwithin{theorem}{section}

\numberwithin{table}{section}

\title[3D ZK asymptotic stability]{Asymptotic stability of solitary waves of the 3D quadratic Zakharov-Kuznetsov equation}

\author[L. G. Farah]{Luiz Gustavo Farah}
\address{Department of Mathematics\\UFMG\\Belo Horizonte, Brazil}
\curraddr{}
\email{lgfarah@gmail.com}
\thanks{}

\author[J. Holmer]{Justin Holmer}
\address{Department of Mathematics\\Brown University\\Providence, RI, USA}
\email{holmer@math.brown.edu}
\thanks{}

\author[S. Roudenko]{Svetlana Roudenko}
\address{Department of Mathematics\\Florida International University\\Miami, FL, USA}
\curraddr{}
\email{sroudenko@fiu.edu}
\thanks{} 

\author[K. Yang]{Kai Yang}
\address{Department of Mathematics\\Florida International University\\Miami, FL, USA}
\curraddr{}
\email{yangk@fiu.edu}
\thanks{}

\date{}


\begin{document}

\begin{abstract}
We consider the quadratic Zakharov-Kuznetsov equation
$$
\partial_t u + \partial_x \Delta u + \partial_x u^2 =0
$$
on $\mathbb{R}^3$.  A solitary wave solution is given by $Q(x-t,y,z)$, where $Q$ is the ground state solution to  $-Q + \Delta Q + Q^2 =0$.
We prove the asymptotic stability of these solitary wave solutions.   Specifically, we show that initial data close to $Q$ in the energy space, evolves to a solution that, as $t\to\infty$, converges to a rescaling and shift of $Q(x-t,y,z)$ in $L^2$ in a rightward shifting region $x> \delta t -\tan \theta \sqrt{y^2+z^2} $  for  $0 \leq \theta \leq \frac{\pi}{3}-\delta$.  
\end{abstract}

\maketitle

\section{Introduction}\label{S:introduction}
We consider the 3D quadratic Zakharov-Kuznetsov equation 
$$
\text{(3D ZK)} \qquad  \qquad  \partial_t u + \partial_x \Delta u + \partial_x u^2 =0, \qquad \qquad 
$$
where $u=u(\mathbf{x},t)$, for $\mathbf{x} = (x,y,z) \in \mathbb{R}^3$, $t \in \mathbb{R}$.
This equation is a natural multidimensional generalization of the well-known Korteweg-de Vries (KdV) equation, which models weakly nonlinear waves in shallow water. The 3D ZK equation was originally proposed by Zakharov and Kuznetsov to describe weakly magnetized ion-acoustic waves in a low-pressure magnetized plasma and the typical reference for that is \cite{ZK1974}. Actually the original announcement and formal derivation from hydrodynamics appeared in 1972 in a preprint of the Soviet Academy of Sciences \cite{ZK1972}, see Figure \ref{fig:ZK-1972}, where the authors write ``until now in hydrodynamics and plasma physics the attention was paid only to the one-dimensional solitons".  In that paper (and its JETP 1974 analog) the discussion of stability of the 3D solitons appeared by giving an argument that a Lyapunov type functional ($E+\lambda M$) is minimized on the soliton.  

\begin{figure}[ht]
\includegraphics[height=4in]{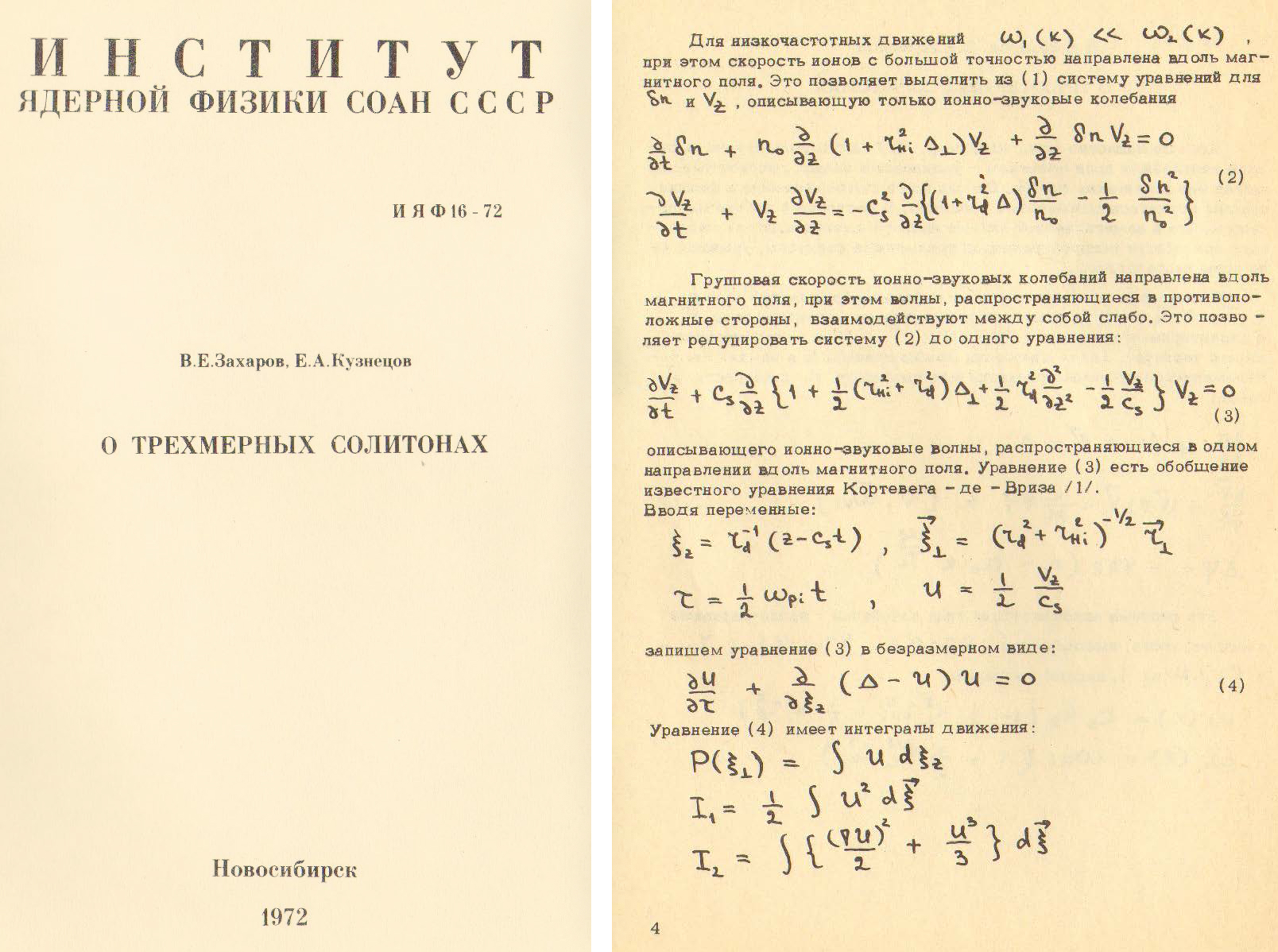}
\caption{\small{V.E.Zakharov and E.A.Kuznetsov, ``On three-dimensional solitons", Siberian branch of USSR Academy of Sciences, Novosibirsk 1972; the title page and p.4 with the derivation of the equation and conserved quantities.} }
\label{fig:ZK-1972}
\end{figure}

The formal and then rigorous derivation of the 3D Zakharov-Kuznetsov equation as a long-wave small-amplitude limit of the Euler-Poisson system in the cold-plasma approximation  was done in \cite{LS} and \cite{LLS}, respectively. Other derivations exist as well -- see, for example, references in \cite{LLS, FHRY, FHR3}. 

Unlike KdV and other generalizations such as Kadomtsev-Petiashvili or Benjamin-Ono equations, the Zakharov-Kuznetsov equation is not completely integrable. However, it has a Hamiltonian structure with three conserved quantities: during their lifespan, solutions $u(t)$ (with sufficient decay) conserve energy (Hamiltonian), $L^2$-norm (often called mass) and the integral:
\begin{equation}\label{MC}
M(u(t))=\int_{\mathbb{R}^3} u^2(t)\, d {\bf{x}} = M(u(0)),
\end{equation}
\begin{equation}\label{EC}
E(u(t))=\dfrac{1}{2}\int_{\mathbb{R}^3}|\nabla u(t)|^2\;d{\bf{x}} - \dfrac{1}{3}\int_{\mathbb{R}^3} u^{3}(t)\;d{\bf{x}} = E(u(0)),
\end{equation}
\begin{equation}\label{L1-inv}
\int_{\mathbb{R}} u(t,{\bf{x}}) \, dx = \int_{\mathbb{R}} u(0, {\bf{x}}) \, dx,
\end{equation}
where the last conservation is obtained by integrating the original equation on $\mathbb{R}$ in the first coordinate $x$.

The equation has a family of traveling waves called solitary waves (sometimes called solitons, although the model is not integrable), moving only in the positive $x$-direction:
$$
u(t, {\bf{x}}) = Q_\lambda(x - \lambda^2 t, y, z)
$$
with $Q_\lambda({\bf{x}}) \to 0$ as $|{\bf x}| \to \infty$, and $Q_\lambda$ is the dilation of the ground state $Q$:
$$
Q_\lambda({\bf{x}}) = \lambda^2 \, Q(\lambda \, {\bf{x}})
$$
with $Q$ being the unique radial positive solution in $H^1(\mathbb{R}^3)$ of the well-known nonlinear elliptic equation $-\Delta Q+Q  - Q^p = 0$.  It is well-known that $Q \in C^{\infty}(\mathbb{R}^3)$, $\partial_r Q(r) <0$ for any $r = |{\bf{x}}|>0$ and for any multi-index $\alpha$
\begin{equation}\label{prop-Q}
|\partial^\alpha Q(x, y, z)| \leq c(\alpha) \, e^{-r} \quad \mbox{for any}\quad {\bf{x}} \in \mathbb{R}^3.
\end{equation}

The orbital stability of these traveling waves was proved by de Bouard \cite{deB}, where she followed the KdV argument of Grillakis, Shatah \& Strauss \cite{GSS}, while considering solutions in weighted spaces.  The more delicate question of asymptotic stability for ZK in dimension $d\geq 2$ was first considered by C\^{o}te, Mu\~{n}oz, Pilod \& Simpson \cite{CPMS}.  They covered the case of the 2D ZK, but their result does not apply to the 3D ZK, since their Liouville theorem fails (e.g., due to their choice of orthogonality conditions and manner of addressing the local virial estimate).   The present paper fills this gap by establishing asymptotic stability for the physical case of the 3D ZK.

The Cauchy problem for the 3D ZK equation has been studied by several authors.  First,  local well-posedness can be established via the classical Kato method in $H^s$ for $s>\frac52$.  This was improved by Linares \&  Saut  \cite{LS}, who obtained the local well-posedness in $H^s$ with $s>\frac98$ following the method of Kenig \cite{K2004}, which was then further improved by Ribaud \& Vento \cite{RV} down to $H^s$ with $s>1$.  The global well-posedness in  $H^s$, $s>1$ was established by Molinet \& Pilod \cite{MP}.   At the time we started writing the present paper, this was the best result, and therefore, we arranged our argument to establish the statement of asymptotic stability as formulated below in Theorem \ref{T:main}  for certain weak solutions that we termed Class B (as defined in Definition \ref{D:ClassB}) that were \emph{assumed} to be orbitally stable.  The best known well-posedness results at the time (Ribaud \& Vento \cite{RV}, Molinet \& Pilod \cite{MP}), combined with the orbital stability argument of de Bouard \cite{deB}, gave a corollary that solutions in $H^s$, $s>1$, with initial data close to a $Q$ with respect to the $H^1$ norm, were $H^1$ orbitally stable, thus, meeting the hypotheses of our Theorem \ref{T:main}, and allowing for the conclusion of $H^1$ asymptotic stability for such solutions.   Recently and after we had nearly completed the present paper, Herr \& Kinoshita \cite{HK} announced a proof of local well-posedness for the 3D ZK in $H^s$ for $s>-\frac12$.  This, when combined with the orbital stability argument of de Bouard \cite{deB} establishes that $H^1$ solutions, initially close to $Q$ in $H^1$, are orbitally stable, thus, meeting the hypotheses of our Theorem \ref{T:main}.  Therefore, we can now state an unconditional version of asymptotic stability as our main result:

\begin{theorem}[main theorem]
For $\alpha \ll 1$, the following statement holds:  if the initial condition $u_0\in H_{\mathbf{x}}^1$ and 
\begin{equation}
\label{E:initial}
\|u_0 - Q \|_{H_{\mathbf{x}}^1} \leq \alpha,
\end{equation}
then the corresponding solution $u(\mathbf{x},t)$ to the 3D ZK  is \emph{asymptotically stable} in the following sense
\begin{enumerate}
\item (orbital stability) there exist trajectories $c(t)>0$ and $\mathbf{a}(t)\in \mathbb{R}^3$ such that 
$$
\| c(t)^2 u(c(t)\mathbf{x} + \mathbf{a}(t),t) - Q(\mathbf{x}) \|_{H_{\mathbf{x}}^1} \lesssim \alpha,
$$
\item (convergence of trajectories) there exists $c_*$ such that $|c_*-1| \lesssim \alpha$ such that 
$$
c(t) \to c_* \,,  \quad \text{and} \quad \mathbf{a}'(t) \to c_*^{-2} \mathbf{i} \quad \text{as} \quad t\to +\infty, 
$$
\item (weak convergence as $t\nearrow \infty$) there holds
\begin{equation}
\label{E:main-weak}
c(t)^2 u(c(t) \mathbf{x} + \mathbf{a}(t),t) \rightharpoonup Q(\mathbf{x}) \quad \text{ (weakly) in } H_{\mathbf{x}}^1 \text{ as }t\to +\infty,
\end{equation}
\item ($L^2$ strong convergence in  conic right-half space) for any $\delta\gtrsim \alpha$, we have strong convergence in $L^2_{\mathbf{x}}$ on the conic right-half space (see Figure \ref{F:cut-cone})
\begin{equation}
\label{E:window-conv-b}
\| c(t)^2 u(c(t) \mathbf{x} + \mathbf{a}(t),t) - Q(\mathbf{x}) \|_{L^2_{\mathbf{x}}(x> (-1+\delta) t -\sqrt{y^2+z^2}\tan \theta )} \to 0 \text{ as } t\to +\infty
\end{equation}
for all $\theta$ such that 
$$
0\leq  \theta \leq \frac{\pi}{3}-\delta.
$$
\end{enumerate}
\end{theorem}
The $L^2$ convergence is stated in \eqref{E:window-conv-b} in the reference frame of the soliton (being at the origin). In the reference frame of the solution, the rightward shifting external conic region is  $x> \delta t -\sqrt{y^2+z^2} \, \tan \theta$. 

As mentioned, this theorem follows from the orbital stability result of de Bouard \cite{deB}, the recent well-posedness result of Herr \& Kinoshita \cite{HK}, and our key theorem 
(Theorem \ref{T:main} below).    We note that any $u_0(\mathbf{x})$ for which there exists $c_0>0$ and $\mathbf{a}_0 \in \mathbb{R}^3$ so that
$$\| c_0^2 u_0(c_0 \mathbf{x} + \mathbf{a}_0) - Q(\mathbf{x}) \|_{H_{\mathbf{x}}^1} \lesssim \alpha$$
can be rescaled and translated to meet the hypothesis \eqref{E:initial}.  

In \S\ref{S:outline} below, we provide an outline of the paper with definitions, the statement of the main theorem (Theorem \ref{T:main}), supporting propositions and lemmas.  These supporting propositions and lemmas are each proved in the sections of the paper (\S \ref{S:RV-estimates}-\ref{S:numerics}) indicated after their statement.  The broad outline of the argument is as follows: monotonicity estimates based on calculating 
$$
\partial_t \int [u(\mathbf{x}+\mathbf{a}(t), t)]^2 \phi(\mathbf{x} ,t) 
\, d\mathbf{x}
$$
for a suitable monotonic-in-$x$ weight $\phi(\mathbf{x},t)$, 
provide the strong $L^2$ convergence in \eqref{E:window-conv-b} away from the soliton center.  Once the weak convergence in \eqref{E:main-weak} is established, the strong $L^2$ convergence on a compact region around the soliton center in \eqref{E:window-conv-b} will follow.  Thus, the main remaining task is to establish \eqref{E:main-weak}, which is proved in several steps.  Taking a limit of solutions along a time sequence $t_n\nearrow +\infty$ yields a radiation-free solution $\tilde u(\mathbf{x},t)$.  The monotonicity estimates give exponential spatial decay of this solution, but the functional analytic methods that produce this limiting solution $\tilde u$ only  yield that it has $H_{\mathbf{x}}^1$ regularity and is a weak type solution (that we call a Class B solution).  One key new element of the paper is showing that the uniform in time strong spatial decay of $\tilde u$ is enough to boost its regularity -- we are, in fact, able to show it is \emph{smooth}, and thus, a strong solution to the 3D ZK. Then we show that $\tilde u$, once renormalized, is $Q$, the soliton, by a rigidity argument based on a virial estimate for the linearized equation.  This is achieved by contradiction -- if the rigidity statement failed, then there would be a sequence of solutions $\tilde u$, from which we could extract (after renormalization) a solution to the linearized equation \emph{without nonlinear terms} (we call this the linear linearized equation).  In the passage of this limit, we again use our regularity boost techniques.  Finally, we can prove a virial estimate for the linear linearized equation by a positive commutator argument after passing to a dual problem and checking a spectral condition with robust numerical analysis.  The regularity boost arguments mentioned above are new to this type of problem, and involve Littlewood-Paley analysis, a discrete Gronwall argument, and the local theory estimates of Ribaud \& Vento, even though these estimates lie at regularity slightly above $H_{\mathbf{x}}^1$.  

\subsection{Acknowledgements}
L.G.F. was partially supported by CNPq, CAPES and  FAPEMIG (Brazil). From September 2017 till August 2019, J.H. served as Program Director in the Division of Mathematical Sciences at the National Science Foundation (NSF), USA, and as a component of this job, J.H. received support from NSF for research, which included work on this paper.  Any opinion, findings, and conclusions or recommendations expressed in this material are those of the authors and do not necessarily reflect the views of the National Science Foundation.  S.R. was partially supported by the NSF CAREER grant DMS-1151618/1929029 and NSF grant DMS-1815873/1927258.  K.Y. research support to work on this project came from grants DMS-1929029 and 1927258 (PI: Roudenko).

\section{Outline of the paper}
\label{S:outline}

\begin{definition}[Class B solutions]
\label{D:ClassB}
We call $u(\mathbf{x},t)$ a \emph{Class B} global solution of the 3D ZK if 
\begin{enumerate}
\item 
for each $T>0$ and for each $s<1$, 
$$
u\in C([-T,T]; H_{\mathbf{x}}^s) \,, \qquad \partial_t u \in  C([-T,T]; H_{\mathbf{x}}^{s-3}),
$$
\item  
for each $t\in \mathbb{R}$, $u(t) \in H_{\mathbf{x}}^1$ and $\partial_t u(t) \in H_{\mathbf{x}}^{-2}$ and there exists $C>0$ such
$$
\sup_{t\in \mathbb{R}} \| u(t) \|_{H_{\mathbf{x}}^1} + \sup_{t\in\mathbb{R}} \| \partial_t u(t) \|_{H_{\mathbf{x}}^{-2}} \leq C,
$$
\item 
for each $t\in \mathbb{R}$, the equation
$$
\partial_t u(t) + \partial_x \Delta u(t) + \partial_x u(t)^2=0
$$
holds as an equality of the sum of three functions each belonging to $H^{-2}_{\mathbf{x}}$.
\end{enumerate}
\end{definition}

\begin{lemma}[Class B solutions satisfy mass conservation]
\label{L:mass-conservation}
Suppose that $u$ is a Class B solution to the 3D ZK.  Then $u$ satisfies mass conservation, i.e., $\| u(t) \|_{L_{\mathbf{x}}^2}^2$ is constant in time, and is denoted $M(u)$.  
\end{lemma}

This is proved in \S\ref{S:ClassB-mass} by computing $\partial_t \| P_{\leq N}u\|_{L_{\mathbf{x}}^2}^2$, deducing a near conservation law with error bounded by $N^{-1/2}$, and then sending $N\to \infty$.  We note that a similar method does not work to prove energy conservation.

\begin{definition}[orbital stability]
\label{D:orb-stab}
Let $\alpha>0$.  We say that $u$ is an \emph{$\alpha$-orbitally stable} solution to the 3D ZK if $u$ is a Class B solution such that 
$$
\sup_{t\in \mathbb{R}}  \inf_{\substack{\mathbf{a}(t) \in \mathbb{R}^3 \\ c(t) \in (0,+\infty)}} \| c(t)^2 u(c(t)\mathbf{x}+\mathbf{a}(t), t) -Q( \mathbf{x}) \|_{H_{\mathbf{x}}^1} \leq \alpha.
$$
\end{definition}

\begin{lemma}[unique parameters]
\label{L:geom-decomp}
There exists $\alpha>0$ sufficiently small so that, if $u$ is a Class B $\alpha$-orbitally stable solution to the 3D ZK, then there exist \emph{unique} translation $\mathbf{a}(t)$ and scale parameters $c(t)>0$ so that $\epsilon$ defined by
$$\epsilon(\mathbf{x},t) = c(t)^2u(c(t)\mathbf{x}+\mathbf{a}(t),t) - Q(\mathbf{x})$$
satisfies, for all $t$, the orthogonality conditions
$$
\la \epsilon(t), \nabla Q \ra =0 \quad \mbox{and} \quad \la \epsilon(t), Q^2 \ra =0,
$$
and
$$
\| \epsilon \|_{L_t^\infty H_{\mathbf{x}}^1} \lesssim \alpha.
$$
Let $\mathcal L = I-\Delta -2Q$ and define
$$
f\defeq \frac{\mathcal{L}\partial_x (Q^2)}{\la \Lambda Q, Q\ra} \,, \qquad \mathbf{g} \defeq \left(  \frac{ \mathcal{L}Q_{xx} }{ \|Q_x\|_{L_{\mathbf{x}}^2}}, \frac{\mathcal{L}Q_{xy}}{ \|Q_y\|_{L_{\mathbf{x}}^2}}, \frac{\mathcal{L}Q_{xz} }{ \|Q_z\|_{L_{\mathbf{x}}^2}}\right).
$$
Denote $b(t) = \|\epsilon(t)\|_{L_{\mathbf{x}}^2}$.  
Then the parameters  $c(t)$ and $\mathbf{a}(t)$ are $C^{1,\frac23}$ and satisfy
$$
|\,c^2c' - \la \epsilon, f\ra | + |\, c(\mathbf{a}' - c^{-2} \mathbf{i}) - \la \epsilon, \mathbf{g}\ra | \lesssim b(t)^2.
$$
\end{lemma}

This is proved in \S\ref{S:geom-decomp} by an implicit function theorem argument.  The equations for the parameters follow by differentiating the orthogonality conditions in time.
We mention that parameter estimates can be found in \cite{FHRY, FHR1,FHR2,FHR3, FHR4}.

Our main theorem is the following

\begin{theorem}[main theorem for Class B]
\label{T:main}
For $\alpha \ll 1$, any $\alpha$-orbitally stable Class B solution $u$ to the 3D ZK with $M(u)=M(Q)$ is \emph{asymptotically stable} in the following sense:  there exists $c_*$ such that $|c_*-1| \lesssim \alpha$ such that as $t\to +\infty$, 
$$
c(t) \to c_* \,,  \qquad \mathbf{a}'(t) \to c_*^{-2} \mathbf{i}
$$
and 
$$
c(t)^2 u(c(t) \mathbf{x} + \mathbf{a}(t),t) \rightharpoonup Q(\mathbf{x}) \quad \text{ (weakly) in } H_{\mathbf{x}}^1.
$$  
Moreover, for any $\delta\gtrsim \alpha$, we have strong convergence in $L^2_{\mathbf{x}}$ on the conic right-half space
\begin{equation}
\label{E:window-conv}
\| c(t)^2 u(c(t) \mathbf{x} + \mathbf{a}(t),t) - Q(\mathbf{x}) \|_{L^2_{\mathbf{x}}(x> (-1+\delta) t -\sqrt{y^2+z^2}\tan \theta )} \to 0
\end{equation}
for all $\theta$ such that 
$$0\leq  \theta \leq \frac{\pi}{3}-\delta.$$
\end{theorem}

The proof of Theorem \ref{T:main} follows from Propositions \ref{P:wk-lim} and \ref{P:rigidity} below, as detailed in \S \ref{S:proof-main-theorem}.  It is deduced from these main results plus the monotonicity estimate in \S\ref{S:monotonicity}, in particular, Lemma \ref{L:Ipm-estimates}, which gives an estimate on the mass of the solution in a conic right-half space region $(\cos \theta, \sin\theta) \cdot (x+(1-\delta) t, \sqrt{1+y^2+z^2})>0$, in the reference frame, where the soliton is at the origin.  Specifically, it estimates this cut-off mass in the future in terms of its value in the past.  In \S \ref{S:proof-main-theorem}, this is applied to give a ``decay on the right'' estimate in the conic region depicted in Figure \ref{F:decay}.  But it can be applied for two different slopes (for example $x>-\frac{1}{10}t$  and $x>-\frac{19}{20}t$) to show that both regions asymptotically trap the same mass, and thus, the region \emph{between} these lines has asymptotically vanishing mass.  This results in a ``decay on the left'' estimate also depicted in Figure \ref{F:decay}.  By the decay on the right and decay on the left estimates, it suffices to prove that the solution in the soliton region $|\mathbf{x}| \lesssim r$ converges weakly to a rescaling of $Q(\mathbf{x})$.  This is accomplished in Propositions \ref{P:wk-lim} and \ref{P:rigidity}.

\begin{proposition}[construction of a smooth spatially decaying asymptotic solution]
\label{P:wk-lim}
There exists $\alpha_0>0$ such that for all $0< \alpha \leq \alpha_0$, the following holds.  Let $u$ be an $\alpha$-orbitally stable Class B solution to the 3D ZK with $M(u)=M(Q)$, and let $c(t)>0$ and $\mathbf{a}(t) \in \mathbb{R}^3$ be the associated modulation parameters of scale and position given by Lemma \ref{L:geom-decomp}.  For each sequence of times $t_m \nearrow +\infty$, there exists a subsequence $t_{m'} \nearrow +\infty$ such that for each $t\in \mathbb{R}$, 
$$
u(\mathbf{x} + \mathbf{a}(t_{m'}), t+ t_{m'}) \rightharpoonup \tilde u(\mathbf{x}, t) \quad \text{(weakly) in }H_{\mathbf{x}}^1,
$$
where $\tilde u$ is a \emph{smooth} $\alpha$-orbitally stable solution to the 3D ZK.  Moreover, letting  $\tilde c(t)>0$ and $\tilde{\mathbf{a}}(t)\in \mathbb{R}^3$ be the modulation parameters associated to $\tilde u$ given by Lemma \ref{L:geom-decomp}, we have the uniform-in-time spatial decay property:  for each  $r>0$, 
$$
\| \tilde u( \mathbf{x} + \tilde{\mathbf{a}}(t), t) \|_{L_{t\in \mathbb{R}}^\infty L^2_{\mathbf{x}}(|\mathbf{x}|>R)} \lesssim e^{- \frac{R}{32}}.
$$
\end{proposition}

\begin{proposition}[rigidity of orbitally stable smooth solutions with spatial decay]
\label{P:rigidity}
There exists $\alpha_0>0$ such that for all $0< \alpha \leq \alpha_0$, the following holds.  Let $\tilde u$ be a smooth $\alpha$-orbitally stable solution to the 3D ZK with associated modulation parameters $\tilde c(t)>0$ and $\tilde{\mathbf{a}}(t)\in \mathbb{R}^3$ given by Lemma \ref{L:geom-decomp}.  Suppose that $\tilde u$ satisfies the uniform-in-time spatial decay property:  for each $k\geq 0$, 
\begin{equation}
\label{E:intro-101}
\| \la  \mathbf{x}  \ra^k \tilde u( \mathbf{x} + \tilde{\mathbf{a}}(t), t) \|_{L_{t\in \mathbb{R}}^\infty L^2_{\mathbf{x}}} < \infty.
\end{equation}
Then there exists $c_+>0$ and $\mathbf{a}_+\in \mathbb{R}^3$ such that
$$
\tilde u( \mathbf{x}, t) = c_+^{-2} Q\left( c_+^{-1}(\mathbf{x} - \mathbf{a}_+- t\,c_+^{-2})\right).
$$
\end{proposition}

\subsection{Outline of proof of Proposition \ref{P:wk-lim}}

The proof of Proposition \ref{P:wk-lim} is decomposed into three key lemmas, as follows.

\begin{lemma}
\label{L:soft-step}
There exists $\alpha_0>0$ sufficiently small so that for all $0<\alpha \leq \alpha_0$, the following holds.  Suppose that $u$ is a Class B solution to the 3D ZK and is $\alpha$-orbitally stable.  Let $t_m \nearrow +\infty$ be an arbitrary sequence of times.  Then there exists a subsequence $t_{m'}$ such that the following holds
\begin{enumerate}
\item 
For each\footnote{meaning that the weak limit exists and we \emph{define} $\tilde u(t)$ to be the value of the limit.} $t\in \mathbb{R}$, $u(\bullet+\mathbf{a}(t_{m'}), t+t_{m'}) \rightharpoonup \tilde u(t)$ weakly in $H_{\mathbf{x}}^1$.
\item For each $R>0$ and each finite time interval $I$, $u(\mathbf{x}+\mathbf{a}(t_{m'}), t+t_{m'}) \mathbf{1}_{<R}(\mathbf{x})$ converges strongly in $C(I; L_{\mathbf{x}}^2)$ to $\tilde u(\mathbf{x},t) \mathbf{1}_{<R}(\mathbf{x})$.
\item $\tilde u$ is a Class B solution to the 3D ZK. 
\item $\tilde u$ is $\alpha$-orbitally stable with associated parameters (as in Lemma \ref{L:geom-decomp}) $\tilde{\mathbf{a}}(t)$ and $\tilde c(t)$.  In fact, for every $t\in \mathbb{R}$, we have 
\begin{equation}
\label{E:param-conv}
\begin{aligned}
&\mathbf{a}(t+t_{m'}) - \mathbf{a}(t_{m'}) \to \tilde{\mathbf{a}}(t)\quad \mbox{and} \quad c(t+t_{m'}) \to \tilde c(t)\quad \text{ as }m'\to \infty. 
\end{aligned}
\end{equation}
In particular, $\tilde{\mathbf{a}}(0)=0$.
\end{enumerate}
\end{lemma}

This is proved in \S \ref{S:wk-lim}.  Lemma \ref{L:soft-step} provides the $\alpha$-orbitally stable limiting solution $\tilde u$, but only as a Class B solution, and it is constructed by weak-* compactness methods.    Using that $\mathbb{Q}$ is countable, a subsequence $t_{m'}$ is obtained along which $u(\bullet+\mathbf{a}(t_{m'}), t+t_{m'})$ converges weakly in $H_{\mathbf{x}}^1$ for each $t\in \mathbb{Q}$.    Using a frequency projected uniform continuity in time property of $u$ and density of $\mathbb{Q}$ in $\mathbb{R}$, this weak convergence is extended to hold for all $t\in \mathbb{R}$ (not just $t\in \mathbb{Q}$).  Defining $\tilde u(t)$ to be this weak limit, the fact that it is an $\alpha$-orbitally stable Class B solution to the 3D ZK is inherited from the corresponding properties of $u$ via elementary arguments. 

The limiting solution $\tilde u$ provided in Lemma \ref{L:soft-step} is obtained merely as a Class B solution -- this is all that is possible using weak-* compactness machinery.  The fact that $\tilde u$ is exponentially decaying and smooth is separately obtained in Lemma \ref{L:exp-decay} and Lemma \ref{L:regularity-boost} below, using monotonicity lemmas and a virial-type regularity gain estimate, respectively

\begin{lemma}
\label{L:exp-decay}
The Class B solution $\tilde u$ constructed in Lemma \ref{L:soft-step} satisfies exponential decay in space, uniformly-in-time.  Specifically, 
$$
\| \tilde u(\mathbf{x}+\tilde{\mathbf{a}}(t),t) \|_{L_t^\infty L^2_{\mathbf{x}}(|\mathbf{x}|>R)} \lesssim e^{-R/32}.
$$
\end{lemma}

This is proved in \S \ref{S:app-mon}, by applying the monotonicity estimates \eqref{E:pf-main-1} and \eqref{E:pf-main-5} in Lemma \ref{L:u-decay}, which were obtained from the $I_+$  monotonicity estimate \eqref{E:Ip-right} in Lemma \ref{L:Ipm-estimates} (in \S\ref{S:monotonicity}).

\begin{lemma}
\label{L:regularity-boost}
Any Class B solution $\tilde u$ of the 3D ZK satisfying the exponential decay as in Lemma \ref{L:exp-decay} is in fact smooth.
\end{lemma}

This is proved in \S \ref{S:higher-regularity}.  The proof hinges on a frequency projected virial-type identity \eqref{E:HR10} for Class B solutions.  When it is integrated in time and terms are estimated using weighted Sobolev estimates and Bernstein's inequality, we obtain in Lemma \ref{L:L2boost} a bound on $\|u\|_{L_I^2 H_{\mathbf{x}}^{5/4-}}$ in terms of weighted $L^2_{\mathbf{x}}$ bounds and (unweighted) energy bounds $H_{\mathbf{x}}^1$.  Note that $\|u\|_{L_I^2 H_{\mathbf{x}}^{5/4-}}$ reflects a gain in regularity, but \emph{averaged in time}.    At this point, we are able to tap into the feature of the Ribaud \& Vento local well-posedness machinery (as outlined in \S\ref{S:RV-estimates}) that the right-side bounds in their argument are slightly above $H_{\mathbf{x}}^1$ but have time integration ``to spare''.  We can then use discrete Grownwall type estimates in the frequency decomposition in Lemmas \ref{L:maximal-compare} and \ref{L:reg-boost-last} to bootstrap the regularity gain to $L_I^\infty H_{\mathbf{x}}^{9/8}$, an honest improvement in regularity (it is $L^\infty$ in time).  This argument can be, in fact, be applied recursively to achieve any level of regularity.   We note that it is possible to gain regularity in this way because the solution is assumed to have exponential spatial decay.  

It is apparent that the conclusions of Lemmas \ref{L:soft-step},  \ref{L:exp-decay}, and \ref{L:regularity-boost} combined yield the conclusions of Proposition \ref{P:wk-lim}.

\subsection{Outline of proof of Proposition \ref{P:rigidity}}

The proof of Proposition \ref{P:rigidity} proceeds by contradiction.  Suppose that the conclusion of Proposition \ref{P:rigidity} is false.  Then there exists a sequence $\tilde u_n$ of smooth $\alpha_n$-orbitally stable solutions to the 3D ZK, $|\alpha_n| \to 0$ such that the following holds.  Let $\tilde c_n(t)>0$ and $\tilde{\mathbf{a}}_n(t) \in \mathbb{R}^3$ be the modulation parameters associated to $\tilde u_n$ given by Lemma \ref{L:geom-decomp}, and let
\begin{equation}
\label{E:intro-100}
\tilde \epsilon_n(t) \defeq \tilde c_n(t)^2 \tilde u_n (\tilde c_n(t)\mathbf{x} + \tilde{\mathbf{a}}_n(t), t) - Q(\mathbf{x}).
\end{equation}
Then for each $n$, for some $t$, 
$$
b_n(t) \defeq \| \tilde \epsilon_n(t) \|_{L_{\mathbf{x}}^2}>0.
$$
It follows that for \emph{all} $t\in \mathbb{R}$, $b_n(t)>0$.  (Indeed, if $b_n(t)=0$ for some $t$, then $b_n(t)=0$ for all $t\in \mathbb{R}$).  We can assume, without loss of generality by replacing $\tilde u_n(t)$ by $\tilde u_n(t+t_{*n})$ for some $t_{*n}$ that
$$ 
b_n(0) \geq \frac12 \sup_{t\in \mathbb{R}} b_n(t) \defeq B_n>0.
$$
Moreover, by a shift and slight rescaling of $\tilde u_n$, for each $n$, we can assume that
$$
\tilde c_n(0) =1 \quad \text{ and } \quad \tilde{\mathbf{a}}_n(0)=0.
$$
Let
\begin{equation}
\label{E:intro-102}
w_n(t) = \frac{\tilde \epsilon_n(t)}{B_n}
\end{equation}
so that for all $n$, 
$$
\|w_n(0) \|_{L_{\mathbf{x}}^2} \geq \frac12 \,, \qquad \|w_n \|_{L_t^\infty L_{\mathbf{x}}^2} \leq 1.
$$
We will obtain a contradiction from the following five lemmas, which, in particular, imply that $w_n(0) \to 0$ strongly in $L_{\mathbf{x}}^2$.

Although we know from \eqref{E:intro-101} that each $\tilde u_n$, and hence, each $\tilde \epsilon_n$, satisfies uniform-in-time spatial decay, we do not know \emph{a priori} that this decay is \emph{uniform in $n$}, and moreover, \emph{normalized} according the mass of $\tilde \epsilon_n$.   Nevertheless, these properties can be proved using the $J_\pm$ monotonicity estimates in \S\ref{S:monotonicity}. The result is

\begin{lemma}[uniform spatial decay]
\label{L:ep-decay}
Let $\tilde \epsilon_n$ be as defined in \eqref{E:intro-100}.  Then $\tilde \epsilon_n$ satisfies uniform-in-$n$, uniform-in-time, exponential spatial decay:
$$
\| \tilde \epsilon_n \|_{L_t^\infty L_\mathbf{x}^2({|\mathbf{x}|>R})} \lesssim e^{-R/32} \| \tilde \epsilon_n \|_{L_t^\infty L_{\mathbf{x}}^2}.
$$
Consequently, $w_n$ defined by \eqref{E:intro-102} satisfies
$$
\| w_n \|_{L_t^\infty L_\mathbf{x}^2({|\mathbf{x}|>R})} \lesssim e^{-R/32} 
$$
uniformly in $n$.
\end{lemma}

This is proved in \S \ref{S:uniform-n-decay}.  As mentioned, it is rather quickly deduced as a consequence of the $J_\pm$ monotonicity in Lemma \ref{L:Jpm-estimates}.

\begin{lemma}[comparability of Sobolev norms]
\label{L:ep-comparability}
Let $\tilde \epsilon_n$ be as defined in \eqref{E:intro-100}.  Then $\tilde \epsilon_n$ satisfies, for all $k$,
\begin{equation}
\label{E:intro-tilde-comp}
\| \tilde \epsilon_n \|_{L_t^\infty H_{\mathbf{x}}^k} \lesssim_k \| \tilde \epsilon_n \|_{L_t^\infty L_{\mathbf{x}}^2}
\end{equation}
uniformly in $n$.  Consequently, $w_n$ defined by \eqref{E:intro-102} satisfies, for each $k\geq 0$,
$$
\| w_n \|_{L_t^\infty H_{\mathbf{x}}^k } \lesssim_k 1
$$
uniformly in $n$.
\end{lemma}

This is proved in \S \ref{S:Sobolev-comparability}.  The proof is similar to the proof of Lemma \ref{L:regularity-boost} given in \S\ref{S:higher-regularity}, although additional ingredients are introduced to handle the $H_{\mathbf{x}}^1$ bound ($k=1$ case of Lemma \ref{L:ep-comparability}), which was automatic in the context of Lemma \ref{L:regularity-boost}.   At issue here is the need to obtain the small factor $\| \tilde \epsilon_n \|_{L_t^\infty L_{\mathbf{x}}^2}$ on the right side of \eqref{E:intro-tilde-comp}. 
The idea is to couple a virial-type identity without frequency localization to one with frequency localization.  The one without frequency location allows for a reduction of order of derivatives via integration by parts  in the nonlinear term, which gives a bound that can be used in the nonlinear term estimates for the virial-type identity \emph{with} frequency localization. 

\begin{lemma}[convergence]
\label{L:convergence}
For each $T>0$, $w_n \to w$ in $C([-T,T]; L_{\mathbf{x}}^2)$ satisfying the following:
\begin{enumerate}
\item 
$w$ is uniform-in-time smooth:  for each $k\geq 0$
$$
\| w\|_{L_t^\infty H_{\mathbf{x}}^k} < \infty,
$$
\item 
$w$ has uniform-in-time spatial decay:  
$$
\|  w \|_{L_t^\infty L_{\mathbf{x}}^2(|\mathbf{x}|>R)} \lesssim e^{-\delta R},
$$
\item 
$w(0)$ is nontrivial:
$$
\| w(0) \|_{L_{\mathbf{x}}^2} = 1,
$$
\item 
$w$ satisfies the equation
$$
\partial_t w = \partial_x \mathcal{L}w + \alpha \Lambda Q + \boldsymbol{\beta} \cdot \nabla Q,
$$
where $\alpha$ and $\boldsymbol{\beta} = (\beta_1,\beta_2,\beta_3)$ are time-dependent coefficients,
\item 
$w$ satisfies the orthogonality conditions
$$
\la w, \nabla Q \ra =0 \quad \text{and} \quad \la w, Q^2 \ra =0.
$$
\end{enumerate} 
\end{lemma}

This is proved in \S \ref{S:convergence}.  Working with $\tilde \zeta_n$, a recentered and renormalized version of $\tilde \epsilon_n$ (see \eqref{E:tilzeta}), first pass to a subsequence via Rellich-Kondrachov compactness so that $\tilde \zeta_n(0) \to \zeta_\infty(0)$, which is smooth and exponentially decaying.  Taking $\zeta_\infty(t)$ to solve the expected limiting equation \eqref{E:con1}, we aim to prove that $\tilde \zeta_n(t) \to \zeta_\infty(t)$ for all $t\in \mathbb{R}$.    Letting $\hat \zeta_n= \tilde \zeta_n - \zeta_\infty$, we derive the evolution equation for the difference, from which we deduce a Gronwall estimate on $\hat \zeta_n$, which shows the convergence in terms of $\tilde b_n \to 0$.  In the original frame of reference, the limit is $w$, as described in the statement of Lemma \ref{L:convergence}.  All the properties of $w$ stated in Lemma \ref{L:convergence} are inherited from the sequence $w_n =\tilde \epsilon_n/B_n$.  

Now that we have constructed a nontrivial limiting solution $w$ with the properties stated in Lemma \ref{L:convergence}, the next step in the argument by contradiction is to prove that it cannot exist.  This is achieved in the following lemma.

\begin{lemma}[linear Liouville property]
\label{L:linear-Liouville}
Suppose that $w$ solves
\begin{equation}
\label{E:w-eq}
\partial_t w =\partial_x \mathcal{L}w + \alpha \Lambda Q + \boldsymbol{\beta} \cdot \nabla Q,
\end{equation}
where $\alpha$ and $\boldsymbol{\beta}$ are time-dependent, and further suppose that $w$ satisfies the orthogonality conditions
\begin{equation}
\label{E:extra-orth}
\la w, Q^2 \ra =0 \quad \text{and} \quad \la w, \nabla Q \ra =0.
\end{equation}
If $w$ satisfies global uniform-in-time spatial decay
\begin{equation}
\label{E:w-dec}
\| \la x \ra^{1/2} w \|_{L_t^\infty H_{\mathbf{x}}^2} < \infty,
\end{equation}
then
$w\equiv 0$.
\end{lemma}

This is proved in \S \ref{S:linear-Liouville} by observing that the quadratic in $w$ quantity 
$$
Q(w) \defeq \la \mathcal{L}w,w \ra + \frac{2}{\la \Lambda Q, Q \ra} \la w, Q \ra^2
$$
is constant in time.  This follows by computing $\partial_t Q(w)$, plugging in the equation for $w$, and appealing to the orthogonality conditions \eqref{E:extra-orth}.  However, the time integral $\int_{t=-\infty}^{\infty} Q(w) \, dt$ is in fact controlled by the left side of \eqref{E:vir-bd}, but the right side of  \eqref{E:vir-bd} is finite by the assumption \eqref{E:w-dec}.  This forces $Q(w)\equiv 0$, and by the positive definiteness of $\mathcal{L}$ (subject to \eqref{E:extra-orth}), this forces $w\equiv 0$.

\begin{lemma}[virial estimate]
\label{L:virial}
Suppose that $w$ solves
$$
\partial_t w =\partial_x \mathcal{L}w + \alpha \Lambda Q + \boldsymbol{\beta} \cdot \nabla Q,
$$
where $\alpha$ and $\boldsymbol{\beta}$ are time-dependent, and further suppose that $w$ satisfies the orthogonality conditions
$$
\la w, Q^2 \ra =0 \quad \text{and} \quad \la w, \nabla Q \ra =0.
$$
Then $w$ satisfies the global-in-time estimate
\begin{equation}
\label{E:vir-bd}
\| w \|_{L_t^2 H_{\mathbf{x}}^3} \lesssim \| \la x \ra^{1/2} w \|_{L_t^\infty H_{\mathbf{x}}^2}.
\end{equation}
\end{lemma}

This is proved in \S \ref{S:virial}.  This inequality is proved via passage to a dual problem in $v=\mathcal{L}w$ and the proof that $v$ satisfies a virial identity.  The desired inequality reduces to the positivity of a certain quadratic form.  The positivity of this quadratic form is checked numerically, and details of the numerical method are provided in \S \ref{S:numerics}.

\subsection{Notational conventions}
We will use $\mathbf{x} = (x,y,z)$ for the spatial variable and $\boldsymbol{\xi}$ for the Fourier variable in $\mathbb{R}^3$.  The Littlewood-Paley frequency projection is $\widehat{P_N f}(\boldsymbol{\xi}) = m(\boldsymbol{\xi}/N) \hat f(\xi)$, where $m(\boldsymbol{\xi})$ is smooth, supported in $\frac12\leq |\boldsymbol{\xi}| \leq 2$, and satisfies $\sum_{N \in 2^{\mathbb{Z}}} m(\boldsymbol{\xi}/N) = 1$.  However, we will use only $N \geq 1$, and in fact replace $P_1 = \sum_{N\leq 1} P_N$ to cover all low frequencies.  We will use the notation $P_{\leq M} = \sum_{N\leq M} P_N$.   While weighted estimates use weight $x$ (not $\mathbf{x}$), all frequency projections are done with respect to all three variables using $P_N$ as defined above in terms of $m(\boldsymbol{\xi})$.  In some arguments in \S\ref{S:RV-estimates}, \S\ref{S:higher-regularity} and \S\ref{S:Sobolev-comparability}, we use the shorthand $\ln^+N = \ln(N+2)$ so that for all $N \geq 1$, we have $\ln^+N \geq 1$ (avoiding $\ln 1=0$).

Throughout the paper we refer to Class B solutions, which were defined in Definition \ref{D:ClassB}.  
For an $\alpha$-orbitally stable solution $u$ to the 3D ZK (as defined in Definition \ref{D:orb-stab}) and modulation parameters $\mathbf{a}(t)$ and $c(t)$ (as given in Lemma \ref{L:geom-decomp}), we use the following notations for the \emph{remainder}:
$$
\epsilon(\mathbf{x},t) \defeq c(t)^2u(c(t)\mathbf{x}+\mathbf{a}(t),t) - Q(\mathbf{x}).
$$
With $Q_{c,\mathbf{a}}(\mathbf{x}) = c^{-2} Q(c^{-1}(\mathbf{x}-\mathbf{a}))$, we define
$$
\eta(\mathbf{x},t) \defeq c^{-2}\epsilon(c^{-1}(\mathbf{x}-\mathbf{a})) = u(\mathbf{x},t) - Q_{c, \mathbf{a}}(\mathbf{x})
$$
(see \eqref{E:eta-def} and \eqref{E:eta-eq}), and 
$$
\zeta(\mathbf{x},t) \defeq B^{-1} \eta(t)
$$
for $B = \|b(t) \|_{L_t^\infty}$, where $b(t) = \|\eta(t) \|_{L_{\mathbf{x}}^2}$  (see \eqref{E:zeta-1}).

Integrals related to the monotonicity property of solutions to the 3D ZK are denoted by $I_\pm$ and $J_\pm$ and defined in \eqref{E:Ipm-de} and \eqref{E:J-def}, respectively.

\section{Review of local theory estimates}
\label{S:RV-estimates}

In this section we review Ribaud \& Vento \cite{RV} local estimates as they become an essential tool later in our arguments. We start with the following result.

\begin{lemma}[Ribaud \& Vento, Lemma 3.3]
\label{L:RV1}
For $M \geq 1$, and $I$, a time interval of length $|I|\leq 1$, we have 
\begin{equation}
\label{E:HR20}
\| P_M U(t) \phi  \|_{L_x^2 L_{yz I}^\infty} \lesssim (\ln^+ M)^2 M \| P_M \phi \|_{L_{\mathbf{x}}^2},
\end{equation}
\begin{equation}
\label{E:HR21}
\| P_M \int_0^t\partial_x U(t-s)f(\bullet,s) \,ds  \|_{L_x^2 L_{yz I}^\infty} \lesssim (\ln^+ M)^2 M \| P_M f \|_{L_x^1 L_{yzI}^2},
\end{equation}
\begin{equation}
\label{E:HR21b}
\| P_M \int_0^t\partial_x U(t-s)f(\bullet,s) \,ds  \|_{L_I^\infty L_{\mathbf{x}}^2} \lesssim  \| P_M f \|_{L_x^1 L_{yzI}^2}.
\end{equation}
\end{lemma}
\begin{proof}
In all of the estimates, the time variables are restricted to the unit-sized interval $I$. The boundedness of the following are equivalent
\begin{itemize}
\item 
$P_M \Lambda:  L_x^2L_{yzI}^1 \to  L_{\mathbf{x}}^2 $, with operator norm $(\ln^+ M)^2M$,
\item 
$P_M \Lambda^*:  L_{\mathbf{x}}^2 \to L_x^2 L_{yzI}^\infty $,  with operator norm $(\ln^+ M)^2M$,
\item 
$P_M^2 \Lambda^* \Lambda: L_x^2L_{yzI}^1 \to L_x^2 L_{yzI}^\infty$,  with operator norm $(\ln^+ M)^4M^2$,
\end{itemize}
where
$$\Lambda f(\mathbf{x}) = \int_{s=0}^1 U(-s) f(\mathbf{x}, s) \, ds,$$
$$\Lambda^* \phi(\mathbf{x},t) = U(t) \phi(\mathbf{x}),$$
$$\Lambda^*\Lambda f(\mathbf{x},t) = \int_{s=0}^1 U(t-s) f(\mathbf{x},s) \, ds.$$
The kernel of the operator $P_M^2 \Lambda^*\Lambda$ is
$$
K_M(\mathbf{x},t) = \int_{|\boldsymbol{\xi}|\sim M} e^{i(\mathbf{x}\cdot \boldsymbol{\xi} + t\xi|\boldsymbol{\xi}|^2)} \, d\boldsymbol{\xi}. 
$$
To establish that $P_M^2 \Lambda^*\Lambda: L_x^2L_{yzI}^1 \to L_x^2 L_{yzI}^\infty$ is bounded with operator norm $\lesssim (\ln^+ M)^4M^2$, it suffices to show that
$$
\|K_M \|_{L_x^1L_{yzI}^\infty} \lesssim (\ln^+ M)^4M^2.
$$
This was proved in Ribaud \& Vento, Lemma 3.3.  Since this establishes that  $P_M^2 \Lambda^*\Lambda: L_x^2L_{yzI}^1 \to L_x^2 L_{yzI}^\infty$ is bounded with operator norm $\lesssim (\ln^+ M)^4M^2$, we have equivalent fact that $P_M \Lambda_*: L_{\mathbf{x}}^2 \to L_x^2 L_{yz I}^\infty$ is  bounded with operator norm $\lesssim (\ln^+ M)^2M$, which is precisely \eqref{E:HR20}.

The local smoothing estimate from Ribaud \& Vento (and other references) asserts  the boundedness of 
$$
\partial_x \Lambda^*:  L_{\mathbf{x}}^2 \to L_x^\infty L_{yz I}^2.
$$ 
Hence, we have also the boundedness of
\begin{equation}
\label{E:RV1}
\partial_x \Lambda:  L_x^1 L_{yz I}^2 \to  L_{\mathbf{x}}^2.
\end{equation}
This, combined with the fact that $P_M \Lambda^*:  L_{\mathbf{x}}^2 \to L_x^2 L_{yzI}^\infty $  is bounded with operator norm $(\ln^+ M)^2M$, yields the boundedness of 
$$
P_M \partial_x \Lambda^* \Lambda:  L_x^1 L_{yz I}^2 \to   L_x^2 L_{yzI}^\infty
$$
with operator norm $(\ln^+ M)^2M$.  Combining with the Christ-Kiselev lemma gives \eqref{E:HR21}. 

The standard unitarity property for $U(t)$ implies the boundedness of the map $\Lambda^*: L_{\mathbf{x}}^2 \to L_I^\infty L_{\mathbf{x}}^2$, which together with \eqref{E:RV1} yields the boundedness of
$$
\partial_x \Lambda^* \Lambda : L_x^1 L_{yz I}^2 \to L_I^\infty L_{\mathbf{x}}^2.
$$
Again, combined with the Christ-Kiselev lemma, it gives \eqref{E:HR21b}. 
\end{proof}

\section{Class B solutions satisfy mass conservation}\label{S:ClassB-mass}
In this section, we prove Lemma \ref{L:mass-conservation}, demonstrating that Class B solutions satisfy mass conservation.     Let $P_{<N}$ be the Littlewood-Paley projection onto frequencies $\lesssim N$.  We note that $P_{<N}^2 \neq P_{<N}$, since the frequency cutoff is smoothed, but nevertheless $P_{<N}^2-P_{<N}$ is a multiplier operator with symbol supported in $|\boldsymbol{\xi}| \sim N$.    We define $P_{>N} = I-P_{<N}$, which yields
$$
\partial_t \| P_{<N} u\|_{L_{\mathbf{x}}^2}^2 = 2 \int P_{<N}u \, \partial_t P_{<N}u \, d\mathbf{x}.
$$
Substituting ZK, we continue as
$$ 
\partial_t \| P_{<N} u\|_{L_{\mathbf{x}}^2}^2 = - 2\int P_{<N} u \, \partial_x \Delta P_{<N}u \, d \mathbf{x} - 2 \int P_{<N} u \, \partial_x P_{<N} u^2 \, d\mathbf{x},
$$
noting that both integrals are finite (absolutely convergent) due to the frequency cutoff (so we are not manipulating infinities!).  By integration by parts
$$ 
\partial_t \| P_{<N} u\|_{L_{\mathbf{x}}^2}^2  =  2\int \nabla P_{<N} u \cdot \partial_x \nabla P_{<N}u \, d \mathbf{x} + 2 \int \partial_x P_{<N}^2 u \, u^2 \, d\mathbf{x}.
$$
The first integral is zero, and for the second integral we insert $I = P_{<N}+P_{>N}$ in front of each copy of $u$ and expand to obtain
\begin{align*}
\partial_t \| P_{<N} u\|_{L_{\mathbf{x}}^2}^2 &=  2 \int \partial_x P_{<N}^2 u \, P_{<N} u \, P_{<N} u \, d\mathbf{x} +  4 \int \partial_x P_{<N}^2 u \, P_{<N} u \, P_{>N} u \, d\mathbf{x} \\
& \qquad +  2 \int \partial_x P_{<N}^2 u \, P_{>N} u \, P_{>N} u \, d\mathbf{x}\,.
\end{align*}
The key is to notice that the first integral becomes zero when $P_{<N}^2$ is replaced by $P_{<N}$, so 
\begin{align*}
\partial_t \| P_{<N} u\|_{L_{\mathbf{x}}^2}^2 &=  -4\int (P_{<N}^2-P_{<N}) u \, P_{<N} u \, \partial_x P_{<N} u \, d\mathbf{x}+  4 \int \partial_x P_{<N}^2 u \, P_{<N} u \, P_{>N} u \, d\mathbf{x} \\
&\qquad + 2 \int \partial_x P_{<N}^2 u \, P_{>N} u \, P_{>N} u \, d\mathbf{x}.
\end{align*}
Now all three integrals involve at least one term at frequency $|\boldsymbol{\xi}| \gtrsim N$.  We use 
H\"older as follows for each of the three terms:
\begin{align*}
\left| \partial_t \| P_{<N} u\|_{L_{\mathbf{x}}^2}^2 \right|  &\lesssim \| (P_{<N}^2-P_{<N}) u \|_{L_{\mathbf{x}}^3} \|P_{<N} u\|_{L_{\mathbf{x}}^6}  \| \partial_x P_{<N} u\|_{L_{\mathbf{x}}^2} \\
&\qquad + \| \partial_x P_{<N}^2 u \|_{L_{\mathbf{x}}^2} \| P_{<N} u \|_{L_{\mathbf{x}}^6} \| P_{>N} u\|_{L_{\mathbf{x}}^3} \\
&\qquad + \| \partial_x P_{<N}^2 u \|_{L_x^2} \| P_{>N} u \|_{L_{\mathbf{x}}^6} \| P_{>N} u\|_{L_{\mathbf{x}}^3}.
\end{align*}
Following with Sobolev embedding, we get 
\begin{align*}
\left| \partial_t \| P_{<N} u\|_{L_{\mathbf{x}}^2}^2 \right| &\lesssim \| (P_{<N}^2-P_{<N}) u \|_{\dot H_\mathbf{x}^{1/2}} \|P_{<N} u\|_{\dot H_\mathbf{x}^1}  \| P_{<N}u\|_{\dot H_\mathbf{x}^1} \\
&\quad + \| P_{<N}^2 u \|_{\dot H_\mathbf{x}^1} \| P_{<N} u \|_{\dot H_\mathbf{x}^1} \| P_{>N} u\|_{\dot H_\mathbf{x}^{1/2}} \\
&\quad + \| P_{<N}^2 u \|_{\dot H_\mathbf{x}^1} \| P_{>N} u \|_{\dot H_\mathbf{x}^1} \| P_{>N} u\|_{\dot H_\mathbf{x}^{1/2}}.
\end{align*}

Since the $\dot H_\mathbf{x}^{1/2}$ norms lie on terms with $P_{>N}$, we can boost to $\dot H_\mathbf{x}^1$ and gain $N^{-1/2}$, i.e., use $\| P_{>N} u \|_{\dot H_\mathbf{x}^{1/2}} \lesssim N^{-1/2} \| u \|_{\dot H_\mathbf{x}^1}$. This gives
$$
\left| \partial_t \| P_{<N} u\|_{L_{\mathbf{x}}^2}^2 \right|   \lesssim N^{-1/2} \| u \|_{\dot H_{\mathbf{x}}^1}^3.
$$
Now integrate in time, for fixed $t_1<t_2$, to obtain
$$
\left| \|P_{<N} u(t_1) \|_{L_{\mathbf{x}}^2}^2 - \|P_{<N} u(t_2) \|_{L_{\mathbf{x}}^2}^2 \right| \lesssim N^{-1/2} \| u \|_{L_{[t_1,t_2]}^\infty \dot H_{\mathbf{x}}^1}^3 |t_2-t_1|.
$$
Send $N\to \infty$, to obtain that 
$$
\|u(t_1) \|_{L_{\mathbf{x}}^2}^2 = \|u(t_2) \|_{L_{\mathbf{x}}^2}^2,
$$
which indicates that the mass at any two distinct times $t_1$ and $t_2$ is the same, completing the proof of Lemma \ref{L:mass-conservation}.

\section{Decomposition of orbitally stable solutions}
\label{S:geom-decomp}

In this section, we introduce three versions of the remainder function: $\epsilon$, $\eta$, and $\zeta$, and derive the equations that each of these functions satisfy, and derive the parameter dynamics.  Some of these lemmas will be proved only under the assumption that the solution is of Class B.  In particular, we will cover the proof of Lemma \ref{L:geom-decomp}.  

Note that in Lemma \ref{L:implicit1}, it is possible to use  $s,k \ll -1$, since $Q_{c,\mathbf{a}}$, $\partial_c Q_{c,\mathbf{a}}$, $\nabla_{\mathbf{a}} Q_{c,\mathbf{a}}$, etc., are smooth and exponentially decaying in space, and $u$ appears as a dual object in the proof.   This will be exploited in Lemma \ref{L:implicit1b}.

\begin{lemma}
\label{L:implicit1}
Suppose $\alpha \ll 1$, $s, k \in \mathbb{R}$.  Suppose $u(\mathbf{x})\in H_{\mathbf{x}}^{s,k}$ (suppressing time dependence)  and there are given $\hat c>0$ and $\hat{\mathbf{a}}\in \mathbb{R}^3$ such that 
$$
\|\hat c^2 u(\hat c \mathbf{x}+\hat{\mathbf{a}}) - Q(\mathbf{x}) \|_{H_{\mathbf{x}}^{s,k}} \leq \alpha.
$$
Then there exists $c>0$ and $\mathbf{a}\in \mathbb{R}^3$ with
$$
|c-\hat c | \lesssim \alpha \quad \mbox{and} \quad |\mathbf{a} - \hat{\mathbf{a}}|  \lesssim \alpha
$$
such that, if we define
$$
\epsilon(\mathbf{x}) = c^2 u(c\mathbf{x}+\mathbf{a}) - Q(\mathbf{x}),
$$
then $\epsilon$ satisfies the orthogonality conditions
$$
\la \epsilon, \nabla Q \ra =0 \quad \mbox{and} \quad \la \epsilon, Q^2 \ra =0.
$$
Moreover, this defines an infinitely differentiable mapping
$$
H_{\mathbf{x}}^{s,k} \to \mathbb{R}^4\, \quad \text{given by} \quad u \mapsto (c, \mathbf{a}).
$$
Specifically, each of the derivative maps $c'$, $a_j'$, for $j=1,2,3$, are Lipschitz continuous maps $H_{\mathbf{x}}^{s,k} \to H_{\mathbf{x}}^{-s,-k}$.  
\end{lemma}
\begin{proof}
By scaling and translation, we can assume that $\hat c=1$ and $\hat{\mathbf{a}}=0$.
Let $Q_{c,\mathbf{a}}(\mathbf{x}) = c^{-2}Q(c^{-1}(\mathbf{x}-\mathbf{a}))$.  Then
$$
F(u,c,\mathbf{a}) = \begin{bmatrix}  \la u - Q_{c,\mathbf{a}}, \partial_c Q_{c,\mathbf{a}} \ra  \\  \la u - Q_{c,\mathbf{a}}, \nabla_{\mathbf{a}} Q_{c,\mathbf{a}} \ra \end{bmatrix}
$$
defines a mapping
$$
F: H_{\mathbf{x}}^{s,k} \times \mathbb{R}^4 \to \mathbb{R}^4,
$$
for which we know that $F(Q,1,0)=0$.  The mapping $F$ is infinitely differentiable in each component ($u$, $c$, $\mathbf{a}$), and each derivative has uniform norms for $\frac12\leq c \leq 2$ and $\mathbf{a}\in \mathbb{R}^3$.  We compute  the $4$-vector valued first derivative functions as
$$
\la d_u F(u,c,\mathbf{a}), v\ra  = \begin{bmatrix}  \la  v , \partial_c Q_{c,\mathbf{a}} \ra  \\  \la v, \nabla_{\mathbf{a}} Q_{c,\mathbf{a}} \ra \end{bmatrix},
$$
$$
\partial_c F(u,c,\mathbf{a})  = -\begin{bmatrix}  \la  \partial_c Q_{c,\mathbf{a}} , \partial_c Q_{c,\mathbf{a}} \ra  \\  \la \partial_c Q_{c,\mathbf{a}}, \nabla_{\mathbf{a}} Q_{c,\mathbf{a}} \ra \end{bmatrix}+ \begin{bmatrix}  \la u - Q_{c,\mathbf{a}}, \partial_c^2 Q_{c,\mathbf{a}} \ra  \\  \la u - Q_{c,\mathbf{a}}, \partial_c \nabla_{\mathbf{a}} Q_{c,\mathbf{a}} \ra \end{bmatrix},
$$
$$
\partial_{a_j} F(u,c,\mathbf{a})  = -\begin{bmatrix}  \la \partial_{a_j}   Q_{c,\mathbf{a}} , \partial_c Q_{c,\mathbf{a}} \ra  \\  \la \partial_{a_j} Q_{c,\mathbf{a}}, \nabla_{\mathbf{a}} Q_{c,\mathbf{a}} \ra \end{bmatrix}+ \begin{bmatrix}  \la u - Q_{c,\mathbf{a}}, \partial_{a_j}\partial_c Q_{c,\mathbf{a}} \ra  \\  \la u - Q_{c,\mathbf{a}}, \partial_{a_j} \nabla_{\mathbf{a}} Q_{c,\mathbf{a}} \ra \end{bmatrix}.
$$
It is straightforward to check that the $4\times 4$ matrix-valued map $\partial_{c,\mathbf{a}}F(u,c,\mathbf{a})$ is invertible at $(u,c, \mathbf{a}) = (Q,1,0)$, and thus, by the implicit function theorem, the mappings  $u \mapsto c(u)$ and $u \mapsto \mathbf{a}(u)$ that satisfy the $4$-vector equation
$$
F(u, c(u), \mathbf{a}(u))=0
$$
exist and are unique.  By implicit differentiation, the following $4$-vector valued identity holds
\begin{align*}
0 &= \la d_u  [F(u,c(u),\mathbf{a}(u))], v\ra \\
&= \la (d_uF)(u,c(u),\mathbf{a}(u)), v\ra  
+ (\partial_c F)(u,c(u),\mathbf{a}(u)) \la c'(u), v\ra  \\
 &  \qquad + \sum_{j=1}^3 (\partial_{a_j} F)(u,c(u),\mathbf{a}(u)) \la a_j'(u), v\ra.
\end{align*}
This is actually four equations in the four unknowns $\la c'(u), v\ra$ and $\la a_j'(u), v\ra$, for $j=1, 2, 3$.
Due to the invertibility of $\partial_{c,\mathbf{a}}F(u,c,\mathbf{a})$, we can solve for $\la c'(u), v\ra$ and $\la a_j'(u), v\ra$, for $j=1, 2, 3$.  We obtain that $c'(u)$, which is a bounded linear map $H_{\mathbf{x}}^{s,k} \to \mathbb{R}$, and hence, associated with an element of $H_{\mathbf{x}}^{-s,-k}$.  Thus, $c'$ itself a Lipschitz continuous map $c':H_{\mathbf{x}}^{s,k} \to H_{\mathbf{x}}^{-s,-k}$.  
\end{proof}

\begin{lemma}
\label{L:implicit1b}
There exists $\alpha>0$ sufficiently small so that, if $u$ is a Class B $\alpha$-orbitally stable solution to the 3D ZK, then there exist \emph{unique} translation $\mathbf{a}(t)$ and scale parameters $c(t)>0$ so that $\epsilon$ defined by
$$\epsilon(\mathbf{x},t) = c(t)^2u(c(t)\mathbf{x}+\mathbf{a}(t),t) - Q(\mathbf{x})$$
satisfies, for all $t$, the orthogonality conditions
$$\la \epsilon(t), \nabla Q \ra =0 \quad \mbox{and} \quad \la \epsilon(t), Q^2 \ra =0.$$
The translation and scale parameters $\mathbf{a}(t)=(a_x(t),a_y(t),a_z(t))$ and $c(t)$ are $C^{1,\frac23}$ functions.
\end{lemma}

We remark that even though the function space mappings $c: H^{s,k} \to \mathbb{R}$ and $a_j:H^{s,k} \to \mathbb{R}$ in Lemma \ref{L:implicit1} are infinitely differentiable, the compositions $c(t) = c(u(t))$ and $a_j(t) =c(u(t))$ are not more than once differentiable, since we do not have a meaning for $u''(t)$ when $u(t)$ is a Class B solution.  Lemma \ref{L:implicit1b} asserts that these parameters have H\"older continuous first derivatives of order $\frac23$, and this seems to be the best we can do.  To see that $u''(t)$ is not defined, formally compute, by substitution of ZK,
$$
\partial_t^2 u = -\partial_t (\partial_x \Delta u + \partial_x (u^2))= - \partial_x \Delta \partial_t u - 2\partial_x( u \, \partial_t u).
$$
All that we know is $u\in H_{\mathbf{x}}^1$ and $\partial_t u \in H_{\mathbf{x}}^{-2}$, and there is no way to define the product of two such functions in 3D.  

\begin{proof}[Proof of Lemma \ref{L:implicit1b}]
To see this, we apply Lemma \ref{L:implicit1} at each time $t$ with  $s=-4$ and $k=0$.   Since in Lemma \ref{L:implicit1}, $c$ and $\mathbf{a}$ are functions of $u$, we have $c: H_{\mathbf{x}}^s \to \mathbb{R}$,
$$
c': H_{\mathbf{x}}^s \to   (H_{\mathbf{x}}^s)^* \simeq H_{\mathbf{x}}^{-s},
$$ 
and for $u_1, u_2 \in H_{\mathbf{x}}^s$,
$$
\| c'(u_2)-c'(u_1) \|_{H_{\mathbf{x}}^{-s}} \lesssim \| u_2 - u_1 \|_{H_x^s}.
$$
Similar statements hold for $a_j'$.    
Taking $c(t) = c(u(t))$ and $\mathbf{a}(t) = \mathbf{a}(u(t))$, we obtain
$$
c'(t) = \la c'(u(t)),u'(t)\ra \,, \qquad a_j'(t) = \la a_j'(u(t)), u'(t) \ra. 
$$ 
With our choice of $s=-4$, we have $c'(u(t)) \in H_{\mathbf{x}}^4$ and $a_j'(u(t)) \in H_{\mathbf{x}}^4$, and thus, we need to estimate $u'(t) \in H_{\mathbf{x}}^{-4}$.    Since the argument for $a_j'(t)$ is similar, we only write the argument for $c'(t)$.  Note that for $t_1<t_2$,
\begin{align*}
c'(t_2) - c'(t_1) &= \la c'(u(t_2)),u'(t_2) \ra - \la c'(u(t_1)),u'(t_1)\ra\\
&= \la c'(u(t_2))-c'(u(t_1)), u'(t_2) \ra + \la c'(u(t_1)), u'(t_2)-u'(t_1) \ra,
\end{align*}
and thus,
\begin{align*}
|c'(t_2)-c'(t_1)| &\lesssim \| c'(u(t_2)) - c'(u(t_1))\|_{H_{\mathbf{x}}^4} \| u'(t_2) \|_{H_{\mathbf{x}}^{-4}} + \| c'(u(t_2)) \|_{H_{\mathbf{x}}^4} \|u'(t_2)-u'(t_1) \|_{H_{\mathbf{x}}^{-4}} \\
&\lesssim \|u(t_2) - u(t_1) \|_{H_{\mathbf{x}}^{-4}}  \| u'(t_2) \|_{H_{\mathbf{x}}^{-4}} + \|u'(t_2)-u'(t_1) \|_{H_{\mathbf{x}}^{-4}} \, .
\end{align*}
By \eqref{E:wk-106b} and \eqref{E:wk-107b},
$$
|c'(t_2)-c'(t_1)| \lesssim |t_2-t_1|^{2/3}.
$$
\end{proof}

Now we prove the remaining properties of the parameters $c(t)=c(u(t))$ and $\mathbf{a}(t)=\mathbf{a}(u(t))$ asserted in Lemma \ref{L:geom-decomp}.  Let $u$ be a solution to the 3D ZK that is $\alpha$-orbitally stable and let $c(t)$ and $\mathbf{a}(t)$ be the unique parameters so that the orthogonality conditions
$$
\la \epsilon, Q^2 \ra =0 \quad \mbox{and} \quad \la \epsilon, \nabla Q \ra=0
$$
hold.
Let
\begin{equation}
\label{E:ep-def}
\epsilon(\mathbf{x},t) = c(t)^2 \,u\left(c(t) \mathbf{x} + \mathbf{a}(t),t\right) - Q(\mathbf{x})
\end{equation}
and
$$
Q_{c, \mathbf{a}}(\mathbf{x}) = c^{-2}Q(c^{-1}(\mathbf{x}-\mathbf{a})).
$$
We further extend this notational convention to an arbitrary function $f(\mathbf{x})$, denoting 
\begin{equation}
\label{E:shift-notation}
f_{c, \mathbf{a}}(\mathbf{x}) \defeq c^{-2}f(c^{-1}(\mathbf{x}-\mathbf{a})).
\end{equation}
In particular,
$$
\nabla Q_{c,\mathbf{a}} = c^{-1} (\nabla Q)_{c,\mathbf{a}} \quad \mbox{and} \quad \partial_c Q_{c,\mathbf{a}} = -c^{-1} (\Lambda Q)_{c,\mathbf{a}}.
$$

Rewriting \eqref{E:ep-def} as
$$
u(\mathbf{x},t) = Q_{c,\mathbf{a}}(\mathbf{x}) + c^{-2}\epsilon(c^{-1}(\mathbf{x}-\mathbf{a})), 
$$
and substituting into the 3D ZK equation, using the equation for $Q$, we obtain the equation for $\epsilon$:
\begin{align} \label{ep-eq}
c^3 \partial_t \epsilon 
=  \partial_x \mathcal{L}\epsilon + c^2c' \Lambda Q + c^2(\mathbf{a}' - c^{-2} \mathbf{i}) \cdot \nabla Q 
+ c^2c' \Lambda \epsilon + c^2(\mathbf{a}' - c^{-2}\mathbf{i}) \cdot \nabla \epsilon  - \partial_x \epsilon^2,
\end{align}
where
$$
\mathcal{L} = I - \Delta - 2Q.
$$
Now let
\begin{equation}
\label{E:eta-def}
\eta(\mathbf{x},t) = c^{-2}\epsilon(c^{-1}(\mathbf{x}-\mathbf{a}))
\end{equation}
so that
$$
u(\mathbf{x},t) = Q_{c, \mathbf{a}}(\mathbf{x}) + \eta(\mathbf{x},t).
$$
The equation for $\eta$ is
\begin{equation}
\label{E:eta-eq}
\partial_t \eta = - \partial_x \Delta \eta - 2 \partial_x ( Q_{c,\mathbf{a}} \eta) - \partial_x \eta^2 + c'c^{-1} (\Lambda Q)_{c,\mathbf{a}} + c^{-1}(\mathbf{a}' - c^{-2} \mathbf{i}) \cdot (\nabla Q)_{c,\mathbf{a}}.
\end{equation}

\begin{lemma}
\label{L:ODE-bounds}
Suppose that $u$ is a Class B, $\alpha$-orbitally stable solution to the 3D ZK with associated parameters $\mathbf{a}(t)$ and $c(t)$ as in Lemma \ref{L:geom-decomp}.  Let $b(t) \defeq \|\epsilon(t) \|_{L_{\mathbf{x}}^2}$.   Then $\mathbf{a}(t)=(a_x(t),a_y(t),a_z(t))$ and $c(t)$ are $C^{1,\frac23}$ functions, and moreover,
\begin{equation}
\label{E:param-ODEs}
\begin{aligned}
&\left|c^2c' - \frac{\la \epsilon, \mathcal{L} \partial_x (Q^2)\ra}{ \la \Lambda Q, Q^2 \ra} \right| \lesssim b^2,  && 
\left| c^2(a_x' -c^{-2}) - \frac{ \la \epsilon, \mathcal{L} Q_{xx}\ra }{ \|Q_x\|_{L^2}^2 } \right| \lesssim b^2, \\
&\left| c^2a_y' - \frac{ \la \epsilon, \mathcal{L} Q_{xy}\ra }{ \|Q_y\|_{L^2}^2 } \right| \lesssim b^2, && 
\left| c^2a_z' - \frac{ \la \epsilon, \mathcal{L} Q_{xz}\ra }{ \|Q_z\|_{L^2}^2 } \right| \lesssim b^2.
\end{aligned}
\end{equation}
\end{lemma}


\begin{proof}
Multiplying equation \eqref{ep-eq} by $Q^2$ and $Q_{x}$, $Q_{y}$, $Q_{z}$, respectively, and integrating by parts, we formally obtain the following equations (with regularization arguments to make computations rigorous)
\begin{align*}
0 &=  -\la \epsilon, \mathcal{L} \partial_{x} (Q^2)\ra + c^2c' \la \epsilon, Q^2-\Lambda Q^2\ra +c^2c' \la \Lambda Q, Q^2\ra\\
&+c^2(a'_{x} -c^{-2}) \left[\la Q_{x},Q^2\ra-\la \epsilon, (Q^2)_{x}\ra\right]\\
&+c^2a'_{y} \left[\la Q_{y},Q^2\ra-\la \epsilon, (Q^2)_{y}\ra\right]+c^2a'_{z} \left[\la Q_{z},Q^2\ra-\la \epsilon, (Q^2)_{z}\ra\right] +\la \epsilon^2, \partial_{x}(Q^2)\ra\\
\end{align*}
and
\begin{align*}
0 &=  -\la \epsilon, \mathcal{L}  Q_{xx}\ra + c^2c' \la \epsilon, Q_{x}-\Lambda Q_{x}\ra +c^2c' \la \Lambda Q, Q_{x}\ra\\
&+c^2(a'_{x} -c^{-2}) \left[\la Q_{x},Q_{x}\ra-\la \epsilon, Q_{xx}\ra\right]\\
&+c^2a'_{y} \left[\la Q_{y},Q_{x}\ra-\la \epsilon, Q_{xy}\ra\right]+c^2a'_{z} \left[\la Q_{z},Q_{x}\ra-\la \epsilon, Q_{xz}\ra\right] +\la \epsilon^2, Q_{xx}\ra, 
\end{align*}
similarly, 
\begin{align*}
0 &=  -\la \epsilon, \mathcal{L}  Q_{yx}\ra + c^2c' \la \epsilon, Q_{y}-\Lambda Q_{y}\ra 
+c^2c' \la \Lambda Q, Q_{y}\ra\\
&+c^2(a'_{x} -c^{-2}) \left[\la Q_{x},Q_{y}\ra-\la \epsilon, Q_{yx}\ra\right]\\
&+c^2a'_{y} \left[\la Q_{y},Q_{y}\ra-\la \epsilon, Q_{yy}\ra\right]+c^2a'_{z} \left[\la Q_{z},Q_{y}\ra-\la \epsilon, Q_{yz}\ra\right] +\la \epsilon^2, Q_{yx}\ra, 
\end{align*}
and
\begin{align*}
0 &=  -\la \epsilon, \mathcal{L}  Q_{zx}\ra + c^2c' \la \epsilon, Q_{z}-\Lambda Q_{z}\ra 
+c^2c' \la \Lambda Q, Q_{z}\ra\\
&+c^2(a'_{x} -c^{-2}) \left[\la Q_{x},Q_{z}\ra-\la \epsilon, Q_{zx}\ra\right]\\
&+c^2a'_{y} \left[\la Q_{y},Q_{z}\ra-\la \epsilon, Q_{zy}\ra\right]+c^2a'_{z} \left[\la Q_{z},Q_{z}\ra-\la \epsilon, Q_{zz}\ra\right]+\la \epsilon^2, Q_{zx}\ra. 
\end{align*}
Noting that $\la \Lambda Q, Q^2\ra=\|Q\|_{L^3}^3$ and $\la \Lambda Q, \nabla Q \ra=0$ ($L^2$-critical case), we deduce the following linear system
\begin{equation}\label{System}
(A - B(\epsilon)) 
\begin{bmatrix} 
c^2c' \\ c^2(a'_{x} -c^{-2})  \\ c^2a'_{y}\\c^2a'_{z} 
\end{bmatrix} 
=  \begin{bmatrix} 
\la \epsilon, \mathcal{L} \partial_{x} (Q^2)\ra \\ 
\la \epsilon, \mathcal{L}  Q_{xx}\ra \\ 
\la \epsilon, \mathcal{L}  Q_{xy}\ra \\
\la \epsilon, \mathcal{L}  Q_{xz}\ra 
\end{bmatrix} 
- \begin{bmatrix} 
\la \epsilon^2, \partial_{x}(Q^2)\ra \\ 
\la \epsilon^2, Q_{xx}\ra \\ 
\la \epsilon^2, Q_{xy}\ra \\
\la \epsilon^2, Q_{xz}\ra  
\end{bmatrix},
\end{equation}
where 
$$
A = 
\begin{bmatrix} 
\|Q\|_{L^3}^3 & & & \\ 
& \|Q_{x}\|_{L^2}^2 && \\ 
& & \|Q_{y}\|_{L^2}^2& \\
&&& \|Q_{z}\|_{L^2}^2 
\end{bmatrix}
$$
and
$$
B(\epsilon) = 
\begin{bmatrix} 
\la \epsilon, \Lambda Q^2 - Q^2\ra & \la \epsilon, (Q^2)_{x}\ra  &  \la \epsilon, (Q^2)_{y}\ra & \la \epsilon, (Q^2)_{z}\ra \\ 
\la \epsilon, \Lambda Q_{x} - Q_{x}\ra & \la \epsilon, Q_{xx}\ra  &  \la \epsilon, Q_{xy}\ra & \la \epsilon, Q_{xz}\ra \\ 
\la \epsilon, \Lambda Q_{y} - Q_{y}\ra & \la \epsilon, Q_{xy}\ra  &  \la \epsilon, Q_{yy}\ra & \la \epsilon, Q_{yz}\ra\\
\la \epsilon, \Lambda Q_{z} - Q_{z}\ra & \la \epsilon, Q_{xz}\ra  &  \la \epsilon, Q_{yz}\ra & \la \epsilon, Q_{zz}\ra 
\end{bmatrix}.
$$
Note that the matrix $B(\epsilon)$ has norm $\|B(\epsilon)\|\lesssim \| \epsilon \|_{L_t^\infty L_{\mathbf{x}}^2}$. Therefore, if $b= \| \epsilon \|_{L_t^\infty L_{\mathbf{x}}^2} \ll 1$, then there exists the inverse matrix $(I+A^{-1}B(\epsilon))^{-1}$, and moreover, the Neumann expansion is given by
$$
(I+A^{-1}B(\epsilon))^{-1}=I+\sum_{k=1}^{\infty}(A^{-1}B(\epsilon))^k.
$$

Setting the matrix $C(\epsilon)=\sum_{k=1}^{\infty}(A^{-1}B(\epsilon))^k$, the system \eqref{System} can be rewritten as 
$$
\begin{bmatrix} 
c^2c' - \frac{\la \epsilon, \mathcal{L} \partial_x (Q^2)\ra}{ \|Q\|_{L^3}^3} \\ 
c^2(a'_{x} -c^{-2}) - \frac{ \la \epsilon, \mathcal{L} Q_{xx}\ra }{ \|Q_{x}\|_{L^2}^2 }  \\ 
c^2a'_{y} - \frac{ \la \epsilon, \mathcal{L} Q_{xy}\ra }{ \|Q_{y}\|_{L^2}^2 }\\
c^2a'_{z}- \frac{ \la \epsilon, \mathcal{L} Q_{xz}\ra }{ \|Q_{z}\|_{L^2}^2 }  
\end{bmatrix} 
=  C(\epsilon)A^{-1}\begin{bmatrix} \la \epsilon, \mathcal{L} \partial_{x} (Q^2)\ra \\ 
\la \epsilon, \mathcal{L}  Q_{xx}\ra \\ 
\la \epsilon, \mathcal{L}  Q_{xy}\ra \\
\la \epsilon, \mathcal{L}  Q_{xz}\ra 
\end{bmatrix} 
- (I+A^{-1}B(\epsilon))^{-1}A^{-1}\begin{bmatrix} \la \epsilon^2, \partial_{x}(Q^2)\ra \\ 
\la \epsilon^2, Q_{xx}\ra \\ 
\la \epsilon^2, Q_{xy}\ra \\
\la \epsilon^2, Q_{xz}\ra  
\end{bmatrix}.
$$

Finally, since $\|C(\epsilon)\|\lesssim b$ and $\|(I+A^{-1}B(\epsilon))^{-1}\|\lesssim 1$, we deduce estimates \eqref{E:param-ODEs}.
\end{proof}


Note that the estimates in \eqref{E:param-ODEs} recast in terms of $\eta$ are the following (where we use the notation \eqref{E:shift-notation}),
\begin{equation}
\label{E:eta-ODEs}
\begin{aligned}
&\left|c c' - \frac{\la \eta, (\mathcal{L} \partial_x (Q^2))_{c,\mathbf{a}} \ra}{ \la \Lambda Q, Q^2 \ra} \right| \lesssim b^2,  && 
\left| c(a_x' -c^{-2}) - \frac{ \la \eta, (\mathcal{L} Q_{xx})_{c,\mathbf{a}} \ra }{ \|Q_x\|_{L^2}^2 } \right| \lesssim b^2, \\
&\left| c a_y' - \frac{ \la \eta, (\mathcal{L} Q_{xy})_{c,\mathbf{a}} \ra }{ \|Q_y\|_{L^2}^2 } \right| \lesssim b^2, && 
\left| c a_z' - \frac{ \la \eta, (\mathcal{L} Q_{xz})_{c,\mathbf{a}} \ra }{ \|Q_z\|_{L^2}^2 } \right| \lesssim b^2.
\end{aligned}
\end{equation}

For convenience, define the functions
$$
f\defeq \frac{\mathcal{L}\partial_x (Q^2)}{\la \Lambda Q, Q\ra} \quad \mbox{and} \quad
\mathbf{g} = \left(  \frac{ \mathcal{L}Q_{xx} }{ \|Q_x\|_{L_{\mathbf{x}}^2}}, \frac{\mathcal{L}Q_{xy}}{ \|Q_y\|_{L_{\mathbf{x}}^2}}, \frac{\mathcal{L}Q_{xz} }{ \|Q_z\|_{L_{\mathbf{x}}^2}}\right). 
$$
Using \eqref{E:eta-ODEs} as a guide, rewrite \eqref{E:eta-eq} as
\begin{equation}
\label{E:eta-eq2}
\begin{aligned}[t]
\partial_t \eta &= - \partial_x \Delta \eta - 2 \partial_x ( Q_{c,\mathbf{a}} \eta) + c^{-2} \la \eta, f_{c,\mathbf{a}}\ra (\Lambda Q)_{c,\mathbf{a}} + c^{-2}\la \eta, \mathbf{g}_{c,\mathbf{a}} \ra \cdot (\nabla Q)_{c,\mathbf{a}}
 \\
& \qquad - \partial_x \eta^2 + (c'c^{-1}  - c^{-2} \la \eta, f_{c, \mathbf{a}}\ra ) (\Lambda Q)_{c,\mathbf{a}} \\
& \qquad + (c^{-1}(\mathbf{a}' - c^{-2} \mathbf{i}) -  \la \eta, \mathbf{g}_{c,\mathbf{a}}\ra )  \cdot (\nabla Q)_{c,\mathbf{a}}
\end{aligned}
\end{equation}
so that now the top line consists of linear terms in $\eta$ and the second and third lines are quadratic.  Let
$$
B = \|b(t) \|_{L_t^\infty},
$$
and define
$$
B \zeta(t) \defeq \eta(t).
$$
Substituting into \eqref{E:eta-eq2}, we obtain
\begin{equation}
\label{E:zeta-1}
\begin{aligned}[t]
\partial_t \zeta &= - \partial_x \Delta \zeta - 2 \partial_x ( Q_{c,\mathbf{a}} \zeta) + c^{-2} \la \zeta, f_{c,\mathbf{a}}\ra (\Lambda Q)_{c,\mathbf{a}} + c^{-2}\la \zeta, \mathbf{g}_{c,\mathbf{a}} \ra \cdot (\nabla Q)_{c,\mathbf{a}}
 \\
& \qquad - B \partial_x \zeta^2 + B \omega_c (\Lambda Q)_{c,\mathbf{a}} + B\boldsymbol{\omega}_{\mathbf{a}} \cdot (\nabla Q)_{c,\mathbf{a}},
\end{aligned}
\end{equation}
where
\begin{equation}
\label{E:omegas}
\omega_c \defeq B^{-2}(c'c^{-1}  - c^{-2} B\la \zeta, f_{c, \mathbf{a}}\ra)\quad \mbox{and} \quad \boldsymbol{\omega}_{\mathbf{a}} \defeq B^{-2}(c^{-1}(\mathbf{a}' - c^{-2} \mathbf{i}) -  Bc^{-2}\la \zeta, \mathbf{g}_{c,\mathbf{a}}\ra).
\end{equation}
By \eqref{E:eta-ODEs}, we have 
$$
|\omega_c| \lesssim 1 \quad \mbox{and} \quad |\boldsymbol{\omega}_{\mathbf{a}}| \lesssim 1.
$$

\section{Monotonicity:  $I_\pm$ lemma for $u$, $J_{\pm}$ lemma for $\eta$}
\label{S:monotonicity}

In this section, we introduce key monotonicity lemmas for controlling the movement of mass of $u$ and $\eta$. The monotonicity properties in various ZK contexts have been used in \cite{FHRY, FHR2, FHR3}. The lemmas below will be needed in later sections.  

\begin{lemma}[weighted Gagliardo-Nirenberg]
\label{L:weighted-GN}
For a weight function $\psi(\mathbf{x})>0$ such that pointwise $|\nabla \psi (\mathbf{x})| \lesssim \psi(\mathbf{x})$, and $E\subset \mathbb{R}^3$ any measurable subset, 
\begin{equation}
\label{E:wk-111}
\int_{E} \psi |u|^3 \, d \mathbf{x} \lesssim  \left( \int_{E} |u|^2 \, d\mathbf{x} \right)^{1/2} \left( \int \psi |u|^2 \, d\mathbf{x} \right)^{1/4} \left( \int \psi (|\nabla u|^2 + |u|^2) \, d \mathbf{x} \right)^{3/4}.
\end{equation}
The estimate holds with constant independent of $E$.
\end{lemma}
\begin{proof}
First, split as follows
$$
\int_E \psi |u|^3 \,d \mathbf{x} = \int_E |u| \cdot \psi^{1/4} |u|^{1/2} \cdot \psi^{3/4} |u|^{3/2} \,d \mathbf{x}.
$$
Applying H\"older with norms $L^2$, $L^4$, and $L^4$, we get
\begin{equation}
\label{E:wk-113}
\int_E \psi |u|^3 \,d \mathbf{x} \leq  \left( \int_E |u|^2\, d\mathbf{x} \right)^{1/2} \left( \int \psi |u|^2\, d\mathbf{x} \right)^{1/4} \left( \int \psi^3 |u|^6 \, d \mathbf{x} \right)^{1/4}.
\end{equation}
Applying Sobolev embedding for the last term, we have
$$ 
\left( \int \psi^3 |u|^6 \, d \mathbf{x} \right)^{1/4} = \|\psi^{1/2} u \|_{L^6}^{3/2} \lesssim \| \nabla[\psi^{1/2} u ] \|_{L^2}^{3/2}. 
$$
Distributing the derivative and using that $|\nabla \psi| \lesssim \psi$, it follows that
$$ 
\left( \int \psi^3 |u|^6 \, d \mathbf{x} \right)^{1/4} \lesssim \left( \int \psi( |\nabla u|^2+|u|^2) \, d \mathbf{x} \right)^{3/4}.
$$
Combining with \eqref{E:wk-113} yields \eqref{E:wk-111}.    
\end{proof}

Recall that if  $u(t)$ is a Class B solution to the 3D ZK with $M(u)=M(Q)$  that is $\alpha$-orbitally stable for $\alpha \ll 1$, and  $\mathbf{a}(t)$ and $c(t)$ are the unique parameters as in Lemma \ref{L:geom-decomp}, then 
$$
|c(t)-1| \leq \alpha
$$
and, with $\mathbf{i}=(1,0,0)$, by \eqref{E:param-ODEs} in Lemma \ref{L:ODE-bounds}, we get
\begin{equation}
\label{E:mon1}
|\mathbf{a}'(t) -\mathbf{i} | \lesssim \alpha.
\end{equation}
By Taylor expansion, $Q(c\mathbf{x}) = Q(\mathbf{x}) + (c-1) \mathbf{x} \cdot \nabla Q(\mathbf{x}) + \cdots $,
and thus,
$$
\| c^{-2} Q(c^{-1} \mathbf{x}) - Q(\mathbf{x}) \|_{H_{\mathbf{x}}^1} \lesssim \alpha.
$$
It follows that
\begin{equation}
\label{E:mon2}
\| u(\mathbf{x}+\mathbf{a}(t), t) - Q(\mathbf{x}) \|_{H_{\mathbf{x}}^1} \lesssim \alpha.
\end{equation}

For the purposes of the following lemma, let $\kappa$ be a constant larger than both implicit constants in \eqref{E:mon1} and \eqref{E:mon2}. 

\begin{figure}[ht]
\includegraphics[scale=0.8]{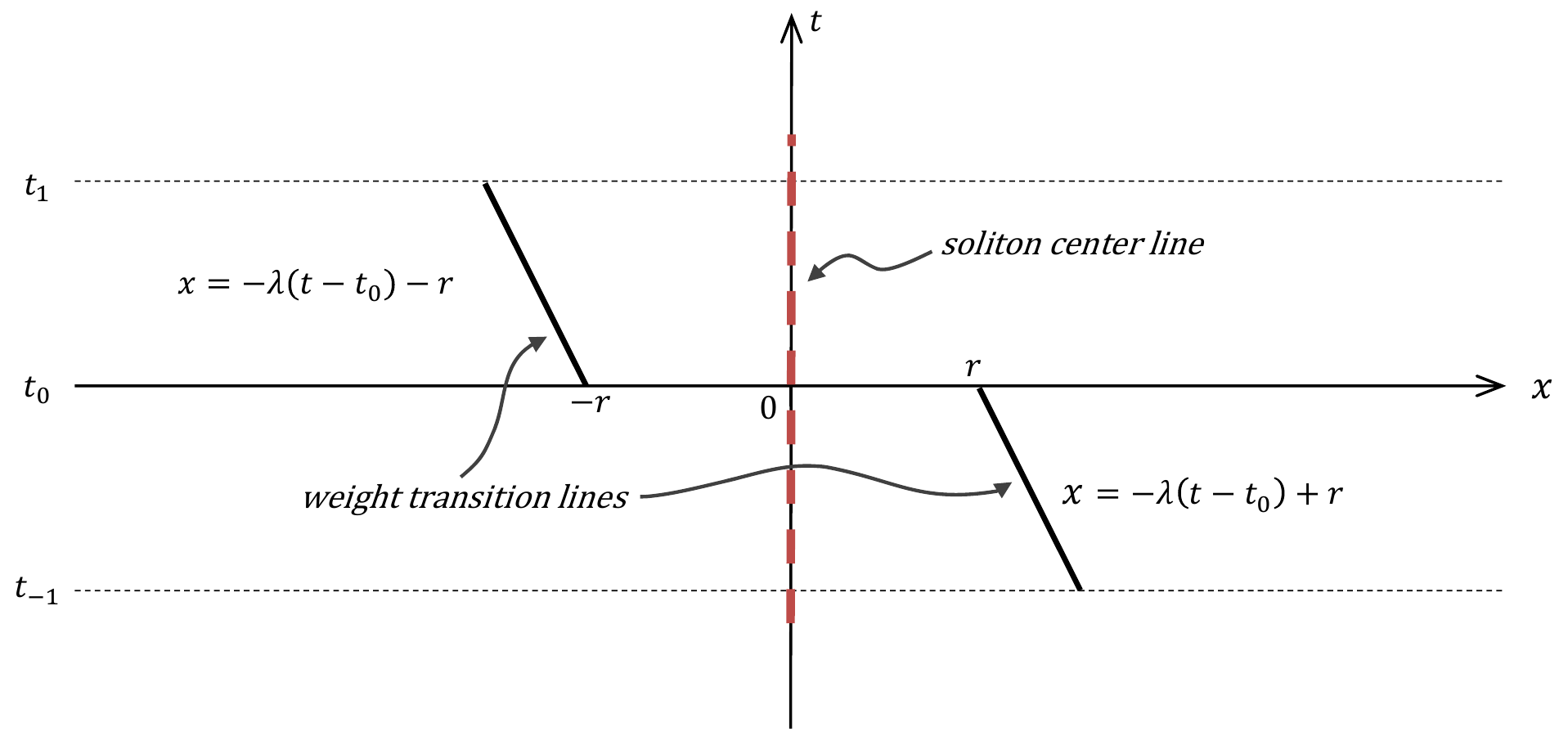}
\caption{The $I_\pm$ estimates.  The vertical line $x=0$ is the soliton center.  The lines $x = -\lambda(t-t_0)+r$ for $t<t_0$ and $x = -\lambda(t-t_0)-r$ for $t>t_0$ are the $\phi_\pm$ weight transition lines in \eqref{E:Ip-right}, \eqref{E:Im-right} and \eqref{E:Ip-left}, \eqref{E:Im-left}, respectively.  Note that this depiction is for $(y,z)=(0,0)$.  Away from $(y,z)=(0,0)$, the weight transition lines are shifted to the left (for $\theta\geq 0$) by $\tan \theta \sqrt{1+y^2+z^2}$.}
\label{F:I}
\end{figure}

\begin{lemma}[conic $I_\pm$ estimates]
\label{L:Ipm-estimates}
Let $u(t)$ be a Class B solution to the 3D ZK with $M(u)=M(Q)$,  that is $\alpha$-orbitally stable for $\alpha \ll 1$, and let $\mathbf{a}(t)$ and $c(t)$ be the unique parameters as in Lemma \ref{L:geom-decomp}.  Let $\delta$ be a constant satisfying $0<16\kappa \alpha \leq \delta \ll 1$.  Let 
$$|\theta| \leq \frac{\pi}{3}-\delta$$ 
be an angle and fix a speed constant $\lambda$ satisfying 
\begin{equation}
\label{E:mon4}
\delta \leq \lambda \leq 1-\delta
\end{equation}
and fix a shift distance $r>0$.  For $K\geq 4\delta^{-1}$, let
\begin{equation}
\label{E:Ipm-de}
I_{\pm, \theta, r, t_0} (t) = \int_{\mathbb{R}^3} \phi_\pm \left(  \cos \theta(x-r+\lambda(t-t_0)) +  \sin \theta \sqrt{1+y^2+z^2} \right) u^2(\mathbf{x}+\mathbf{a}(t), t) \, d\mathbf{x},
\end{equation}
where
$$
\phi_+(x) = \frac{2}{\pi} \operatorname{arctan}(e^{x/K}) \,, \qquad \phi_-(x) = \phi_+(-x)
$$
so that $\phi_+(x)$ increases from $0$ to $1$ and $\phi_-(x)$ decreases from $1$ to $0$.  
Suppose that 
$$
t_{-1}<t_0<t_1.
$$

The estimates for $I_+$ bound the \emph{future in terms of the past}, see Figure \ref{F:I}. We have
\begin{equation}
\label{E:Ip-right}
I_{+,\theta, r,t_0}(t_0) \leq I_{+,\theta, r,t_0}(t_{-1}) + C e^{-\delta r}
\end{equation}
and
\begin{equation}
\label{E:Ip-left}
I_{+,\theta, -r,t_0}(t_1) \leq I_{+,\theta, -r,t_0}(t_0) + Ce^{-\delta r}
\end{equation}
for some $C$ depending on $\delta$ and $K$. The estimates for $I_-$ bound the \emph{past in terms of the future}.  We have
\begin{equation}
\label{E:Im-left}
I_{-,\theta,-r,t_0}(t_0) \leq I_{-,\theta,-r,t_0}(t_1) + Ce^{-\delta r}
\end{equation}
and
\begin{equation}
\label{E:Im-right}
I_{-,\theta, r,t_0}(t_{-1}) \leq I_{-,\theta, r,t_0}(t_0) + Ce^{-\delta r}.
\end{equation}
\end{lemma}

\begin{remark}
For Lemma \ref{L:Ipm-estimates}, one needs only to assume that $u$ is a Class B solutions, since the calculations in the proof can be reproduced using frequency projected regularizations, and the errors managed as in the proof of mass conservation for Class B solutions in \S\ref{S:ClassB-mass}.  We will not carry out the details.
\end{remark}

\begin{proof}[Proof of Lemma \ref{L:Ipm-estimates}]
We consider first the case $\phi=\phi_+$ and $I=I_+$.  The estimates for $\phi_-$ and $I_-$ follow by time inversion, as explained at the end of the proof.   Note that
$$
\phi'(\omega) = \frac{1}{\pi K} \sech(\omega/K)
$$
and
$$
|\phi''(\omega)| \leq \frac{1}{K} |\phi'| \,, \qquad |\phi'''(\omega)| \leq \frac{1}{K^2} |\phi'|.
$$
In the following, 
$$
\phi(\cdots) = \phi( \cos \theta(x-r+\lambda(t-t_0)) +  \sin \theta \sqrt{1+y^2+z^2} ).
$$
Before proceeding, let us note that
$$
\nabla [\phi(\cdots)] =  ( \cos \theta, \frac{y}{\sqrt{1+y^2+z^2}} \sin \theta, \frac{z}{1+y^2+z^2} \sin \theta) \phi'(\cdots),
$$
and thus,
$$
| (\mathbf{a}- \mathbf{i}) \cdot \nabla[ \phi(\cdots) ]| \leq \alpha \kappa \phi' \leq \frac{\delta}{16} \phi'.
$$
Also note that, by integration by parts
\begin{align*}
\indentalign -2\int \phi \, u \, \partial_x \Delta u \,d \mathbf{x} \\
&= \int \left\{ -\partial_x[ \phi(\cdots)] (3u_x^2+u_y^2+u_z^2) - 2\partial_y[\phi(\cdots)] u_xu_y - 2\partial_z[ \phi(\cdots)] u_x u_z \right\} \, d\mathbf{x}\\
& \qquad+ \int \partial_x \Delta[ \phi(\cdots) ] u^2 \,d \mathbf{x} \\
&= \int \phi' \left\{  -\cos\theta  (3u_x^2+u_y^2+u_z^2) - \frac{2y\sin \theta}{\sqrt{1+y^2+z^2}} u_xu_y - \frac{2z\sin \theta}{\sqrt{1+y^2+z^2}} u_x u_z \right\} \, d\mathbf{x}\\
& \qquad+ \int \partial_x \Delta[ \phi(\cdots) ] u^2 \,d \mathbf{x}.
\end{align*}

Using Peter-Paul, we split the products as
$$ 
\frac{2|y| }{\sqrt{1+y^2+z^2}} |u_x u_y| \leq \frac{ y^2 \sqrt{3}}{1+y^2+z^2} u_x^2 + \frac{1}{\sqrt{3}} u_y^2,
$$
$$ 
\frac{2|z| }{\sqrt{1+y^2+z^2}} |u_x u_z| \leq \frac{ z^2 \sqrt{3}}{1+y^2+z^2} u_x^2 + \frac{1}{\sqrt{3}} u_z^2,
$$
and adding, we obtain
$$
\frac{2|y| }{\sqrt{1+y^2+z^2}} |u_x u_y| + \frac{2|z| }{\sqrt{1+y^2+z^2}} |u_x u_z| \leq \frac{1}{\sqrt{3}}(3u_x^2+u_y^2+u_z^2).
$$
Thus, we see that we need the condition $\frac{|\sin \theta|}{\sqrt{3}} < \cos \theta - \delta$, which is implied by the condition $|\tan \theta| \leq \sqrt{3}- 2\delta$, which is implied by the angle condition in the hypothesis.

Note 
\begin{align*}
\partial_x \Delta [ \phi(\cdots)] &= \left[\cos^3\theta + \frac{ (y^2+z^2) \sin^2\theta \cos\theta}{1+y^2+z^2} \right]\phi'''(\cdots) \\ 
& \qquad + \frac{(2+y^2+z^2) \sin \theta \cos \theta}{ (1+y^2+z^2)^{3/2} } \phi''(\cdots),
\end{align*}
and thus,
$$
|\partial_x \Delta [ \phi(\cdots)] | \leq \frac{2}{K} \phi'.
$$
Putting all this together (and using that $\frac{2}{K} \leq \delta$), we obtain
$$
-2\int \phi \, u \, \partial_x \Delta u \,d \mathbf{x}  \leq - \delta \int \phi' (3u_x^2+u_y^2+u_z^2).
$$

We compute
$$
I' = \lambda \cos \theta \int \phi' \, u^2 \, d\mathbf{x} + 2 \int \, \phi \, u \, \nabla u \cdot \mathbf{a}' \,d \mathbf{x} - 2\int \phi \, u \, \partial_x \Delta u \, d\mathbf{x} +\frac43 \int \phi' \, u^3 \, d \mathbf{x}.
$$

Note that
\begin{align*}
2\int \phi \, u \, \nabla u \cdot \mathbf{a}' \,d \mathbf{x} 
&= - \int \mathbf{a} \cdot \nabla [ \phi(\cdots) ]  \, u^2 \, d\mathbf{x}\\
&= - \int (\mathbf{a}-\mathbf{i}) \cdot \nabla [ \phi(\cdots) ]  \, u^2 \, d\mathbf{x} - \cos\theta \int \phi' \, u^2 \, d \mathbf{x}.
\end{align*}

Putting all the inequalities together, yields
$$
I' \leq - \delta \int \phi'  (u^2 + 3u_x^2+u_y^2+u_z^2)+\frac43 \int \phi' \, u^3 \, d \mathbf{x}.
$$

Apply \eqref{E:wk-111} in Lemma \ref{L:weighted-GN} with $\psi( \mathbf{x}) = \phi'( \cdots)$, 
and with the set $E\subset \mathbb{R}^3$ taken to be the exterior of a neighborhood of $0$ large enough so that $\|Q\|_{L^2_E} \leq \kappa \alpha$.  Then it follows that
$$
\| u( \mathbf{x}+\mathbf{a}(t),t)\|_{L^2_E} \leq \| u( \mathbf{x}+\mathbf{a}(t),t) - Q(\mathbf{x}) \|_{L_{\mathbf{x}}^2} + \|Q\|_{L^2_{E}} \leq 2\kappa \alpha \leq \frac{\delta}{8}.
$$
By \eqref{E:wk-111}
$$
\int_E \phi'(\cdots) |u|^3 \, d \mathbf{x} \leq \frac{\delta}{8}  \int \phi'(\cdots) (|\nabla u|^2 + |u|^2) \,d \mathbf{x}.
$$
On $E^c$, we use the standard Gagliardo-Nirenberg inequlaity
$$
\int_{E^c} \phi' |u|^3 \,d \mathbf{x} \leq \sup_{\mathbf{x} \in E^c} |\phi'(\cdots)| \int |u|^3 \,d \mathbf{x} \leq \sup_{\mathbf{x} \in E^c} |\phi'(\cdots)| \|\nabla u \|_{L_t^\infty L_{\mathbf{x}}^2}^{3/2},
$$
combined with the following pointwise bounds for $\phi'(\cdots)$ on $E^c$. 

On $E^c$ (that is, near $0$), if $t<t_0$, then we have $-r<0$ and $\lambda(t-t_0)<0$, so that 
$$
|\phi'(\cdots)| \leq e^{-r} e^{-\lambda|t-t_0|}, 
$$
and consequently,
$$
I_{+,r,t_0}'(t) \lesssim e^{-r} e^{-\delta |t-t_0|}
$$ 
with constant depending on $\delta$ and $K$.  After integrating from $t_{-1}$ to $t_0$, we obtain \eqref{E:Ip-right}.

On $E^c$ (that is, near $0$), if $t>t_0$, then we have $r>0$ and we have $\lambda(t-t_0)>0$, so that
$$
|\phi'(\cdots)| \leq e^{-r} e^{-\lambda|t-t_0|} 
$$
again, and consequently,
$$
I_{+,-r,t_0} '(t) \lesssim e^{-r} e^{-\delta |t-t_0|} $$ 
with constant depending on $\delta$ and $K$.  After integrating from $t_0$ to $t_1$, 
we obtain \eqref{E:Ip-left}.

Now we turn to the $I_-$ estimates involving $\phi_-$.  We will obtain these as consequences of the $I_+$ estimates involving $\phi_+$ by space-time inversion, as follows.  Given $u$, let
$$
\bar u(\mathbf{x},t) = u(-\mathbf{x},-t).
$$
Then $\bar u$ is an $\alpha$-orbitally stable Class B solution to the 3D ZK, with associated modulation parameters $\bar c$ and $\bar{\mathbf{a}}$ satisfying
$$
\bar c(t) = c(-t) \,, \qquad \bar{\mathbf{a}}(t) = -\mathbf{a}(-t).
$$
In referencing $I_+$ and $I_-$ we will add an additional subscript indicating the function $u$ or $\bar u$ as well.  Plugging $\bar u$ into $I_+$, we note the change of variables $\mathbf{x}\to -\mathbf{x}$ in the integration shows that
\begin{equation}
\label{E:pm-conversion}
I_{\bar u, +,-\theta,-r,-t_0}(-t) = I_{u, -, \theta,r,t_0}(t).
\end{equation}
Given $t_{-1}<t_0<t_1$, note that $-t_1<-t_0<-t_{-1}$, so we can apply \eqref{E:Ip-right} with $t_0$ replaced by $-t_0$ and $t_{-1}$ replaced by $-t_1$ to obtain
$$
I_{\bar u, +, -\theta,r,-t_0}(-t_0) \leq I_{\bar u, +, -\theta, r, -t_0}(-t_1) + Ce^{-\delta r}.
$$
Using \eqref{E:pm-conversion}, this gives
$$
I_{u,-,\theta,-r,t_0}(t_0) \leq I_{u,-,\theta,-r,t_0}(t_1) + Ce^{-\delta r},
$$
which is \eqref{E:Im-left}.  We also apply \eqref{E:Ip-left} with $t_0$ replaced by $-t_0$ and $t_1$ replaced by $-t_{-1}$ to obtain
$$
I_{\bar u,+,-\theta, -r,-t_0}(-t_{-1}) \leq I_{\bar u,+,-\theta, -r,-t_0}(-t_0) + Ce^{-\delta r}.
$$
Using \eqref{E:pm-conversion}, this gives
$$
I_{u,-,\theta, r,t_0}(t_{-1}) \leq I_{u,-,\theta, r,t_0}(t_0) + Ce^{-\delta r},
$$
which is \eqref{E:Im-right}.

\end{proof}


Replacing $u$ by $\eta$ in $I_\pm$ gives us new quantities that we denote $J_\pm$ that will be applied to obtain uniform decay estimates for $\tilde \epsilon_n$ in \S \ref{S:uniform-n-decay}.   The main difference is that, out of the four estimates \eqref{E:Ip-right}, \eqref{E:Ip-left}, \eqref{E:Im-right}, \eqref{E:Im-left} for $I_\pm$, only \eqref{E:Ip-right} and \eqref{E:Im-left} have analogues for $J_\pm$. (See also Figure \ref{F:I}.)  The reason is that $\phi_\pm Q$ needs to be small over the relevant interval.  On the interval $[t_{-1},t_0]$ with weight transition line to the right of $x=0$, the product $\phi_+Q$ is small.  On the interval $[t_0,t_1]$ with weight transition line to the left of $x=0$, the product $\phi_-Q$ is small. 

\begin{lemma}[conic $J_\pm$ estimates]
\label{L:Jpm-estimates}
Let $\eta(t)$ be defined by \eqref{E:eta-def}, so that $\eta$ solves \eqref{E:eta-eq}.  Let 
$$
|\theta| \leq \frac{\pi}{3}-\delta
$$ 
be an angle and fix a speed constant $\lambda$ satisfying 
\begin{equation}
\label{E:mon4J}
\delta \leq \lambda \leq 1-\delta,
\end{equation}
and fix a shift distance $r>0$.  For $K\geq 4\delta^{-1}$, let
\begin{equation}
\label{E:J-def}
J_{\pm, \theta, r, t_0} (t) = \int_{\mathbb{R}^3} \phi_\pm \left(  \cos \theta(x-r+\lambda(t-t_0)) +  \sin \theta \sqrt{1+y^2+z^2} \right) \eta^2(\mathbf{x}+\mathbf{a}(t), t) \, d\mathbf{x},
\end{equation}
where
$$
\phi_+(x) = \frac{2}{\pi} \operatorname{arctan}(e^{x/K}) \,, \qquad \phi_-(x) = \phi_+(-x)
$$
so that $\phi_+(x)$ increases from $0$ to $1$ and $\phi_-(x)$ decreases from $1$ to $0$.  
Suppose that 
$$
t_{-1}<t_0<t_1.
$$
The estimate for $J_+$ bounds the \emph{future in terms of the past}, and is only available on the right of the soliton:
\begin{equation}
\label{E:Jp-right}
J_{+,\theta, r,t_0}(t_0) \leq J_{+,\theta, r,t_0}(t_{-1}) + C e^{-\delta r} \| \eta \|_{L_{[t_{-1},t_0]}^\infty L_{\mathbf{x}}^2}^2 
\end{equation}
for some $C$ depending on $\delta$ and $K$.   The estimate for $J_-$ bounds the \emph{past in terms of the future}, and is only available on the left of the soliton:
\begin{equation}
\label{E:Jm-left}
J_{-,\theta,-r,t_0}(t_0) \leq J_{-,\theta,-r,t_0}(t_1) + Ce^{-\delta r} \| \eta \|_{L_{[t_0,t_1]}^\infty L_{\mathbf{x}}^2}^2. 
\end{equation}
\end{lemma}

\begin{proof}
We carry out only the proof of \eqref{E:Jp-right} for $J_+$ with $\phi_+$, and suppress the subscript notation.  Abbreviating the expression for $J$ by suppressing the arguments of $\phi$ and $\eta$,
$$
J = \int_{\mathbb{R}^3} \phi \eta^2 \, d\mathbf{x},
$$
we have
$$
J' = \lambda \cos\theta \int_{\mathbb{R}^3} \phi'  \, \eta^2 \, d\mathbf{x} + 2\mathbf{a}' \cdot \int_{\mathbb{R}^3} \phi \,\eta \, \nabla \eta  \, d\mathbf{x} + 2\int_{\mathbb{R}^3} \phi\, \eta \,\partial_t\eta \, d\mathbf{x}.
$$
Using that $\nabla [ \phi(\cdots) ] = \phi'(\cdots)\boldsymbol{\Omega}_\theta(y,z)$, where
$$
\boldsymbol{\Omega}_\theta(y,z) =  \left( \cos \theta, \sin \theta \frac{y}{\sqrt{1+y^2+z^2}}, \sin \theta \frac{z}{\sqrt{1+y^2+z^2}}\right),
$$
combined with integration by parts in the middle term, gives
$$
J' =  
\int_{\mathbb{R}^3} [\lambda \cos\theta - \mathbf{a}'\cdot \boldsymbol{\Omega}_\theta] \, \phi'  \, \eta^2  \, d\mathbf{x} + 2\int_{\mathbb{R}^3} \phi\, \eta \,\partial_t\eta \, d\mathbf{x}.
$$
Replacing $\mathbf{a}' = \mathbf{a}'-\mathbf{i} + \mathbf{i}$, yields
$$
J' = -(1-\lambda) \cos \theta \int \phi' \, \eta^2 \, d\mathbf{x} + 2\int_{\mathbb{R}^3} \phi\, \eta \,\partial_t\eta \, d\mathbf{x} + ( \mathbf{i}- \mathbf{a}')\cdot \int_{\mathbb{R}^3} \boldsymbol{\Omega}_\theta \, \phi'  \, \eta^2  \, d\mathbf{x}. 
$$
Plugging in \eqref{E:eta-eq}, we obtain
\begin{equation}
\label{E:Jp-1}  
\begin{aligned}
J' = \; & -(1-\lambda) \cos \theta \int \phi' \, \eta^2 \, d\mathbf{x} 
- 2\int_{\mathbb{R}^3} \phi\, \eta \, \partial_x \Delta \eta \, d\mathbf{x} \\
&- 4\int_{\mathbb{R}^3} \phi \, \eta \, \partial_x (Q_{c,\mathbf{a}} \eta) \, d\mathbf{x} 
-2 \int_{\mathbb{R}^3} \phi \, \eta \, \partial_x (\eta^2) \, d\mathbf{x}  \\
\quad&+ c' c^{-1} \int_{\mathbb{R}^3} (\Lambda Q)_{c,\mathbf{a}} \phi \, \eta  \, d\mathbf{x} 
+ c^{-1} ( \mathbf{a}- c^{-2}\mathbf{i}) \cdot \int_{\mathbb{R}^3} \, (\nabla Q)_{c,\mathbf{a}} \phi  \, \eta \\
\quad&+(\mathbf{a}' - \mathbf{i})\cdot \int_{\mathbb{R}^3} \boldsymbol{\Omega}_\theta \, \phi'  \, \eta^2  \, d\mathbf{x} \\
= \; & A_1+A_2+A_3+A_4+A_5+A_6+A_7.
\end{aligned}
\end{equation}
We note that $1-\lambda \geq \delta$ and by the same calculations as in the proof of Lemma \ref{L:Ipm-estimates}, 
$$
A_2=-2\int \phi \, \eta \, \partial_x \Delta \eta \, d \mathbf{x} \leq - \delta \int \phi' |\nabla \eta|^2 \,d \mathbf{x}.
$$
Thus, the first two terms $A_1$ and $A_2$ in \eqref{E:Jp-1} are ``good terms'' with the negative upper bound
\begin{equation}
\label{E:Jp-2}
A_1+A_2 \lesssim -\delta \int \phi' (|\nabla \eta|^2 +\eta^2) \, d\mathbf{x}.
\end{equation}  

Note that $\phi(\omega) \lesssim e^{\omega/K}$ for all $x\in \mathbb{R}$ (although it is a terrible estimate for $\omega\gg 1$), and recall $K \sim \delta^{-1}$ and $\cos\theta \geq \frac12$.  With
$$
\omega = \cos\theta(-r + \lambda(t-t_0)) + (x \cos \theta + \sqrt{1+y^2+z^2} \sin \theta),
$$
we have
$$
\phi(\omega) \lesssim e^{\delta(-r+ \lambda(t-t_0))} e^{+\delta |\mathbf{x}|}.
$$
Recall that since the $\eta$ terms are evaluated at $\mathbf{x}+\mathbf{a}(t)$, the functions $Q_{c,\mathbf{a}}$, $\nabla Q_{c,\mathbf{a}}$ and $\Lambda Q_{c,\mathbf{a}}$ are exponentially concentrated near $\mathbf{x}=0$.  Hence,
$$
\phi(\omega) Q_{c,\mathbf{a}}(\mathbf{x}+\mathbf{a}(t)) \lesssim e^{\delta(-r+\lambda(t-t_0))} e^{-|\mathbf{x}|/4},
$$
and similarly, for $\phi |\nabla Q_{c,\mathbf{a}}|$ and $\phi |\Lambda Q_{c,\mathbf{a}}|$.  
For $t<t_0$, this is a good estimate and can be written as
$$
\phi(\omega) Q_{c,\mathbf{a}}(\mathbf{x}+\mathbf{a}(t)) \lesssim e^{-\delta r} e^{-\delta^2 |t-t_0|} e^{-|\mathbf{x}|/4}.
$$
In \eqref{E:Jp-1}, this estimate is used to control the three terms $A_3$, $A_5$, and $A_6$ and to obtain the bounds (using also \eqref{E:eta-ODEs}),
$$
|A_3|+|A_5|+|A_6| \lesssim  e^{-\delta r} e^{-\delta^2|t-t_0|} \| \eta\|_{L_{t\in [t_{-1},t_0]}^\infty L_{\mathbf{x}}^2}^2.
$$
 
In \eqref{E:Jp-1}, it remains to consider $A_4$ and $A_7$, given by  
$$
A_4=-2 \int_{\mathbb{R}^3} \phi \, \eta \, \partial_x (\eta^2) \, d\mathbf{x} \,,  \qquad A_7= (\mathbf{a}' - \mathbf{i})\cdot \int_{\mathbb{R}^3} \boldsymbol{\Omega}_\theta \, \phi'  \, \eta^2  \, d\mathbf{x}.
$$
By integration by parts,
$$
A_4 = \frac43 \cos \theta \int \phi' \, \eta^3 \, d\mathbf{x},
$$
and by Lemma \ref{L:weighted-GN},
$$
|A_4| \lesssim \| \eta\|_{L_{\mathbf{x}}^2} \int \phi' (|\nabla \eta|^2 + \eta^2) \, d\mathbf{x}.
$$
Since $\|\eta\|_{L_{\mathbf{x}}^2} \ll \delta$, this term is absorbed by the right side of \eqref{E:Jp-2}.  Since $|\mathbf{a}'-\mathbf{i}| \ll \delta$ and $|\boldsymbol{\Omega}_\theta| \leq 1$, we also have that $A_7$ is absorbed by the right side of \eqref{E:Jp-2}.

Combining the above bounds into \eqref{E:Jp-1}, we have for $t<t_0$,
$$
J'(t) \leq e^{-\delta r} e^{-\delta^2 |t-t_0|} \| \eta \|_{L_{t\in [t_{-1},t_0]}^\infty L_{\mathbf{x}}^2}^2.
$$
Integrating from $t=t_{-1}$ to $t=t_0$ gives \eqref{E:Jp-right}.
\end{proof}

\section{Weak convergence implies asymptotic stability}
\label{S:proof-main-theorem}

In this section, we obtain Lemma \ref{L:u-decay} below as a consequence of monotonicity estimates in Lemma \ref{L:Ipm-estimates}.  At the end of the section, Lemma \ref{L:u-decay} is applied to show that Theorem \ref{T:main} follows from Proposition \ref{P:wk-lim} and Proposition \ref{P:rigidity}.  We note that Lemma \ref{L:u-decay} is also applied in \S\ref{S:app-mon} to prove Lemma \ref{L:exp-decay}, part of the proof of Proposition \ref{P:wk-lim} itself.

With $\phi_+$ as defined in Lemma \ref{L:Ipm-estimates}, let
\begin{equation}
\label{E:pf-main-3}
\frac{1}{c_*} \defeq \frac{1}{\|Q\|_{L_{\mathbf{x}}^2}^2} \limsup_{t\nearrow +\infty} \int \phi_+(x  + \frac1{10}t) u^2(\mathbf{x}+ \mathbf{a}(t),t) \, d\mathbf{x}.
\end{equation}
From the assumed orbital stability of $u$, we have
$$
|c_*-1| \lesssim \alpha_0 \quad \mbox{and} \quad |\mathbf{a}'(t) - c_*^{-2} \mathbf{i}| \lesssim \alpha_0.
$$
To fix reference constants, take $\alpha_0$ small enough so that 
$$
| \mathbf{a}'(t) - \mathbf{i} | \leq \frac{1}{100} \quad \mbox{and} \quad |c_*-1| \leq \frac{1}{100}.
$$

\begin{figure}
\begin{center}\includegraphics[scale=0.87]{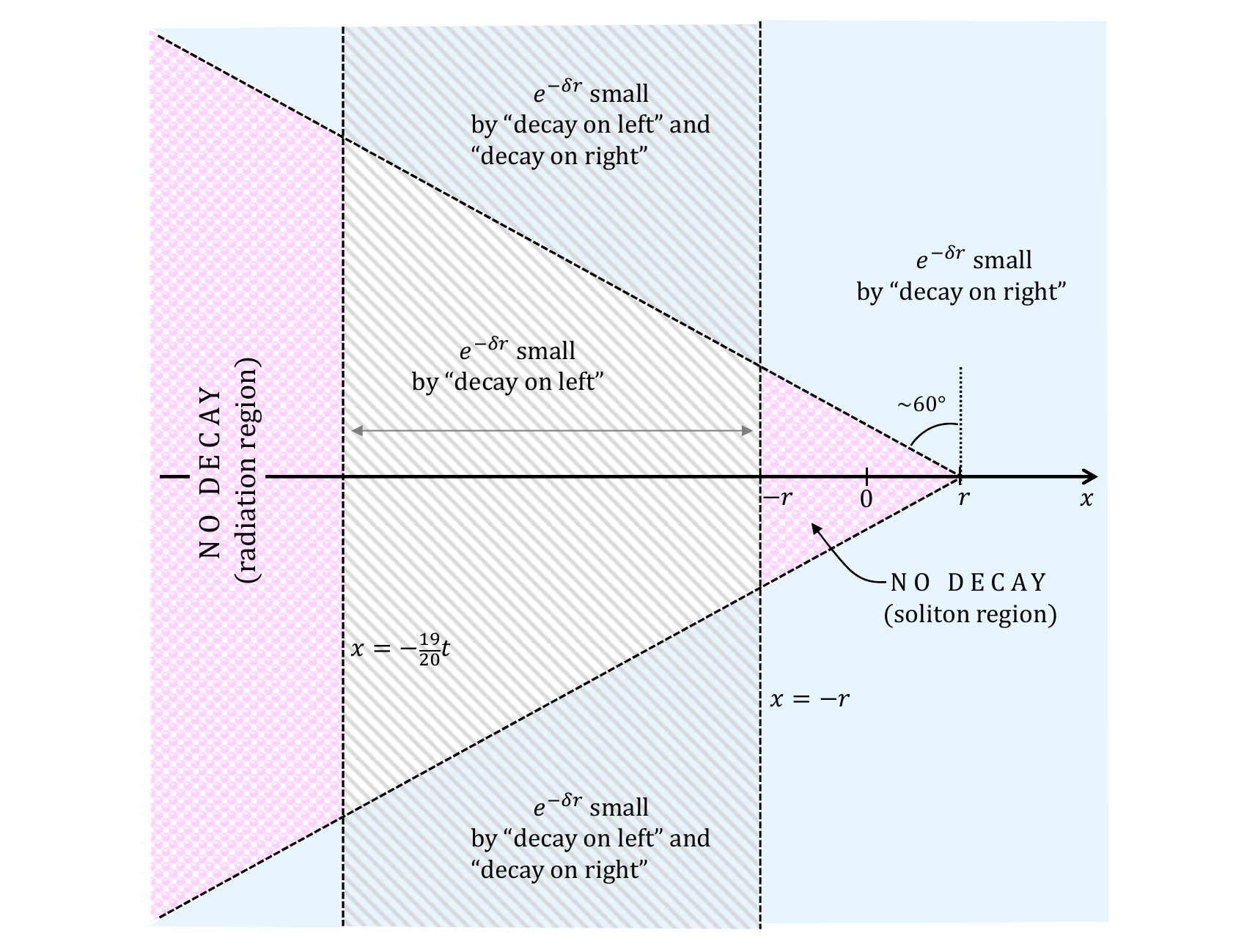}\end{center}
\caption{\label{F:decay}In Lemma \ref{L:u-decay}, \eqref{E:pf-main-1} gives a ``decay on the right estimate'', and \eqref{E:pf-main-5} gives a ``decay on the left estimate''.  The weight $\phi_+(\rho)$ transitions from $0$ to $1$ smoothly as $\rho$ moves from left to right across $0$.  Thus $\rho=0$ corresponds to a ``transition line''.  In \eqref{E:pf-main-1}, $\rho>0$ corresponds to $x>r-\tan \theta \sqrt{1+y^2+z^2}$, where we can take $\theta$ close to $60^\circ$.  Thus, this gives decay in the conic region pictured.  For \eqref{E:pf-main-5}, we take $\theta=0$, so this ``decay on the left'' estimate occurs between the vertical lines $x=-\frac{19}{20}t$ and $x=-r$.  When the two are combined, we obtain $L^2$-smallness outside the triangular region around $0$ but we have no estimate in the region labeled ``no decay''.}
\end{figure}

See Figure \ref{F:decay} for a depiction of the estimates in the following lemma.

\begin{lemma}
\label{L:u-decay}
In \eqref{E:pf-main-3}, the $\limsup$ can be replaced by $\lim$.  Moreover, for any $0\leq \theta \leq \frac{\pi}{3}-\delta$, we have the decay on the right estimate
\begin{equation}
\label{E:pf-main-1}
\lim_{t\nearrow +\infty} \int \phi_+\left(\cos\theta(x-r) + \sin\theta \sqrt{1+y^2+z^2}\right)u^2(\mathbf{x}+\mathbf{a}(t),t) \, d\mathbf{x} \lesssim e^{-\delta r},
\end{equation}
and the decay on the left estimate
\begin{equation}
\label{E:pf-main-5}
\lim_{t\nearrow +\infty} \int [\phi_+(x+\frac{19t}{20}) - \phi_+(x+r)] u^2(\mathbf{x}+\mathbf{a}(t),t)] \, d\mathbf{x} \lesssim e^{-\delta r}.
\end{equation}
By \eqref{E:pf-main-3}, \eqref{E:pf-main-1}, and \eqref{E:pf-main-5},  for each $r>0$, for $t$ sufficiently large,
\begin{equation}
\label{E:pf-main-11}
\Big| \| u( \mathbf{x} + \mathbf{a}(t),t) \|_{L_{\mathbf{x}}^2(|\mathbf{x}|\leq r)}^2 - c_*^{-1} \|Q\|_{L_{\mathbf{x}}^2}^2 \Big| \lesssim e^{-\delta r}.
\end{equation}
\end{lemma}

\begin{proof}
Apply \eqref{E:Ip-right} in Lemma \ref{L:Ipm-estimates} with $0\leq \theta\leq \frac{\pi}{3}-\delta$, $\lambda = \frac12$, $t_0=t$, $t_{-1}=0$, and any $r>0$, to obtain
\begin{align*}
\indentalign \int \phi_+\left(\cos\theta(x-r)+\sin\theta \sqrt{1+y^2+z^2}\right)u^2(\mathbf{x}+\mathbf{a}(t),t) \, d\mathbf{x} \\
&\leq \int \phi_+\left(\cos\theta(x-r-\frac12t)+\sin\theta \sqrt{1+y^2+z^2}\right) u^2(\mathbf{x}+\mathbf{a}(0),0) \,d \mathbf{x} + Ce^{-\delta r}.
\end{align*}
As $t \nearrow +\infty$, the integral on the right-side goes to $0$, since $u(0)$ is a fixed function and the effective support window $x>r+\frac12 t - \tan\theta \sqrt{1+y^2+z^2}$ moves outside of any compact set.  Thus, we obtain the decay on the right estimate \eqref{E:pf-main-1}.

Now we begin the left-side estimates.  Suppose that $t\geq t'>0$.  Apply \eqref{E:Ip-left} in Lemma \ref{L:Ipm-estimates} with $\theta=0$, $\lambda = \frac{19}{20}$, $t_1=t$, $t_0=t'$, $r=\frac{19}{20}t'$ to get
\begin{equation}
\label{E:pf-main-2}
\int \phi_+(x +\frac{19}{20}t) u^2(\mathbf{x}+\mathbf{a}(t),t) \,d \mathbf{x} \leq \int \phi_+(x+\frac{19}{20}t') u^2(\mathbf{x}+\mathbf{a}(t'),t') \, d\mathbf{x} + e^{-19\delta t'/20}.
\end{equation}
Consequently, 
\begin{equation}
\label{E:pf-main-4}
\ell \defeq \frac{1}{\|Q\|_{L_{\mathbf{x}}^2}^2} \lim_{t\to +\infty} \int \phi_+(x +\frac{19}{20}t) u^2(\mathbf{x}+\mathbf{a}(t),t) \,d \mathbf{x}
\end{equation}
exists.  To prove this, take, for the moment
$$
\ell(t) \defeq \frac{1}{\|Q\|_{L_{\mathbf{x}}^2}^2} \int \phi_+(x +\frac{19}{20}t) u^2(\mathbf{x}+\mathbf{a}(t),t) \,d \mathbf{x},
$$
$$ 
\ell_- \defeq \liminf_{t\to \infty} \ell(t)\,, \quad \ell_+\defeq\limsup_{t\to \infty} \ell(t).
$$ 
We will show that $\ell_+=\ell_-$.  Construct two sequences $t_m'$ and $t_m$ as follows:
\begin{itemize}
\item 
select $t_1'$ so that $t_1'>1$ and $|\ell(t_1') - \ell_-| \leq 2^{-1}$,
\item 
select $t_1$ so that $t_1>t_1'$ and $|\ell(t_1) - \ell_+| \leq 2^{-1}$,
\item 
select $t_2'$ so that $t_2'>2$ and $|\ell(t_2') - \ell_-| \leq 2^{-2}$,
\item 
select $t_2$ so that $t_2>t_2'$ and $|\ell(t_2) - \ell_+| \leq 2^{-2}$,
\item 
etc.
\end{itemize}
Then $t_m' \nearrow +\infty$, and for all $m$, $t_m>t_m'$, and moreover,
$$
\ell_- = \lim_{m\to \infty} \ell(t_m') \,, \qquad \ell_+ = \lim_{m\to \infty} \ell(t_m).
$$
By \eqref{E:pf-main-2}, we have
$$
\ell(t_m) \leq \ell(t_m') + e^{-19\delta t_m'/20}.
$$
Sending $m\to \infty$, we obtain $\ell_+ \leq \ell_-$, completing the proof that $\ell$ exists.  
\begin{figure}[ht]
\includegraphics[scale=0.71]{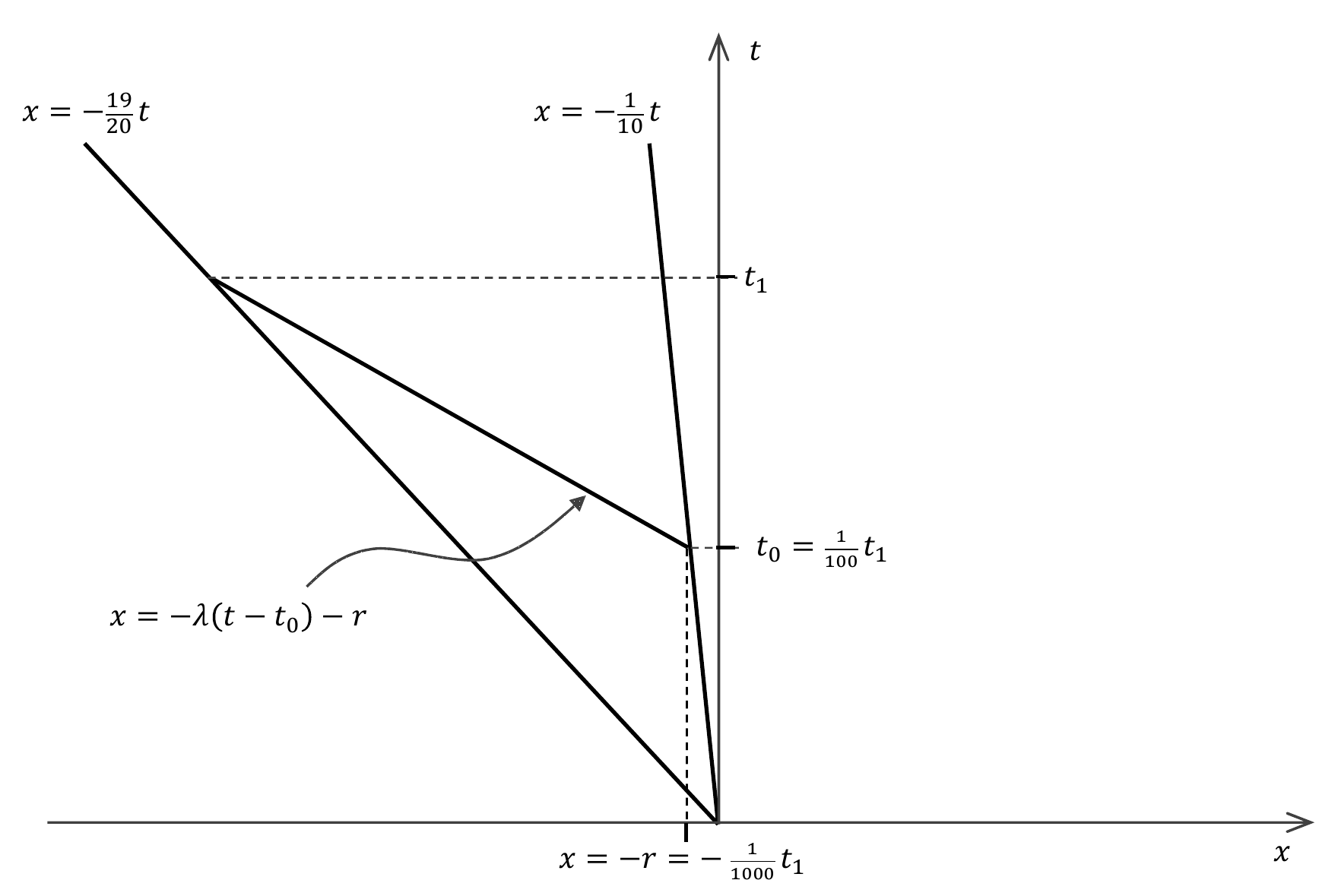}
\caption{\label{F:mon-2}Take $0<t_0 = \frac{t_1}{100} < t_1$.  To link $x=-\frac{19 t}{20}$ at $t=t_1$ to $x=-\frac{t}{10}$ at $t=t_0$, we follow the line $x=-\lambda(t-t_0)-r$ with $r=-\frac{t_0}{10}$.  Solving, yields $\lambda = \frac{100}{99}(\frac{19}{20}-\frac{1}{1000})<1$.  }
\end{figure}

Next, we claim that in fact $\ell = c_*^{-1}$.  For this, see Figure \ref{F:mon-2}.  Take
$$
0< t_0 = \frac{t_1}{100} < t_1.
$$
Apply  \eqref{E:Ip-left} in Lemma \ref{L:Ipm-estimates} with $\theta=0$, $\lambda = \frac{100}{99}(\frac{19}{20}-\frac{1}{1000})$,  $r=\frac{1}{10}t_0$ to obtain
$$
\int \phi_+( x + \frac{19}{20} t_1) u^2(\mathbf{x}+\mathbf{a}(t_1),t_1) \,d \mathbf{x} \leq
\int \phi_+(x + \frac{1}{10}t_0) u^2(\mathbf{x}+ \mathbf{a}(t_0),t_0) \, d\mathbf{x} + Ce^{-\delta t_0/10}.$$
Sending $t_0\nearrow +\infty$ along a sequence that achieves the $\liminf$ (since $t_1=100t_0$, $t_1\nearrow +\infty$), we obtain
$$
\ell \leq \liminf_{t\nearrow +\infty} \int \phi_+(x + \frac{1}{10}t_0) u^2(\mathbf{x}+ \mathbf{a}(t_0),t_0) \, d\mathbf{x}.
$$ 
On the other hand, noting that for all $x$ and all $t>0$,  $\phi_+(x+\frac{1}{10}t) \leq \phi_+(x+\frac{19}{20}t)$, it is straightforward from the definitions that 
\begin{align*}
\frac{1}{c_*} &= \frac{1}{\|Q\|_{L_{\mathbf{x}}^2}^2} \limsup_{t\nearrow +\infty} \int \phi_+(x  + \frac1{10}t) u^2(\mathbf{x}+ \mathbf{a}(t),t) \, d\mathbf{x} \\
&\leq \frac{1}{\|Q\|_{L_{\mathbf{x}}^2}^2} \limsup_{t\nearrow +\infty} \int \phi_+(x +\frac{19}{20}t) u^2(\mathbf{x}+\mathbf{a}(t),t) \,d \mathbf{x} = \ell.
\end{align*}
Hence, $\ell = \frac{1}{c_*}$, and the $\limsup$ in the definition \eqref{E:pf-main-3} can be replaced by $\lim$.   Taking the difference between \eqref{E:pf-main-4} and \eqref{E:pf-main-3}, using that $\ell=\frac{1}{c_*}$, we obtain
\begin{equation}
\label{E:pf-main-6}
0= \lim_{t\nearrow +\infty} \int [\phi_+(x+\frac{19}{20}t)-\phi_+(x  + \frac1{10}t)] u^2(\mathbf{x}+ \mathbf{a}(t),t) \, d\mathbf{x}.
\end{equation}
Now, apply \eqref{E:Ip-left} in Lemma \ref{L:Ipm-estimates} with $\theta=0$, $\lambda = \frac12$, and any $r>0$, for
$$ 
0<t_0= \frac{4}{5}t_1+2r<t_1
$$
to obtain
$$
\int \phi_+(x+\frac{t_1}{10}) u^2(\mathbf{x}+\mathbf{a}(t_1),t_1) \,d \mathbf{x} \leq \int \phi_+(x+r) u^2(\mathbf{x}+\mathbf{a}(t_0),t_0) \, d\mathbf{x}+Ce^{-\delta r},
$$
and hence,
$$
\lim_{t_0\nearrow +\infty} \int [\phi_+(x+\frac{t_1}{10}) u^2(\mathbf{x}+\mathbf{a}(t_1),t_1)- \phi_+(x+r) u^2(\mathbf{x}+\mathbf{a}(t_0),t_0)] \, d\mathbf{x} \lesssim e^{-\delta r}.
$$
However, given that the limit in \eqref{E:pf-main-3} exists,
$$
\lim_{t_0\nearrow +\infty} \int [\phi_+(x+\frac{t_1}{10}) u^2(\mathbf{x}+\mathbf{a}(t_1),t_1)- \phi_+(x+\frac{t_0}{10}) u^2(\mathbf{x}+\mathbf{a}(t_0),t_0)] \, d\mathbf{x} =0.
$$
Taking the difference of the above two equations, we obtain
$$
\lim_{t_0\nearrow +\infty} \int [\phi_+(x+\frac{t_0}{10}) - \phi_+(x+r)] u^2(\mathbf{x}+\mathbf{a}(t_0),t_0)] \, d\mathbf{x} \lesssim e^{-\delta r}.
$$
Making the notational changes of replacing $t_0$ by $t$ in this estimate, and adding it to \eqref{E:pf-main-6}, we obtain \eqref{E:pf-main-5}.
\end{proof}

Now, we complete the proof that Proposition \ref{P:wk-lim} and \ref{P:rigidity} imply Theorem \ref{T:main}.    First, we claim that
\begin{equation}
\label{E:pf-main-8}
\begin{aligned}
&u(\mathbf{x}+\mathbf{a}(t),t) \rightharpoonup c_*^{-2} Q(c_*^{-1}\mathbf{x}) \text{ weakly in }H_{\mathbf{x}}^1,\\
&u(\mathbf{x}+\mathbf{a}(t),t) \to c_*^{-2} Q(c_*^{-1}\mathbf{x}) \text{ strongly in }L_{\mathbf{x}}^2(|\mathbf{x}|\leq R) \text{ for any }R>0 . 
\end{aligned}
\end{equation}

Let $t_m \nearrow +\infty$ be any sequence.  By Proposition \ref{P:wk-lim}, there exists a subsequence $t_{m'}$ such that
\begin{equation}
\label{E:pf-main-9}
\begin{aligned}
&u(\mathbf{x}+ \mathbf{a}(t_{m'}),t_{m'}+t) \rightharpoonup \tilde u(x,t) \text{ weakly in }H_{\mathbf{x}}^1,\\
&u(\mathbf{x}+ \mathbf{a}(t_{m'}),t_{m'}+t) \to \tilde u(x,t) \text{ strongly in }L_{\mathbf{x}}^2(|\mathbf{x}|\leq R)\text{ for any }R>0  
\end{aligned}
\end{equation}
for every $t\in \mathbb{R}$, with $\tilde u$ satisfying the conditions of Proposition \ref{P:rigidity}.  By Proposition \ref{P:rigidity}, there exists $c_+>0$ and $\mathbf{a}_+\in \mathbb{R}^3$ such that for all $t\in \mathbb{R}$,
$$
\tilde u(x,t) = c_+^{-2}Q(c_+^{-1}(\mathbf{x}-\mathbf{a}_+-tc_+^{-2}))
$$
so that $\mathbf{a}_+= \tilde{\mathbf{a}}(0) = 0$ and $\tilde c(t)=c_+$ for all $t\in \mathbb{R}$. 
Inserting this into \eqref{E:pf-main-9} and evaluating at $t=0$, we obtain
\begin{equation}
\label{E:pf-main-10}
\begin{aligned}
&u(\mathbf{x}+\mathbf{a}(t_{m'}),t_{m'}) \rightharpoonup c_+^{-2} Q(c_+^{-1}\mathbf{x}) \text{ weakly in }H_{\mathbf{x}}^1,\\
&u(\mathbf{x}+\mathbf{a}(t_{m'}),t_{m'}) \to c_+^{-2} Q(c_+^{-1}\mathbf{x}) \text{ strongly in }L_{\mathbf{x}}^2(|\mathbf{x}|\leq R) \text{ for any }R>0, 
\end{aligned}
\end{equation}
where \emph{a priori} $c_+$ can depend on the choice of sequence $t_m$.   To complete the proof of \eqref{E:pf-main-8}, we must show that  $c_+=c_*$ as defined in \eqref{E:pf-main-3}.   The estimate \eqref{E:pf-main-11}, and the fact that $u(\mathbf{x}+\mathbf{a}(t_{m'}),t_{m'})$ converges strongly to $\tilde u(\mathbf{x},0)$ in $L^2(|\mathbf{x}|\leq r)$ yields that for every $r>0$,
$$
\Big| \| \tilde u( \mathbf{x},0) \|_{L_{\mathbf{x}}^2(|\mathbf{x}|\leq r)}^2 - c_*^{-1} \|Q\|_{L_{\mathbf{x}}^2}^2 \Big| \lesssim e^{-\delta r}.
$$
By \eqref{E:tilde-u-decay-1}, for every $r>0$,
$$
\Big| M(\tilde u) - c_*^{-1} M(Q)\Big| \lesssim e^{-\delta r},
$$
from which it follows that $M(\tilde u)= c_*^{-1}M(Q)$.   Since $\tilde u(\mathbf{x},0) = c_+^{-2}Q(c_+^{-1}\mathbf{x})$, we have $M(\tilde u) = c_+^{-1}$.  Hence, $c_+=c_*$, and \eqref{E:pf-main-8} is established.    

By \eqref{E:param-conv}, 
$$
c_* = c_+ = \tilde c(0) = \lim_{m'\to \infty} c(t_{m'}).
$$
Since this limit is independent of the choice of sequence $t_m$, we conclude $c(t) \to c_*$ as $t\to \infty$.  

Next we remark on how this implies the strong convergence \eqref{E:window-conv} asserted in Theorem \ref{T:main}.  We explain this in the reference frame of Lemma \ref{L:u-decay}, where $\mathbf{x}=0$ corresponds to the soliton center.  Thus, we are looking to show that we have $L_{\mathbf{x}}^2$ strong convergence in the conic region 
\begin{equation}
\label{E:pure-wedge}
x>-\frac{9}{10}t - \tan \theta \sqrt{1+y^2+z^2} \qquad \text{(pure wedge)},
\end{equation}
where $\theta< \frac{\pi}{3}-\delta$.  The local convergence \eqref{E:pf-main-8} implies the convergence in a compact neighborhood of $0$.  Taking $\tilde\theta$ such that $\theta < \tilde \theta < \frac{\pi}{3}-\delta$, then for $t$ sufficiently large, the region
\begin{equation}
\label{E:cut-wedge}
\left\{
\begin{aligned}
&x> r- \tan \tilde{\theta} \sqrt{1+y^2+z^2}  \qquad \text{(cut wedge)}
\\
& x> - \frac{19}{20}t
\end{aligned}
\right.
\end{equation}
fits inside the region \eqref{E:pure-wedge}, as depicted in Figure \ref{F:cut-cone}.  Since \eqref{E:pf-main-1} (with $\theta$ replaced by $\tilde \theta$) and \eqref{E:pf-main-5} imply the convergence in \eqref{E:cut-wedge} away from $\mathbf{x}=0$, the convergence also holds in \eqref{E:pure-wedge} away from $\mathbf{x}=0$.    This completes the proof of Theorem \ref{T:main}.

\begin{figure}
\includegraphics[scale=0.65]{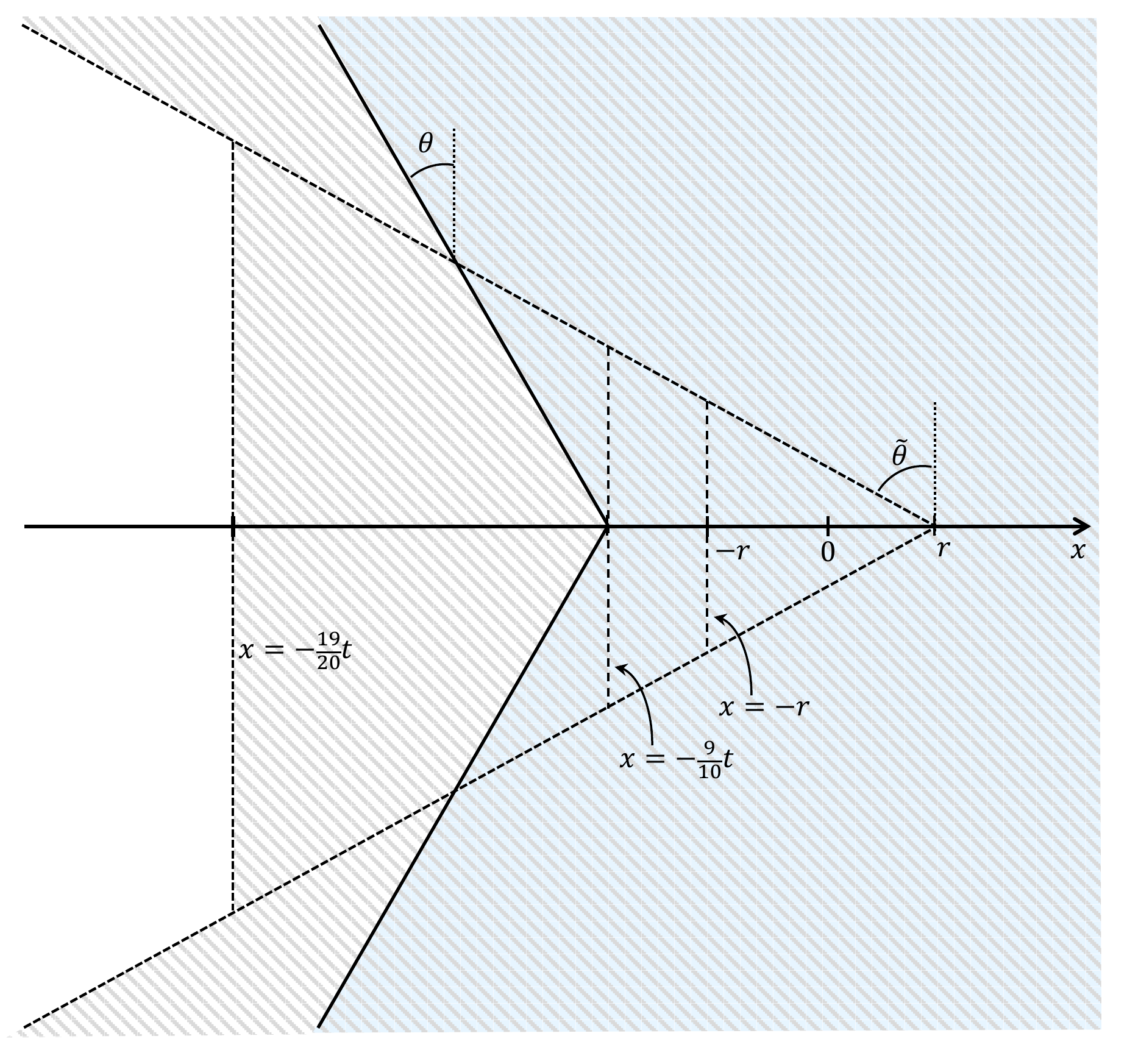}
\caption{The pure wedge \eqref{E:pure-wedge} and cut wedge \eqref{E:cut-wedge} regions.}
\label{F:cut-cone}
\end{figure}

\section{Construction of the weak time limit Class B solution $\tilde u$}
\label{S:wk-lim}

In this section, we prove Lemma \ref{L:soft-step}.  The entire contents of Lemma \ref{L:soft-step} follow from the combination of Lemmas \ref{L:wk-1}, \ref{L:wk-2}, \ref{L:wk-3}, \ref{L:wk-5}, \ref{L:wk-6} stated and proved below.

\begin{lemma}[rational time shifts]
\label{L:wk-1}
Given $t_m\nearrow +\infty$, there exists a subsequence $t_{m'}$ such that 
\begin{enumerate}
\item 
for each $t\in \mathbb{Q}$, $u(\mathbf{x}+\mathbf{a}(t_{m'}),t+t_{m'})$ converges weakly in $H^1_{\mathbf{x}}$ as $m'\to\infty$,
\item 
for each $t\in \mathbb{Q}$, $\partial_t u (\mathbf{x}+\mathbf{a}(t_{m'}),t+t_{m'})$ converges weakly in $H^{-2}_{\mathbf{x}}$ as $m'\to \infty$,
\item 
for each $t\in \mathbb{Q}$, $\mathbf{a}(t_{m'}+t) - \mathbf{a}(t_{m'})$ converges (in $\mathbb{R}^3$) as $m'\to \infty$,
\item 
for each $t\in \mathbb{Q}$, $c(t_{m'}+t)$ converges as $m'\to \infty$.
\end{enumerate}
\end{lemma}

\begin{proof}
By \eqref{E:param-ODEs} in Lemma \ref{L:ODE-bounds}, we have that
$$
|\mathbf{a}(t_{m}+t) - \mathbf{a}(t_{m})| \lesssim \alpha |t|
$$
uniformly in $m$.  Also, mass conservation (Lemma \ref{L:mass-conservation}) and the definition of orbital stability (Definition \ref{D:orb-stab}) yield
$$
| c(t_{m}) -1 | \lesssim \alpha.
$$
These bounds and a diagonal argument, using that $\mathbb{Q}$ is countable, imply that there is a subsequence such that items (3) and (4) hold.  By passing to a further subsequence, (1) and (2) follow from the Banach--Alaoglu theorem, and a diagonal argument using that $\mathbb{Q}$ is countable.   Thus, there is a single subsequence, denoted $m'$ for which all properties (1)-(4) hold.
\end{proof}

\begin{lemma}[uniform continuity for frequency projected solution]
\label{L:wk-2}
Given dyadic $M\geq 1$, we have that for all $m'$
\begin{equation}
\label{E:wk-106}
\| P_{\leq M} u( t+t_{m'}) - P_{\leq M} u(t'+t_{m'}) \|_{L^2_{\mathbf{x}}} \lesssim \min(M^2 |t-t'|, M^{-1}).
\end{equation}
Consequently, for any $-2<s<1$,
\begin{equation}
\label{E:wk-106b}
\| u( t+t_{m'}) -  u(t'+t_{m'}) \|_{H^s_{\mathbf{x}}} \lesssim |t-t'|^{(1-s)/3},
\end{equation}
and for any $-4\leq s<-2$,
\begin{equation}
\label{E:wk-107b}
\|  \partial_t u( t+t_{m'}) - \partial_t u(t'+t_{m'}) \|_{H^s_{\mathbf{x}}} \lesssim |t-t'|^{(-s-2)/3}.
\end{equation}
\end{lemma}

\begin{proof}
The bound of $M^{-1}$ follows immediately from the bound on $\| u(t) \|_{L_t^\infty H_{\mathbf{x}}^1}$.
We have
\begin{align*}
\indentalign P_{<M} u(t_{m'}+t') - P_{<M}  u(t_{m'}+t) \\
&= P_{<M} (U(t'-t)-I) u(t_{m'}+t)  - P_{<M} \int_{s=t_{m'}+t}^{t_{m'}+t'} U(t_{m'}+t'-s) \partial_x (u^2)(s) \,ds.
\end{align*}
For the first term, we use that $P_{<M} [U(s)-I]$ is $H_{\mathbf{x}}^1 \to L_{\mathbf{x}}^2$ bounded with operator norm $\leq \min(1,|s|M^2)$.  
For the second term, estimating in $L^2_{\mathbf{x}}$ in the usual way, bounding half of the derivative to $M^{1/2}$, distributing the other half via the fractional Leibniz rule, applying Sobolev, yields a bound of $M^{1/2}|t-t'|$.  The two estimates together complete the proof of \eqref{E:wk-106}.  

Now we explain how \eqref{E:wk-106b} follows from \eqref{E:wk-106}.  Dividing frequency space into dyads,
$$
\| u( t+t_{m'}) -  u(t'+t_{m'}) \|_{H^s_{\mathbf{x}}} \lesssim \sum_{M\geq 1 \text{ dyadic}} M^s \| P_M [u( t+t_{m'}) -  u(t'+t_{m'})] \|_{L^2_{\mathbf{x}}}.
$$
Applying \eqref{E:wk-106},
$$
\| u( t+t_{m'}) -  u(t'+t_{m'}) \|_{H^s_{\mathbf{x}}} \lesssim \sum_{M\geq 1 \text{ dyadic}} M^s \min( M^2|t-t'|, M^{-1}).
$$
Since $-2<s<1$, $M^{s+2}$ is a positive power of $M$ and $M^{s-1}$ is a negative power of $M$.  For $M \leq |t-t'|^{-1/3}$, the first bound is better, and for $M\geq |t-t'|^{-1/3}$ the second bound is better.  Carrying out the sum yields \eqref{E:wk-106b}.  

Next, we deduce \eqref{E:wk-107b} as a consequence of \eqref{E:wk-106b}.  Writing $u_2 = u(t+t_{m'})$ and $u_1 = u(t'+t_{m'})$, we use the 3D ZK equation 
$$
\partial_t u = -\partial_x \Delta u - \partial_x(u^2)
$$
for $u=u_2$ and $u=u_1$ to obtain
$$
\partial_t (u_2-u_1) = - \partial_x\Delta(u_2-u_1) - \partial_x[ (u_2-u_1)(u_2+u_1) ],
$$
from which it follows that
$$
\| \partial_t (u_2-u_1) \|_{H_{\mathbf{x}}^s} \lesssim \|u_2 -u_1 \|_{H_{\mathbf{x}}^{s+3}} + \|(u_2-u_1)(u_2+u_1) \|_{H_{\mathbf{x}}^{s+1}}.
$$
Then apply the inequality, for $-\infty<\alpha\leq \frac12$,
\begin{equation}
\label{E:prod-Sob}
\| fg \|_{H_{\mathbf{x}}^\alpha} \lesssim  \|f \|_{H_{\mathbf{x}}^{\max(\alpha+\frac12,-1)}}  \|g \|_{H_{\mathbf{x}}^1}
\end{equation}
to obtain, if $-4\leq s< -2$,
$$
\| \partial_t (u_2-u_1) \|_{H_{\mathbf{x}}^s }\lesssim \|u_2 - u_1 \|_{H_{\mathbf{x}}^{s+3}} \lesssim |t-t'|^{(-2-s)/3}.
$$
For $s\leq -4$, it seems, we cannot improve on the estimate $|t-t'|^{2/3}$, since the right side of \eqref{E:prod-Sob} cannot be improved if $\alpha<-\frac32$.  
\end{proof}

\begin{lemma}[density and convergence]
\label{L:wk-3} \quad
\begin{enumerate}
\item 
For all $t\in \mathbb{R}$, $u(\mathbf{x}+\mathbf{a}(t_{m'}),t+t_{m'})$ converges weakly in $H_{\mathbf{x}}^1$ as $m'\to\infty$ and $\partial_t u(\mathbf{x}+\mathbf{a}(t_{m'}),t+t_{m'})$ converges weakly in $H_{\mathbf{x}}^{-2}$ as $m'\to\infty$.   
\item 
\emph{Define}, for all $t\in \mathbb{R}$,
$$
\tilde u(t) = \operatorname{wk-lim}_{m'\to \infty} u(\mathbf{x}+\mathbf{a}(t_{m'}),t+t_{m'}),
$$
$$
\tilde v(t) = \operatorname{wk-lim}_{m'\to \infty} \partial_t u(\mathbf{x}+\mathbf{a}(t_{m'}),t+t_{m'}),
$$
where the first is a weak limit in $H^1_{\mathbf{x}}$ and the second is a weak limit in $H^{-2}_{\mathbf{x}}$.  Then we have, for every $t\in \mathbb{R}$, that $\partial_t \tilde u = \tilde v$, and $\tilde u$ is uniformly-in-time bounded in $H^1_{\mathbf{x}}$ and $\partial_t \tilde u$ is uniformly-in-time bounded in $H^{-2}_{\mathbf{x}}$.
\item 
For every $T>0$ and all $s<1$, $\tilde u \in C([-T,T]; H^s_{\mathbf{x}})$ and $\partial_t \tilde u \in C([-T,T]; H^{s-2}_{\mathbf{x}})$.
\item 
For every $T>0$ and $R>0$, $u(\mathbf{x}+\mathbf{a}(t_{m'}),t+t_{m'})\mathbf{1}_{<R}(x)$ converges to $\tilde u(\mathbf{x},t)\mathbf{1}_{<R}(\mathbf{x})$ strongly in $C([-T,T]; L_{\mathbf{x}}^2)$.
\item 
For all $t\in \mathbb{R}$, $\mathbf{a}(t_{m'}+t)-\mathbf{a}(t_{m'})$ converges.  The limit, that we denote by $\tilde{\mathbf{a}}(t)$, is Lipschitz continuous.
\item 
For all $t\in \mathbb{R}$, $c(t_{m'}+t)$ converges.  The limit, that we denote by $\tilde c(t)$, is Lipschitz continuous.
\end{enumerate}
\end{lemma}

\begin{proof}
(1) Let $t\in \mathbb{R}\backslash \mathbb{Q}$ and let $\phi \in H^{-1}_{\mathbf{x}}$ be a test function. We must show that  $\la u(\bullet+\mathbf{a}(t_{m'}), t+ t_{m'}), \phi \ra_{\mathbf{x}}$ is a Cauchy sequence (of numbers).  Let $\epsilon>0$.  Since $u(t+t_{m'})$ is bounded in $H^1_{\mathbf{x}}$ (uniformly in $m'$), there exists dyadic $M>0$ sufficiently large so that 
\begin{equation}
\label{E:wk-105}
| \la u(\bullet+\mathbf{a}(t_{m'}),t+t_{m'}), P_{>M} \phi \ra | \leq ( \sup_{t\in \mathbb{R}} \|u(t) \|_{H_{\mathbf{x}}^1} ) \| P_{>M} \phi \|_{H^{-1}_{\mathbf{x}}} \leq \epsilon.
\end{equation}
It suffices to find $m_0'$ so that for any $m_1', m_2' \geq m_0'$ chosen from the $m'$ sequence, we have
\begin{equation}
\label{E:wk-102}
| \la u(\bullet+\mathbf{a}(t_{m_1'}), t+t_{m_1'}) - u(\bullet+\mathbf{a}(t_{m_2'}), t+t_{m_2'}) , P_{<M} \phi \ra_{\mathbf{x}} | \leq 3\epsilon.
\end{equation}
Indeed, once \eqref{E:wk-102} is established, \eqref{E:wk-105} and \eqref{E:wk-102} combined give  that for any $m_1', m_2' \geq m_0'$ chosen from the $m'$ sequence,
$$
| \la u(\bullet+\mathbf{a}(t_{m_1'}), t+t_{m_1'}) - u(\bullet+\mathbf{a}(t_{m_2'}),t+t_{m_2'}) , \phi \ra_{\mathbf{x}} | \leq 5\epsilon,
$$
completing the proof.  To establish \eqref{E:wk-102}, first note that the frequency restriction transfers to $u$, i.e.,
$$ 
\la u(\bullet+\mathbf{a}(t_{m_1'}), t+t_{m_1'}) - u(\bullet+\mathbf{a}(t_{m_2'}), t+t_{m_2'}) , P_{<M} \phi \ra_{\mathbf{x}}  
$$
$$
= \la  P_{<2M} u(\bullet+\mathbf{a}(t_{m_1'}), t+t_{m_1'}) - P_{<2M} u(\bullet+\mathbf{a}(t_{m_2'}), t+t_{m_2'}) , P_{<M} \phi \ra_{\mathbf{x}},
$$
and thus, we can apply Lemma \ref{L:wk-2} to obtain that for any $t'$, and either $j=1$ or $j=2$,
$$
|\la u(\bullet+\mathbf{a}(t_{m_j'}), t+t_{m_j'}) -  u(\bullet+\mathbf{a}(t_{m_j'}),t'+t_{m_j'}), P_{<M} \phi \ra | \lesssim M^{1/2} |t-t'|.
$$
We just chose $t'\in \mathbb{Q}$  so that $M^{1/2}|t-t'| \lesssim  \epsilon$ to obtain 
\begin{equation}
\label{E:wk-103}
|\la u(\bullet+\mathbf{a}(t_{m_j'}),t+t_{m_j'}) -  u(\bullet+\mathbf{a}(t_{m_j'}),t'+t_{m_j'}), P_{<M} \phi \ra | \leq \epsilon.
\end{equation}
By Lemma \ref{L:wk-1}, since $t'\in \mathbb{Q}$, there exists $m_0'$ so that for any $m_1', m_2' \geq m_0'$ chosen from the $m'$ sequence, we have
\begin{equation}
\label{E:wk-104}
| \la u(\bullet+\mathbf{a}(t_{m_1'}),t'+t_{m_1'}) - u(\bullet+\mathbf{a}(t_{m_2'}),t'+t_{m_2'}) , P_{<M} \phi \ra_{\mathbf{x}} | \leq \epsilon.
\end{equation}
Combining \eqref{E:wk-103} (for both $j=1$ and $j=2$) and \eqref{E:wk-104} gives \eqref{E:wk-102}.  This completes the proof that  $u(\mathbf{x}+\mathbf{a}(t_{m'}),t+t_{m'})$ converges weakly in $H_{\mathbf{x}}^1$ as $m'\to\infty$.  

The fact that for all $t\in \mathbb{R}$,  $\partial_t u(\mathbf{x}+\mathbf{a}(t_{m'}),t+t_{m'})$ converges weakly in $H_{\mathbf{x}}^{-2}$ as $m'\to\infty$ follows similarly, using \eqref{E:wk-107b} in place of \eqref{E:wk-106}.

(2) Now we can, as in the lemma statement, define $\tilde u$ and $\tilde v$.  Our objective is to show that in fact $\partial_t \tilde u = \tilde v$, where $\partial_t$ is defined for functions of $t$ taking values in $H_{\mathbf{x}}^{-2}$.  Now for fixed test function $\phi(\mathbf{x})$,
\begin{align*}
\indentalign \la u(\mathbf{x}+\mathbf{a}(t_{m'}), t+t_{m'}), \phi(\mathbf{x}) \ra - \la u(\mathbf{x}+\mathbf{a}(t_{m'}),t_0+t_{m'}), \phi(\mathbf{x})\ra \\
&= \int_{s=t_0}^t \la \partial_s u(\mathbf{x}+\mathbf{a}(t_{m'}), s+t_{m'}), \phi(\mathbf{x}) \ra \,ds.
\end{align*}
Send $m'\to \infty$, which gives by dominated convergence
$$
\la \tilde u(\mathbf{x}, t), \phi(\mathbf{x}) \ra - \la \tilde u(\mathbf{x},t_0), \phi(\mathbf{x})\ra = \int_{s=t_0}^t \la \tilde v(\mathbf{x}, s), \phi(\mathbf{x}) \ra \,ds.
$$
Taking $\partial_t$ we obtain
$$
\la \partial_t \tilde u(\mathbf{x}, t), \phi(\mathbf{x}) \ra = \la \tilde v(\mathbf{x}, s), \phi(\mathbf{x}) \ra .
$$
Since this holds for arbitrary $\phi$, we conclude $\partial_t \tilde u = \tilde v$.

(3) For the continuity claim for $\tilde u$, we note that by a standard property of weak limits
$$
\| \tilde u(t) - \tilde u(t') \|_{H_{\mathbf{x}}^s} \leq \liminf_{m'\to+\infty} \| u( \bullet+\mathbf{a}(t_{m'}), t+t_{m'}) - u( \bullet+\mathbf{a}(t_{m'}), t'+t_{m'}) \|_{H_{\mathbf{x}}^s},
$$
and thus, by \eqref{E:wk-106b} in Lemma \ref{L:wk-2}, we have
\begin{equation}
\label{E:wk-108}
\| \tilde u(t) - \tilde u(t') \|_{H_{\mathbf{x}}^s}  \lesssim |t-t'|^{\frac23(1-s)}.
\end{equation}
Similarly, one can argue for the claimed continuity of $\partial_t \tilde u$ appealing to \eqref{E:wk-107b} in Lemma \ref{L:wk-2}.

(4) Fix $T>0$ and $R>0$, and we aim to establish the claimed uniform-in-time convergence. Let $\epsilon>0$.   Let $S\subset [-T,T]$ be a \emph{finite} set of time points, so that any point of $[-T,T]$ is less than $\sim \epsilon^{3/2}$ from a point in $S$.    Since $u(\bullet+\mathbf{a}(t_{m'}),t+t_{m'}) \rightharpoonup \tilde u(\bullet, t)$ in $H^1$, by the Rellich-Kondrachov compactness theorem, for each $t_j\in S$, there exists $m'_j$ such that $m' \geq m'_j$ implies 
$$
\| u(\bullet+\mathbf{a}(t_{m'}),t_j+t_{m'}) - \tilde u(\bullet, t_j)\|_{L_{|\mathbf{x}|\leq R}^2} \leq \tfrac12\epsilon.
$$
By taking $m'_0$ to be the maximum over all $m'_j$ as $t_j$ ranges over the finite set $S$, we obtain that for any $m' \geq m'_0$ and any $t'\in S$,
\begin{equation}
\label{E:wk-109}
\| u(\bullet+\mathbf{a}(t_{m'}),t'+t_{m'}) - \tilde u(\bullet, t')\|_{L_{|\mathbf{x}|\leq R}^2} \leq \tfrac12\epsilon.
\end{equation}
Now for any $t\in [-T,T]$, take $t'\in S$ such that $|t-t'|\lesssim \epsilon^4$.  Note that
\begin{align*}
\indentalign \| u(\bullet+\mathbf{a}(t_{m'}),t+t_{m'}) - \tilde u(\bullet, t)\|_{L_{|\mathbf{x}|\leq R}^2} \\
& \lesssim \| u(\bullet+\mathbf{a}(t_{m'}),t+t_{m'}) -u(\bullet+\mathbf{a}(t_{m'}),t'+t_{m'})\|_{L_{\mathbf{x}}^2} \\
&\qquad +\| u(\bullet+\mathbf{a}(t_{m'}),t'+t_{m'}) - \tilde u(\bullet, t')\|_{L_{|\mathbf{x}|\leq R}^2} + \| \tilde u(t')-\tilde u(t) \|_{L_{\mathbf{x}}^2}. 
\end{align*}
By \eqref{E:wk-106b} for $s=0$, \eqref{E:wk-109}, and \eqref{E:wk-108},
$$
\| u(\bullet+\mathbf{a}(t_{m'}),t+t_{m'}) - \tilde u(\bullet, t)\|_{L_{|\mathbf{x}|\leq R}^2} \leq \epsilon$$
for $m'\geq m'_0$.

(5)-(6) By \eqref{E:param-ODEs} in Lemma \ref{L:ODE-bounds}, for any $t,t'\in \mathbb{R}$,
\begin{equation}
\label{E:par-1}
\begin{aligned}
& |c(t_{m'}+t) - c(t_{m'}+t')| \lesssim |t-t'|,\\
& |\mathbf{a}(t_{m'}+t) - \mathbf{a}(t_{m'}+t') | \lesssim |t-t'|
\end{aligned}
\end{equation}
independently of $m'$. In Lemma \ref{L:wk-2} (3)-(4), the convergence was established for $t'\in \mathbb{Q}$.  Similar to the arguments used above, we can approximate any $t\in \mathbb{R}$ by $t'\in \mathbb{Q}$ and use the estimates \eqref{E:par-1} to deduce that $c(t_{m'}+t)$ and  $\mathbf{a}(t_{m'}+t)-\mathbf{a}(t_{m'})$ are Cauchy sequences, and thus, converge, and we can define $\tilde c(t)$ and $\tilde{\mathbf{a}}(t)$ to be their limits.    Then by \eqref{E:par-1} the Lipschitz continuity of $\tilde c(t)$ and $\tilde{\mathbf{a}}(t)$ follows.
\end{proof}

\begin{lemma}
\label{L:wk-5}
$\tilde u$ is a Class B solution to the 3D ZK.
\end{lemma}
\begin{proof}
The regularity claims in Definition \ref{D:ClassB} have been established in Lemma \ref{L:wk-3}(3).  It remains to show that
$$
\partial_t \tilde u(t) + \partial_x \Delta \tilde u(t) + \partial_x \tilde u(t)^2=0
$$
holds for each $t\in \mathbb{R}$, where each of the three terms in the equation belongs to $H_{\mathbf{x}}^{-2}$.  This will follow if we show that for each test function $\phi \in C_c^\infty(\mathbb{R}^3)$ 
$$
\la \partial_t \tilde u(t), \phi \ra + \la \partial_x \Delta \tilde u(t), \phi\ra + \la \partial_x \tilde u(t)^2, \phi\ra=0.
$$
Since $u$ is a Class B solution of the 3D ZK, we have for each $t\in \mathbb{R}$ and each $m'$,
\begin{align*}
0 &= \la (\partial_t u)(\mathbf{x}+ \mathbf{a}(t_{m'}),t+t_{m'}), \phi(\mathbf{x}) \ra + \la \partial_x \Delta u(\mathbf{x}+ \mathbf{a}(t_{m'}),t+t_{m'}), \phi(\mathbf{x}) \ra \\
& \qquad + \la \partial_x u(\mathbf{x}+ \mathbf{a}(t_{m'}),t+t_{m'})^2, \phi(\mathbf{x}) \ra.
\end{align*}
Shifting spatial derivatives to the test function in the second and third terms,
\begin{align*}
0 &= \la (\partial_t u)(\mathbf{x}+ \mathbf{a}(t_{m'}),t+t_{m'}), \phi(\mathbf{x}) \ra - \la u(\mathbf{x}+ \mathbf{a}(t_{m'}),t+t_{m'}), \partial_x \Delta \phi(\mathbf{x}) \ra \\
& \qquad - \la u(\mathbf{x}+ \mathbf{a}(t_{m'}),t+t_{m'})^2, \partial_x \phi(\mathbf{x}) \ra.
\end{align*}
Send $m'\to \infty$.  In the first term, we use that $(\partial_t u)(\bullet+ \mathbf{a}(t_{m'}),t+t_{m'}) \rightharpoonup \tilde \partial_t \tilde u(\bullet, t)$ weakly in $H_{\mathbf{x}}^{-2}$.  In the second term, we use that $u(\bullet+ \mathbf{a}(t_{m'}),t+t_{m'}) \rightharpoonup  \tilde u(\bullet, t)$ weakly in $H_{\mathbf{x}}^1$.  In the third term, we let $R>0$ sufficiently large so that $\supp \phi$ is contained in the ball of radius $R$.  Since $u(\bullet+ \mathbf{a}(t_{m'}),t+t_{m'})\mathbf{1}_{<R}(\mathbf{x}) \to \tilde u(\bullet, t)\mathbf{1}_{<R}(\mathbf{x}) $ strongly in $L_{\mathbf{x}}^2$, it follows that 
$$ 
\la u(\mathbf{x}+ \mathbf{a}(t_{m'}),t+t_{m'})^2, \partial_x \phi(\mathbf{x}) \ra \to  \la  \tilde u(\mathbf{x},t)^2, \partial_x \phi(\mathbf{x}) \ra.
$$
\end{proof}

\begin{lemma}
\label{L:wk-6}
$\tilde u$ is $\alpha$-orbitally stable and $\tilde{\mathbf{a}}(t)$ and $\tilde c(t)$, constructed above in Lemma \ref{L:wk-3}, are the modulation parameters as in Lemma \ref{L:geom-decomp}. 
\end{lemma}

\begin{proof}
From Lemma \ref{L:wk-3}, we have that for all $t\in \mathbb{R}$
$$
u(\mathbf{x}+\mathbf{a}(t_{m'}),t+t_{m'}) \rightharpoonup \tilde u(\mathbf{x},t),
$$
weakly in $H_{\mathbf{x}}^1$, and also
$$
\tilde{\mathbf{a}}(t) \defeq \lim_{m'\to \infty} [\mathbf{a}(t+t_{m'}) - \mathbf{a}(t_{m'})]\,, \qquad \tilde c(t) \defeq \lim_{m'\to \infty} c(t+t_{m'}).
$$
Hence, for all $t\in \mathbb{R}$
\begin{align*}
& c(t+t_{m'})^2 u(c(t+t_{m'})\mathbf{x}+\mathbf{a}(t+t_{m'}),t+t_{m'})\\
& \quad = c(t+t_{m'})^2u(c(t+t_{m'})\mathbf{x}+[\mathbf{a}(t+t_{m'})-\mathbf{a}(t_{m'})]+\mathbf{a}(t_{m'}),t+t_{m'}) \\
& \quad \quad \rightharpoonup \tilde c(t)^2 \tilde u(\tilde c(t) \mathbf{x}+\tilde{\mathbf{a}}(t),t)
\end{align*}
weakly in $H_{\mathbf{x}}^1$.  Consequently,
\begin{align*}
&\epsilon(\mathbf{x},t+t_{m'}) \\
&= c(t+t_{m'})^2 u(c(t+t_{m'})\mathbf{x}+\mathbf{a}(t+t_{m'}),t+t_{m'}) -Q(\mathbf{x}) \\
&\rightharpoonup \tilde c(t)^2 \tilde u(\tilde c(t) \mathbf{x}+\tilde{\mathbf{a}}(t),t)
- Q(\mathbf{x})\\
&= \tilde \epsilon(\mathbf{x},t)
\end{align*}
weakly in $H_{\mathbf{x}}^1$. 
Hence,
$$
\| \tilde\epsilon(t) \|_{H_{\mathbf{x}}^1} \leq \liminf_{m'\to \infty} \| \epsilon(t+t_{m'})\|_{H_{\mathbf{x}}^1}\leq \alpha.
$$
Thus, $\tilde u$ is $\alpha$-orbitally stable.  Moreover,
$$
\la \tilde \epsilon(t), Q^2 \ra = \lim_{m'\to \infty} \la \epsilon(t+t_{m'}), Q^2 \ra =0,
$$
$$
\la \tilde \epsilon(t), \nabla Q \ra = \lim_{m'\to \infty} \la \epsilon(t+t_{m'}), \nabla Q \ra =0,
$$
so that $\tilde{\mathbf{a}}(t)$ and $\tilde c(t)$ are the (unique) parameter values that achieve the orthogonality conditions in Lemma \ref{L:geom-decomp}.  
\end{proof}

\section{$\tilde u$ has exponential decay in space}
\label{S:app-mon}

In this section, we prove Lemma \ref{L:exp-decay} by applying the estimates \eqref{E:pf-main-1} and \eqref{E:pf-main-5} in Lemma \ref{L:u-decay}, which were obtained from the $I_+$ estimate \eqref{E:Ip-right} in Lemma \ref{L:Ipm-estimates}.

We know from Lemma \ref{L:soft-step} that
$$
\mathbf{a}(t+t_{m'}) - \mathbf{a}(t_{m'}) \to \tilde{\mathbf{a}}(t) \text{ as } m' \to \infty
$$
and
$$
u(\mathbf{x}+\mathbf{a}(t_{m'}),t_{m'}+t) \rightharpoonup \tilde u(\mathbf{x}, t) \text{ as } m'\to \infty  \text{ (weakly) in }H_{\mathbf{x}}^1.
$$
It follows that\footnote{This is the following elementary fact:  If $f_n(x)\rightharpoonup f(x)$ and $a_n\to a$, then $f_n(x+a_n) \rightharpoonup f(x+a)$.}
\begin{align*}
& u(\mathbf{x}+\mathbf{a}(t+t_{m'}) ,t_{m'}+t)  = u(\mathbf{x}+[\mathbf{a}(t+t_{m'}) - \mathbf{a}(t_{m'})] + \mathbf{a}(t_{m'}),t_{m'}+t) \\
& \qquad \rightharpoonup \tilde u(\mathbf{x}+\tilde{\mathbf{a}}(t), t) \text{ as } m'\to \infty \text{ (weakly) in }H_{\mathbf{x}}^1.
\end{align*}
Since the norm of a weak limit is less than, or equal to, the limit of the norms,
\begin{align*}
\indentalign \int \phi_+(\cos\theta(x-r) + \sin \theta \sqrt{1+y^2+z^2}) \tilde u^2(\mathbf{x}+\tilde{\mathbf{a}}(t),t) d\mathbf{x} \\
&\leq \lim_{m'\to \infty} \int \phi_+(\cos\theta(x-r) + \sin \theta \sqrt{1+y^2+z^2}) u^2(\mathbf{x}+\mathbf{a}(t+t_{m'}) ,t_{m'}+t)  \, d\mathbf{x}.
\end{align*}
By \eqref{E:pf-main-1}, we have
\begin{equation}
\label{E:tilde-right}
\int \phi_+(\cos\theta(x-r)+ \sin\theta \sqrt{1+y^2+z^2}) \tilde u^2(\mathbf{x}+\tilde{\mathbf{a}}(t),t) d\mathbf{x}  \lesssim e^{-\delta r},
\end{equation}
which yields the decay on the right estimate for $\tilde u$.  Likewise,
\begin{align*}
&\int (1- \phi_+(x+r)) \tilde u^2(\mathbf{x}+\tilde{\mathbf{a}}(t),t) d\mathbf{x} \\
& \qquad \leq \lim_{m'\to \infty} \int [\phi_+(x+\frac{19(t+t_{m'})}{20}) - \phi_+(x+r)] u^2(\mathbf{x}+\mathbf{a}(t+t_{m'}) ,t_{m'}+t)  \, d\mathbf{x}.
\end{align*}
By \eqref{E:pf-main-5}, we deduce
\begin{equation}
\label{E:tilde-left}
\int (1- \phi_+(x+r)) \tilde u^2(\mathbf{x}+\tilde{\mathbf{a}}(t),t) d\mathbf{x}  \lesssim e^{-\delta r},
\end{equation}
which yields the decay on the left estimate for $\tilde u$.  

Combining $\theta=\frac{\pi}{4}$ in \eqref{E:tilde-right} and \eqref{E:tilde-left} yields, for all $t\in \mathbb{R}$,
\begin{equation}
\label{E:tilde-u-decay-1}
\int_{|\mathbf{x}| > r} \tilde u^2(\mathbf{a}+\tilde{\mathbf{a}}(t),t) \,d \mathbf{x} \lesssim e^{-\delta r}.
\end{equation}
This completes the proof of Lemma \ref{L:exp-decay}.

\section{Higher regularity of spatially decaying Class B solutions}
\label{S:higher-regularity}

In this section, we prove Lemma \ref{L:regularity-boost}. As a reminder of notation, note that in many places in this section, $x$ appears as a weight (not $\mathbf{x}$).  Also recall that $P_N$ refers to the Littlewood-Paley multiplier, and this operator acts in all three variables.  We will use the notation
$$\ln^+N \defeq \ln(N+2)$$
for $N\geq 1$ dyadic.

We note two weighted Sobolev interpolation inequalities.  First, for $0<\theta\leq 1$,
\begin{equation}
\label{E:HR1}
\| |x|^\alpha u \|_{L_{\mathbf{x}}^2} \leq \| |x|^{\alpha/\theta} u \|_{L_{\mathbf{x}}^2}^\theta \, \|u\|_{L_{\mathbf{x}}^2}^{1-\theta}.
\end{equation}
More generally, for $p\geq 2$ and $0< \theta \leq \frac{2}{p}$,
\begin{equation}
\label{E:HR2}
\| |x|^\alpha u \|_{L_{\mathbf{x}}^p} \leq \| |x|^{\alpha/\theta} u \|_{L_{\mathbf{x}}^2}^\theta \|u \|_{L_{\mathbf{x}}^{\tilde p}}^{1-\theta}\,, \quad \text{where } \tilde p = p \cdot \frac{(1-\theta)}{1-p \theta/2}.
\end{equation}
Note that \eqref{E:HR2} reduces to \eqref{E:HR1} when $p=2$.

The inequality \eqref{E:HR1} is proved by writing
$$ 
\| |x|^\alpha u \|_{L_{\mathbf{x}}^2}^2 = \int |x|^{2\alpha} |u|^{2\theta} \, \cdot \, |u|^{2-2\theta} \, d\mathbf{x},
$$
and then applying H\"older with dual pair $L_{\mathbf{x}}^{1/\theta}$ and $L_{\mathbf{x}}^{1/(1-\theta)}$.   Likewise \eqref{E:HR2} is proved by writing
$$ 
\| |x|^\alpha u \|_{L_{\mathbf{x}}^p}^p = \int |x|^{p\alpha} |u|^{p\theta} \, \cdot \, |u|^{p(1-\theta)} \, d\mathbf{x},
$$
and then applying H\"older with dual pair $L_{\mathbf{x}}^{2/p\theta}$ and $L_{\mathbf{x}}^{1/(1-p\theta/2)}$.

Second, we need the elementary fact that the commutator of $x$ and $P_N$, 
$$xP_N - P_N x,$$
is an $L_{\mathbf{x}}^2\to L_{\mathbf{x}}^2$ bounded operator with operator norm $\lesssim N^{-1}$.   This follows since the kernel of the commutator $xP_N - P_N x$ is
$$
K(\mathbf{x},\mathbf{x}') = N^3 \check{\chi}(N(\mathbf{x}-\mathbf{x}')) (x-x').
$$
More generally, we have 
\begin{lemma}
For any $N\geq 1$ and $\alpha \geq 1$,
\begin{equation}
\label{E:HR4}
\| (\la x \ra^\alpha P_N - P_N \la x \ra^\alpha) f \|_{L_{\mathbf{x}}^2} \lesssim N^{-1} \|\la x \ra^{\alpha-1} f\|_{L_{\mathbf{x}}^2},
\end{equation}
where the implicit constant depends only on $\alpha$.
\end{lemma}
\begin{proof}
This is equivalent to stating that the operator
$$(\la x\ra^\alpha P_N \la x \ra^{-\alpha} - P_N ) \la x \ra$$ 
is $L_{\mathbf{x}}^2\to L_{\mathbf{x}}^2$ bounded with operator norm $\lesssim N^{-1}$.  To see this, note that the kernel associated to the operator is
$$
K(\mathbf{x},\mathbf{x}')=\left( \frac{\la x\ra^\alpha}{\la x' \ra^\alpha} - 1 \right) \, N^3 \, \check \chi( N(\mathbf{x}-\mathbf{x}')) \la x' \ra.
$$
We note the pointwise estimate
$$
\left| \frac{\la x\ra^\alpha}{\la x' \ra^\alpha} - 1 \right| \lesssim  \la x' \ra^{-1}|x-x'|,
$$
which is proved by considering the regions $|x-x'| \ll \la x' \ra$ and $|x-x'| \gtrsim \la x' \ra$, separately. In the first case, the bound follows by Taylor expansion, for fixed $x'$, of the function $\la x \ra^{\alpha}$ around center $x=x'$.  In the second case, it follows by bounding $\la x \ra^\alpha \leq 2^\alpha(\la x-x' \ra^\alpha + \la x' \ra^\alpha)$.

By this pointwise estimate, we have
$$
|K(\mathbf{x},\mathbf{x}')| \lesssim N^{-1} \cdot N^3 |\check\chi(N(\mathbf{x}-\mathbf{x}'))|\, N|x-x'|,
$$
and thus, the $L_{\mathbf{x}}^2\to L_{\mathbf{x}}^2$ boundedness claim follows by Young's inequality.
\end{proof}

Let us note a corollary:  For any $N \geq 1$,
\begin{equation}
\label{E:HR3}
\| \la x \ra^\alpha P_N u \|_{L_{\mathbf{x}}^2} \lesssim \| \la x \ra^\alpha u \|_{L_{\mathbf{x}}^2}.
\end{equation}
In other words, we can drop $P_N$.  To prove \eqref{E:HR3}, write
$$
\la x \ra^\alpha P_N u = (\la x \ra^\alpha P_N - P_N \la x \ra^\alpha) u + P_N \la x \ra^\alpha u.
$$
Then apply the $L^2$ norm, and use \eqref{E:HR4} and the $L_{\mathbf{x}}^2 \to L_{\mathbf{x}}^2$ boundedness of $P_N$, which concludes the proof.

\begin{lemma}
Suppose that $u$ is a Class B solution to the 3D ZK.  Then
\begin{equation}
\label{E:HR10}
\begin{aligned}
-\frac12\partial_t \int  x |P_Nu|^2 \, d\mathbf{x} 
&= \frac32 \int |\partial_x P_N u |^2 \, d\mathbf{x} + \frac12 \int |\partial_y P_N u |^2 \, d\mathbf{x}+ \frac12 \int |\partial_z P_N u |^2 \, d\mathbf{x} \\
& \qquad +  \int \, x \, P_N u \, P_N\partial_x (u^2) \, d\mathbf{x}.
\end{aligned}
\end{equation}
\end{lemma}

\begin{proof}
This is a direct calculation. Note that due to the $P_N$ operators, there is no divergent integrals issue for Class B solutions.
\end{proof}

\begin{lemma}
\label{L:L2boost}
Suppose that $u$ is a Class B solution of the 3D ZK on a time interval $I$ of length $|I|\leq 1$, then for $0<\theta<\frac14$ we have
\begin{equation}
\label{E:HR11}
\| u\|_{L_I^2H_{\mathbf{x}}^{\frac54-\theta}}^2 \lesssim \|\la x \ra^{1/\theta} u \|_{L_I^\infty L_{\mathbf{x}}^2}^\theta \la \| u \|_{L_I^\infty H_{\mathbf{x}}^1} \ra^{3-\theta}.
\end{equation}
\end{lemma}
This indicates that we can nearly achieve $H_{\mathbf{x}}^{5/4}$ regularity but averaged in time.

\begin{proof}
First, we prove that
\begin{equation}
\label{E:HR6}
\left|\int \, x \, P_N u \, P_N\partial_x (u^2) \, d\mathbf{x}\right| \lesssim N^{-\frac12(1 - 2\theta)}\| \la x\ra^{1/\theta} u \|_{L_{\mathbf{x}}^2}^\theta \| u \|_{H_{\mathbf{x}}^1}^{3-\theta}. 
\end{equation}
Applying H\"older, $L_{\mathbf{x}}^{3/2} \to L_{\mathbf{x}}^{3/2}$ boundedness of $P_N$, and Sobolev embedding
\begin{align*}
\left|\int \, x \, P_N u \, P_N\partial_x (u^2) \, d\mathbf{x}\right| &\lesssim \|xP_N u \|_{L_{\mathbf{x}}^3} \|P_N(  u_x u) \|_{L_{\mathbf{x}}^{3/2}} \\
&\lesssim \|xP_N u \|_{L_{\mathbf{x}}^3} \|u_x \|_{L_{\mathbf{x}}^2} \|u\|_{L_{\mathbf{x}}^6} \\
&\lesssim \| u \|_{H_{\mathbf{x}}^1}^2 \|xP_Nu \|_{L_{\mathbf{x}}^3}.
\end{align*}
Now we apply \eqref{E:HR2} for $0<\theta<\frac23$ and \eqref{E:HR3},
\begin{equation}
\label{E:HR7}
\left|\int \, x \, P_N u \, P_N\partial_x (u^2) \, d\mathbf{x}\right| \lesssim \| u \|_{H_{\mathbf{x}}^1}^2 \| \la x\ra^{1/\theta} u \|_{L_{\mathbf{x}}^2}^\theta \|P_N u \|_{L_{\mathbf{x}}^{\tilde p}}^{1-\theta},
\end{equation}
where in this case
$$
\tilde p = 3 \frac{1-\theta}{1-\frac32\theta}= 3(1+ \frac{\theta}{2-3\theta})=3+.
$$
Provided $0<\theta< \frac12$ so that $\tilde p <6$, we still have room to gain from Bernstein's inequality:  
\begin{equation}
\label{E:HR8}
\|P_N u \|_{L_{\mathbf{x}}^{\tilde p}} \lesssim N^s \|P_N u\|_{L_{\mathbf{x}}^2} \leq N^{-(1-s)}\|u\|_{H_{\mathbf{x}}^1},
\end{equation}
where
$$
s = \frac12(1+ \frac{\theta}{1-\theta})=\frac12+\,, \qquad 1-s = \frac12( 1- \frac{\theta}{1-\theta})=\frac12-.
$$
Plugging \eqref{E:HR8} into \eqref{E:HR7} yields the claimed estimate \eqref{E:HR6}.  

Next, we claim
\begin{equation}
\label{E:HR9}
\left| \int x\, |P_Nu|^2 \, d\mathbf{x} \right| \lesssim N^{-(2-\theta)} \|\la x \ra^{1/\theta}  u \|_{L_{\mathbf{x}}^2}^\theta \| u \|_{H_{\mathbf{x}}^1}^{2-\theta}.
\end{equation}
Note that by Cauchy-Schwarz and \eqref{E:HR1}, we have
\begin{align*}
\left| \int x |P_Nu|^2 \, d\mathbf{x} \right|&\leq \| x P_Nu \|_{L_{\mathbf{x}}^2} \|P_N u \|_{L_{\mathbf{x}}^2} \\
&\lesssim \|\la x \ra^{1/\theta} P_N u \|_{L_{\mathbf{x}}^2}^\theta \|P_N u \|_{L_{\mathbf{x}}^2}^{2-\theta} \\
&\lesssim N^{-(2-\theta)} \|\la x \ra^{1/\theta} P_N u \|_{L_{\mathbf{x}}^2}^\theta \|\nabla P_N u \|_{L_{\mathbf{x}}^2}^{2-\theta}.
\end{align*}
Then \eqref{E:HR9} follows from \eqref{E:HR3} and the $L^2\to L^2$ boundedness of $P_N$.

Now by \eqref{E:HR10}, \eqref{E:HR6} and \eqref{E:HR9}, over a time interval $I$ of length $|I|\leq 1$, 
$$
\begin{aligned}
\int_I \int_{\mathbf{x}} |\nabla P_N u|^2 \, d\mathbf{x}\, dt 
&\lesssim   N^{-\frac12(1 - 2\theta)} \| \la x\ra^{1/\theta} u \|_{L_I^\infty L_{\mathbf{x}}^2}^\theta  \| u \|_{L_I^\infty H_{\mathbf{x}}^1}^{3-\theta}\\
& \qquad +N^{-(2-\theta)} \|\la x \ra^{1/\theta} u \|_{L_I^\infty L_{\mathbf{x}}^2}^\theta \| u \|_{L_I^\infty H_{\mathbf{x}}^1}^{2-\theta}.
\end{aligned}
$$
We can now multiply this by $N^{\frac12(1-4\theta)}$ to obtain
$$
\begin{aligned}
N^{\frac12(1-4\theta)}\int_I \int_{\mathbf{x}} |\nabla P_N u|^2 \, d\mathbf{x}\, dt 
&\lesssim   N^{-\theta} \| \la x\ra^{1/\theta} u \|_{L_I^\infty L_{\mathbf{x}}^2}^\theta  \| u \|_{L_I^\infty H_{\mathbf{x}}^1}^{3-\theta}\\
& \qquad +N^{-\frac32-\theta} \|\la x \ra^{1/\theta} u \|_{L_I^\infty L_{\mathbf{x}}^2}^\theta \| u \|_{L_I^\infty H_{\mathbf{x}}^1}^{2-\theta}.
\end{aligned}
$$
By summing over $N\geq 1$, we obtain \eqref{E:HR11}.
\end{proof}

\begin{lemma}
\label{L:maximal-compare}
For any $t_0\in \mathbb{R}$, let $I=[t_0-\delta,t_0+\delta]$ for  $\delta \ll 1$.  Suppose that $u$ is a Class B solution of the 3D ZK on $I$, and for $0<\theta<\frac14$ we have
\begin{equation}
\label{E:HR27}
\| u\la x \ra^{1/\theta} \|_{L_I^\infty L_{\mathbf{x}}^2} <\infty \quad \mbox{and} \quad \|u \|_{L_I^\infty H_{\mathbf{x}}^1} < \infty
\end{equation}
so that \eqref{E:HR11} is available.
Then for each $N\geq 1$,
$$
\|P_N u(t) -P_NU(t-t_0) u(t_0) \|_{L_x^2 L_{yz I}^\infty} \lesssim \delta^{1/4} N^{-\frac18+\frac{\theta}{2}}(\ln^+ N)^5
$$
with implicit constant depending on the norms in \eqref{E:HR27}.  Consequently, by \eqref{E:HR20}
\begin{equation}
\label{E:HR28}
\| P_N u(t)  \|_{L_x^2 L_{yz I}^\infty} \lesssim (\ln^+ N)^2.
\end{equation}
\end{lemma}
\begin{proof}
By the Duhamel formula,
$$
P_N u(t) = P_NU(t-t_0) u(t_0) - \int_{t_0}^t P_N U(t-s) \partial_x u(s)^2 \, ds.
$$
By \eqref{E:HR21},
\begin{equation}
\label{E:HR22}
\|P_N u(t) - P_NU(t-t_0) u(t_0) \|_{L_x^2 L_{yz I}^\infty} \lesssim (\ln^+ N)^2 N \| P_N( u^2) \|_{L_x^1L_{yzI}^2}.
\end{equation}
Using the paraproduct decomposition
$$
P_N ( u^2) \sim P_N (P_{\lesssim N} u \, P_N u) + P_N \sum_{N'\gg N} (P_{N'}u \, P_{N'}u),
$$
we obtain
$$
\| P_N (u^2) \|_{L_x^1 L_{yzI}^2} \lesssim \| P_{\lesssim N} u \|_{L_x^2L_{yzI}^\infty} \|P_N u \|_{L_{\mathbf{x}I}^2} + \sum_{N'\gg N} \| P_{N'} u \|_{L_x^2L_{yzI}^\infty} \|P_{N'} u \|_{L_{\mathbf{x}I}^2}.
$$
For the terms on the right involving  $L_{yzI}^\infty$, we replace
$$
u(t)= (u(t) - U(t-t_0)u(t_0))+U(t-t_0)u(t_0)
$$
and obtain the estimate
\begin{equation}
\label{E:HR23}
\begin{aligned}
\indentalign (\ln^+ N)^2 N \| P_N (u^2) \|_{L_x^1 L_{yzI}^2} \\
&\lesssim   (\ln^+ N)^2 N\|P_{\lesssim N} (u(t)-U(t-t_0)u(t_0)) \|_{L_x^2 L_{yz I}^\infty} \|P_N u \|_{L_{\mathbf{x}I}^2} \\
&\qquad + (\ln^+ N)^2 N\sum_{N' \gg N} \|P_{N'}(u(t)-U(t-t_0)u(t_0)) \|_{L_x^2 L_{yz I}^\infty}  \|P_{N'} u \|_{L_{\mathbf{x}I}^2}\\
&\qquad +  (\ln^+ N)^2 N\|P_{\lesssim N} U(t-t_0)u(t_0) \|_{L_x^2 L_{yz I}^\infty} \|P_N u \|_{L_{\mathbf{x}I}^2} \\
&\qquad + (\ln^+ N)^2 N\sum_{N' \gg N} \|P_{N'}U(t-t_0)u(t_0) \|_{L_x^2 L_{yz I}^\infty}  \|P_{N'} u \|_{L_{\mathbf{x}I}^2}.
\end{aligned}
\end{equation}
For the last two terms, we use that \eqref{E:HR20} implies
\begin{equation}
\label{E:HR24}
\begin{aligned}
&\|P_{\lesssim N} U(t-t_0)u(t_0) \|_{L_x^2 L_{yz I}^\infty} \lesssim (\ln^+ N)^3 \|u(t_0)\|_{H_{\mathbf{x}}^1}, \\
& \|P_{N'}U(t-t_0)u(t_0) \|_{L_x^2 L_{yz I}^\infty}\lesssim (\ln^+ N')^2 \|u(t_0)\|_{H_{\mathbf{x}}^1}.
\end{aligned}
\end{equation}
By \eqref{E:HR11} in Lemma \ref{L:L2boost}, 
\begin{equation}
\label{E:HR25}
\begin{aligned}
N \, \|P_N u \|_{L_{\mathbf{x}I}^2} &\leq \min\left( \delta^{1/2} \|u\|_{L_I^\infty H_{\mathbf{x}}^1}, N^{-\frac14+\theta} \|P_N u \|_{L_I^2H_{\mathbf{x}}^{\frac54-\theta}} \right) \\
&\lesssim \min(\delta^{1/2}, N^{-\frac14+\theta}) \lesssim \delta^{1/4} N^{-1/8}.
\end{aligned}
\end{equation}
 
Let
$$
\gamma(N) = \|P_Nu(t)-P_NU(t-t_0)u(t_0)\|_{L_x^2L_{yz I}^\infty}.
$$
Plugging \eqref{E:HR23}, \eqref{E:HR24}, and \eqref{E:HR25} into the right side of \eqref{E:HR22}, we obtain
\begin{equation}
\label{E:HR26}
\begin{aligned}
\gamma(N) & \lesssim \delta^{1/4} N^{-1/8} (\ln^+ N)^2 \sum_{N' \lesssim N}\gamma(N') + \delta^{1/4} (\ln^+ N)^2 \sum_{N' \gg N} (N')^{-1/8} \gamma(N') \\
& \qquad + \delta^{1/4} N^{-1/8} (\ln^+ N)^5.
\end{aligned}
\end{equation}
Let 
$$\Gamma(N) = \sum_{N' \lesssim N} \gamma(N').$$
If $N'' \lesssim N$, then
\begin{align*}
\sum_{N' \gg N''} (N')^{-1/8} \gamma(N') 
&\leq \sum_{N' \gg N} (N')^{-1/8} \gamma(N') + \sum_{N'' \ll N' \lesssim N} (N')^{-1/8} \gamma(N') \\
&\leq \sum_{N' \gg N} (N')^{-1/8} \Gamma(N') + (N'')^{-1/8} \Gamma(N).
\end{align*}
Hence, if $N'' \lesssim N$, then
\begin{align*}
\gamma(N'') &\lesssim \delta^{1/4} (\ln^+ N'')^2 (N'')^{-1/8} \Gamma(N) + \delta^{1/4} (\ln^+ N'')^2 \sum_{N'\gg N} (N')^{-1/8} \Gamma(N') \\
& \qquad + \delta^{1/4} (\ln^+ N'')^5 (N'')^{-1/8}.
\end{align*}
Summing in $N''$ from $1$ to $N$,
$$
\Gamma(N) \lesssim \delta^{1/4} \Gamma(N) + \delta^{1/4} (\ln^+ N)^3\sum_{N'\gg N} (N')^{-1/8} \Gamma(N') + \delta^{1/4}.
$$
For $\delta$ sufficiently small,
$$
\Gamma(N) \lesssim  \delta^{1/4}(\ln^+ N)^3 \sum_{N'\gg N} (N')^{-1/8} \Gamma(N') + \delta^{1/4}.
$$
Therefore, for any $N'' \geq N$,
$$
\Gamma(N'') \lesssim  \delta^{1/4}(\ln^+ N'')^3 \sum_{N'\gg N} (N')^{-1/8} \Gamma(N') + \delta^{1/4}.
$$
Multiply by $(N'')^{-1/8}$ and sum over $N'' \gg N$ to obtain 
\begin{align*}
\sum_{N'' \gg N} (N'')^{-1/8} \Gamma(N'') &\lesssim \delta^{1/4}\sum_{N''\gg N} (\ln^+ N'')^3 (N'')^{-1/8} \sum_{N'\gg N} (N')^{-1/8} \Gamma(N') \\
&\qquad + \delta^{1/4} \sum_{N'' \gg N} (N'')^{-1/8}.
\end{align*}
From this, we obtain (that for $\delta$ sufficiently small)
$$
\sum_{N'\gg N} (N')^{-1/8} \Gamma(N') \lesssim \delta^{1/4}N^{-1/8}.
$$
Thus, for all $N$, 
$$
\Gamma(N) \lesssim 1.
$$
Returning to \eqref{E:HR26}, we obtain
$$
\gamma(N) \lesssim \delta^{1/4} N^{-1/8} (\ln^+ N)^5.
$$
\end{proof}

\begin{lemma}
\label{L:reg-boost-last}
For any $t_0\in \mathbb{R}$, let $I=[t_0-\delta,t_0+\delta]$ for  $\delta \ll 1$.  Suppose that $u$ is a Class B solution of the 3D ZK on $I$ and for $0<\theta<\frac14$ we have
\begin{equation}
\label{E:HR27b}
\| u\la x \ra^{1/\theta} \|_{L_I^\infty L_{\mathbf{x}}^2} <\infty \quad \mbox{and} \quad \|u \|_{L_I^\infty H_{\mathbf{x}}^1} < \infty
\end{equation}
so that \eqref{E:HR11} and \eqref{E:HR28} are available.
Then, for each $N\geq 1$,
\begin{equation}
\label{E:HR30}
\|P_N u(t) -P_NU(t-t_0) u(t_0) \|_{L_I^\infty L_{\mathbf{x}}^2} \lesssim N^{-\frac54+\theta} (\ln^+ N)^3,
\end{equation}
from which it follows that
\begin{equation}
\label{E:HR31}
\|P_N u(t_0) \|_{L_{\mathbf{x}}^2} \lesssim \delta^{-1/2} (\ln^+ N)^3 N^{-\frac54+\theta}
\end{equation}
with implicit constant depending on the norms in \eqref{E:HR27b}.  
\end{lemma}

\begin{proof}
By the Duhamel formula
$$
P_Nu(t)-P_NU(t-t_0) u(t_0) = -\int_{t_0}^t U(t-s) \, \partial_x u(s)^2 \, ds.
$$
By \eqref{E:HR21b},
\begin{equation}
\label{E:HR29}
\| P_Nu(t)-P_NU(t-t_0) u(t_0)\|_{L_I^\infty L_{\mathbf{x}}^2} \lesssim \| P_N(u^2) \|_{L_x^1L_{yz I}^2}.
\end{equation}
Using the paraproduct decomposition
$$
P_N ( u^2) \sim P_N (P_{\lesssim N} u \, P_N u) + P_N \sum_{N'\gg N} (P_{N'}u \, P_{N'}u),
$$
we obtain
$$
\| P_N (u^2) \|_{L_x^1 L_{yzI}^2} \lesssim \| P_{\lesssim N} u \|_{L_x^2L_{yzI}^\infty} \|P_N u \|_{L_{\mathbf{x}I}^2} + \sum_{N'\gg N} \| P_{N'} u \|_{L_x^2L_{yzI}^\infty} \|P_{N'} u \|_{L_{\mathbf{x}I}^2}.
$$
By \eqref{E:HR11} and \eqref{E:HR28}, we get
$$
\| P_N (u^2) \|_{L_x^1 L_{yzI}^2} \lesssim (\ln^+ N)^3 N^{-\frac54+\theta} + \sum_{N'\gg N} (\ln^+ N')^2 (N')^{-\frac54+\theta} \lesssim (\ln^+ N)^3 N^{-\frac54+\theta}.
$$
Combining this with \eqref{E:HR29}, we obtain \eqref{E:HR30}.

Since $\| P_N U(t-t_0) u(t_0)\|_{L_{\mathbf{x}}^2}$ is conserved in time, we have
\begin{align*}
\|P_N u(t_0) \|_{L_{\mathbf{x}}^2} &= (2\delta)^{-1/2} \|P_N U(t-t_0) u(t_0) \|_{L_I^2 L_{\mathbf{x}}^2} \\
&\leq \delta^{-1/2} \|P_N U(t-t_0) u(t_0)- P_N u(t) \|_{L_I^2 L_{\mathbf{x}}^2} + \delta^{-1/2} \|P_N u(t) \|_{L_I^2 L_{\mathbf{x}}^2}.
\end{align*}
By \eqref{E:HR30} and \eqref{E:HR11}, we conclude that \eqref{E:HR31} holds.
\end{proof}

We note that \eqref{E:HR31} implies that $u \in L_t^\infty H_{\mathbf{x}}^{\frac54-2\theta}$.  Now we give the arguments to achieve higher regularity. 

\begin{lemma}
Suppose that $u$ is a Class B solution of the 3D ZK on a time interval $I$ of length $|I|\leq 1$, then for $\theta>0$ sufficiently small, $s_1\geq 1$ and 
\begin{equation}
\label{E:s2s1}
s_2 = 
\begin{cases}
\frac32s_1-\frac14 - \theta,  & \text{if }1\leq s_1 < \frac32, \\
(s_1+\frac12)(1- \frac12\theta), & \text{if } s_1> \frac32,
\end{cases}
\end{equation}
we have the estimate
\begin{equation}
\label{E:HR12}
\| u\|_{L_I^2H_{\mathbf{x}}^{s_2}}^2 \lesssim \|\la x \ra^{1/\theta} u \|_{L_I^\infty L_{\mathbf{x}}^2}^\theta \, \la \| u \|_{L_I^\infty H_{\mathbf{x}}^{s_1}} \ra^{3-\theta}.
\end{equation}
\end{lemma}
Thus, for $1\leq s_1 < \frac32$, we can gain nearly $\frac12 s_1-\frac14$ derivatives, and for $s_1>\frac32$, we can gain nearly $\frac12$ derivatives, although averaged in time.  It should be noted that in the case $s_1>\frac32$, the gain is precisely $\frac12 - \frac12\theta(s_1+\frac12)$, so that one needs to take $\theta \sim 1/(2s_1)$ for large $s_1$ in order to increment the regularity by, say, $\frac14$ derivatives.  Since the power on the weight on the right side is $\la x \ra^{1/\theta}$,  the power on the weight grows like $\sim 2s_1$ as we proceed to very high regularity.  

\begin{proof}
We will need the estimate
\begin{equation}
\label{E:HR13}
\|\partial_x P_N(u^2) \|_{L^2} \lesssim 
\begin{cases}
N^{\frac52-2s_1} \| u \|_{H^{s_1}}^2,  & \text{if }1\leq s_1 < \frac32, \\
N^{1-s_1} \| u \|_{H^{s_1}} \| u \|_{H^{\frac32+}},& \text{if } s_1 > \frac32.
\end{cases}
\end{equation}
To prove \eqref{E:HR13}, we will now need the paraproduct decomposition
\begin{equation}
\label{E:HR14}
P_N( u^2 ) \approx P_N( P_N u \,  P_{\lesssim N} u + \sum_{N' \gg N} P_{N'} u \, P_{N'}u).
\end{equation}
Hence, for $1\leq s_1 < \frac32$, we estimate as
$$
\|\partial_x P_N(u^2) \|_{L^2} \lesssim N \|P_N u \|_{L^{3/s_1}} \| P_{\lesssim N} u \|_{L^{p'}} +  N \sum_{N' \gg N} \| P_{N'} u \|_{L^2} \| P_{N'}u \|_{L^\infty},
$$
where 
$$\frac{1}{p'} = \frac12 - \frac{s_1}{3}.$$
By Bernstein and Sobolev embedding
$$
\|\partial_x P_N(u^2) \|_{L^2} \lesssim N^{\frac52-s_1} \|P_N u \|_{L^2} \| u \|_{H^{s_1}} +  N \sum_{N' \gg N} (N')^{-2s_1+\frac32}\| u \|_{H^{s_1}}^2,
$$
and hence, \eqref{E:HR13} holds for $1\leq s_1 < \frac32$.  For $s_1>\frac32$, we start with \eqref{E:HR14} but apply H\"older as follows
$$
\|\partial_x P_N(u^2) \|_{L^2} \lesssim N \| P_N u \|_{L^2} \|P_{\lesssim N} u \|_{L^\infty} + N \sum_{N' \gg N} \| P_{N'} u \|_{L^2} \| P_{N'}u \|_{L^\infty}.
$$
Then \eqref{E:HR13} again follows by Bernstein.

As in the proof of Lemma \ref{L:L2boost}, the key is the estimates of the type \eqref{E:HR6} and \eqref{E:HR9}:
\begin{equation}
\label{E:HR6b}
\left|\int \, x \, P_N u \, P_N\partial_x (u^2) \, d\mathbf{x}\right| \lesssim 
\| \la x\ra^{1/\theta} u \|_{L_{\mathbf{x}}^2}^\theta \, \| u \|_{H_{\mathbf{x}}^{s_1}}^{3-\theta} \begin{cases} N^{\frac52-3s_1+\theta s_1}, & \text{if }1\leq s_1 < \frac32, \\
N^{1-2s_1+s_1\theta}, & \text{if } s_1\geq \frac32,
\end{cases}
\end{equation}
\begin{equation}
\label{E:HR9b}
\left| \int x |P_Nu|^2 \, d\mathbf{x} \right| \lesssim 
N^{-2s_1+\theta s_1} 
\|\la x \ra^{1/\theta}  u \|_{L_{\mathbf{x}}^2}^\theta \| u \|_{H_{\mathbf{x}}^{s_1}}^{2-\theta}.
\end{equation}
To prove \eqref{E:HR6b}, we estimate by H\"older 
$$
\left|\int \, x \, P_N u \, P_N\partial_x (u^2) \, d\mathbf{x}\right|  \lesssim \| xP_N u \|_{L^2} \|\partial_x P_N (u^2) \|_{L^2}.
$$
By \eqref{E:HR1}, 
$$
\left|\int \, x \, P_N u \, P_N\partial_x (u^2) \, d\mathbf{x}\right| \lesssim \| |x|^{1/\theta} P_N u \|_{L^2}^\theta  \|  P_N u \|_{L^2}^{1-\theta} \|\partial_x P_N (u^2) \|_{L^2}.
$$
Combining with \eqref{E:HR13}, we obtain \eqref{E:HR6b}.  To prove \eqref{E:HR9b}, we estimate by 
H\"older
$$
\left| \int x \, |P_Nu|^2 \, d\mathbf{x} \right| \lesssim \| xP_N u \|_{L^2} \|P_N u \|_{L^2}.
$$
By \eqref{E:HR1}, 
$$
\left| \int x \, |P_Nu|^2 \, d\mathbf{x} \right| \lesssim \| |x|^{1/\theta} P_N u \|_{L^2}  \|P_N u \|_{L^2}^{2-\theta},
$$
and hence, \eqref{E:HR9b} follows.   

Let us consider first the case $1\leq s_1 < \frac32$.  Plugging \eqref{E:HR6b} and \eqref{E:HR9b} into \eqref{E:HR10} integrated over $I$, we obtain
$$ 
\| \nabla P_N u \|_{L_I^2L_{\mathbf{x}}^2}^2 \lesssim N^{-2s_1+\theta s_1} 
\|\la x \ra^{1/\theta}  u \|_{L_I^\infty L_{\mathbf{x}}^2}^\theta \| u \|_{L_I^\infty H_{\mathbf{x}}^{s_1}}^{2-\theta}+N^{\frac52-3s_1+\theta s_1}
\| \la x\ra^{1/\theta} u \|_{L_I^\infty L_{\mathbf{x}}^2}^\theta \| u \|_{L_I^\infty H_{\mathbf{x}}^{s_1}}^{3-\theta}.
$$
Multiplying by $N^{3s_1-\frac52-2\theta}$, we get 
$$ 
N^{3s_1-\frac52-2\theta} \| \nabla P_N u \|_{L_I^2L_{\mathbf{x}}^2}^2 \lesssim 
\begin{aligned}[t]
&N^{s_1-\frac52 +\theta (s_1-2)} 
\|\la x \ra^{1/\theta}  u \|_{L_I^\infty L_{\mathbf{x}}^2}^\theta \| u \|_{L_I^\infty H_{\mathbf{x}}^{s_1}}^{2-\theta}\\
&+N^{\theta (s_1-2)} \| \la x\ra^{1/\theta} u \|_{L_I^\infty L_{\mathbf{x}}^2}^\theta \| u \|_{L_I^\infty H_{\mathbf{x}}^{s_1}}^{3-\theta}. 
\end{aligned}
$$
Summing in $N$, we obtain the claimed estimate \eqref{E:HR12} (for $1\leq s_1 < \frac32$).  Next, consider the case $s_1\geq \frac32$.   Plugging \eqref{E:HR6b} and \eqref{E:HR9b} into \eqref{E:HR10} integrated over $I$, we obtain
$$ 
\| \nabla P_N u \|_{L_I^2L_{\mathbf{x}}^2}^2 \lesssim N^{-2s_1+\theta s_1} 
\|\la x \ra^{1/\theta}  u \|_{L_I^\infty L_{\mathbf{x}}^2}^\theta \| u \|_{L_I^\infty H_{\mathbf{x}}^{s_1}}^{2-\theta}+N^{1-2s_1+\theta s_1}
\| \la x\ra^{1/\theta} u \|_{L_I^\infty L_{\mathbf{x}}^2}^\theta \| u \|_{L_I^\infty H_{\mathbf{x}}^{s_1}}^{3-\theta}. 
$$
With $s_2 = (s_1+\frac12)(1-\frac12\theta)$, multiplying by $N^{-2s_2-2}$, we have
$$ 
N^{2s_2-2} \| \nabla P_N u \|_{L_I^2L_{\mathbf{x}}^2}^2 \lesssim 
\begin{aligned}[t]
&N^{-1-\frac12 \theta} 
\|\la x \ra^{1/\theta}  u \|_{L_I^\infty L_{\mathbf{x}}^2}^\theta \| u \|_{L_I^\infty H_{\mathbf{x}}^{s_1}}^{2-\theta}\\
&+N^{-\frac12\theta} \| \la x\ra^{1/\theta} u \|_{L_I^\infty L_{\mathbf{x}}^2}^\theta \| u \|_{L_I^\infty H_{\mathbf{x}}^{s_1}}^{3-\theta}. 
\end{aligned}
$$
Summing in $N$, we obtain the claimed estimate \eqref{E:HR12} (for $s_1>\frac32$).
\end{proof}

We can assume $s_1>\frac54-$.  By Lemma \ref{L:maximal-compare}, 
$$
\| P_N u(t) - P_N U(t-t_0) u(t_0) \|_{L_x^2 L_{yzI}^\infty} \lesssim \delta^{1/4} N^{-\frac18+ \frac{\theta}{2}} (\ln^+ N)^5.
$$
Applying Lemma \ref{L:RV1}, \eqref{E:HR20} to estimate the term $P_N U(t-t_0) u(t_0)$, we obtain
\begin{equation}
\label{E:regb-2}
\begin{aligned}
\|P_N u(t) \|_{L_x^2L_{yzI}^\infty} &\lesssim N \ln^+ N \|P_N u(t_0) \|_{L_x^2} +  \delta^{1/4} N^{-\frac18+ \frac{\theta}{2}} (\ln^+ N)^5 \\
&\lesssim N^{1-s_1} \ln^+ N  + \delta^{1/2} N^{-\frac18+\frac{\theta}{2}} (\ln^+ N)^5 \lesssim N^{-1/16}.
\end{aligned}
\end{equation}
Revisiting the proof of Lemma \ref{L:reg-boost-last}, 
\begin{equation}
\label{E:regb-1}
\| P_Nu(t)-P_NU(t-t_0) u(t_0)\|_{L_I^\infty L_{\mathbf{x}}^2} \lesssim \| P_N(u^2) \|_{L_x^1L_{yz I}^2}.
\end{equation}
Using the paraproduct decomposition
$$
P_N ( u^2) \sim P_N (P_{\lesssim N} u \, P_N u) + P_N \sum_{N'\gg N} (P_{N'}u \, P_{N'}u), 
$$
we obtain
$$
\| P_N (u^2) \|_{L_x^1 L_{yzI}^2} \lesssim \| P_{\lesssim N} u \|_{L_x^2L_{yzI}^\infty} 
\|P_N u \|_{L_{ \mathbf{x}I}^2} + \sum_{N'\gg N} \| P_{N'} u \|_{L_x^2L_{yzI}^\infty} 
\|P_{N'} u \|_{L_{\mathbf{x}I}^2}.
$$
Plugging into \eqref{E:regb-1}, we have
\begin{align*}
\indentalign \| P_Nu(t)-P_NU(t-t_0) u(t_0)\|_{L_I^\infty L_{\mathbf{x}}^2} \\
& \lesssim \| P_{\lesssim N} u \|_{L_x^2L_{yzI}^\infty} \|P_N u \|_{L_{\mathbf{x}I}^2} + \sum_{N'\gg N} \| P_{N'} u \|_{L_x^2L_{yzI}^\infty} \|P_{N'} u \|_{L_{\mathbf{x}I}^2}.
\end{align*}
By \eqref{E:HR12}, \eqref{E:regb-2}, we get
\begin{equation}
\label{E:regb-2b}
\| P_Nu(t)-P_NU(t-t_0) u(t_0)\|_{L_I^\infty L_{\mathbf{x}}^2}  \lesssim N^{-s_2} + \sum_{N' \gg N} (N')^{-1/16} (N')^{-s_2} \lesssim N^{-s_2}.
\end{equation}

Since $\| P_N U(t-t_0) u(t_0)\|_{L_{\mathbf{x}}^2}$ is conserved in time, we have
\begin{align*}
\|P_N u(t_0) \|_{L_{\mathbf{x}}^2} &= (2\delta)^{-1/2} \|P_N U(t-t_0) u(t_0) \|_{L_I^2 L_{\mathbf{x}}^2} \\
&\leq \delta^{-1/2} \|P_N U(t-t_0) u(t_0)- P_N u(t) \|_{L_I^2 L_{\mathbf{x}}^2} + \delta^{-1/2} \|P_N u(t) \|_{L_I^2 L_{\mathbf{x}}^2}.
\end{align*}
Plugging \eqref{E:regb-2b} and \eqref{E:HR12} into the above, we get
$$
\|P_N u(t_0) \|_{L_{\mathbf{x}}^2}  \lesssim N^{-s_2}.
$$
Multiplying by $N^{s_2-}$ and square summing, we obtain that $u(t_0) \in H^{s_2-}$, while we started with the assumption that $\|u \|_{L_I^\infty H_{\mathbf{x}}^{s_1}} < \infty$.  Noting that $t_0$ was arbitrary in $I$, and recalling \eqref{E:s2s1} expressing $s_2$ in terms of $s_1$, we see that we can incrementally step up to arbitrarily high regularity.  


\section{$\tilde \epsilon_n$ has exponential decay}
\label{S:uniform-n-decay}

This is the first section addressing Prop. 2.  We use the $J_\pm$ monotonicity in Lemma \ref{L:Jpm-estimates} to prove Lemma \ref{L:ep-decay}, which establishes the uniform-in-$n$ exponential spatial decay of $\tilde \epsilon_n$.    In place of $\tilde \epsilon_n$, we pass to $\eta$ (subscript $n$ and tildes dropped) defined by \eqref{E:eta-def} and solving equation \eqref{E:eta-eq} in terms of which Lemma \ref{L:Jpm-estimates} is phrased.  In the estimates, we can path back and forth between the $\tilde \epsilon_n$ and $\eta$, since $\tilde c_n\sim 1$ uniformly in time.

Fix any $t_0\in \mathbb{R}$ and apply Lemma \ref{L:Jpm-estimates}.  In particular, we apply \eqref{E:Jp-right} and use that the uniform-in-time $L^2$ compactness hypothesis on $\tilde \epsilon_n$ implies
$$
\lim_{t_{-1} \searrow -\infty} J_{+,\theta,r,t_0}(t_{-1}) = 0$$
to conclude that 
\begin{equation}
\label{E:Jmon-p}
J_{+,\theta,r,t_0}(t_0) \lesssim e^{-\delta r} \sup_{t\in \mathbb{R}} \| \tilde \epsilon_n \|_{L_t^\infty L_{\mathbf{x}}^2}^2.
\end{equation}
Likewise, we apply \eqref{E:Jm-left} and use that the uniform-in-time $L^2$ compactness hypothesis on $\tilde \epsilon_n$ implies
$$\lim_{t_1\nearrow +\infty} J_{-,\theta,-r,t_0}(t_1) = 0$$
to conclude that
\begin{equation}
\label{E:Jmon-m}
J_{-,\theta,-r,t_0}(t_0) \lesssim e^{-\delta r} \sup_{t\in \mathbb{R}}  \| \tilde \epsilon_n \|_{L_t^\infty L_{\mathbf{x}}^2}^2.
\end{equation}
Let us take $\theta = \frac{\pi}{4}$ (any number less than $\frac{\pi}3-\delta$ will suffice).  Note that
$$
J_{\pm, \theta, r, t_0}(t_0) = \int_{\mathbb{R}^3} \phi_{\pm} (\cos\theta(x-r) + \sin\theta \sqrt{1+y^2+z^2}) \eta^2(\mathbf{x}+\mathbf{a}(t_0),t_0) \,d \mathbf{x}.
$$ 
The estimate \eqref{E:Jmon-p} gives the $L_{\mathbf{x}}^2$ estimate outside the cone of angle $\frac{\pi}{4}$ with the negative $x$-axis, with vertex at $(x,y,z) = (r,0,0)$.  The estimate \eqref{E:Jmon-m} gives the $L_{\mathbf{x}}^2$ estimate outside the cone of angle $\frac{\pi}{4}$ with the positive $x$-axis, with vertex at $(x,y,z) = (-r,0,0)$.  Combined they give the $L_{\mathbf{x}}^2$ estimate outside the ball of radius $r$, completing the proof of Lemma \ref{L:ep-decay} (since $t_0\in \mathbb{R}$ selected arbitrarily).  

\section{Comparability of higher Sobolev norms for $\tilde \epsilon_n$}
\label{S:Sobolev-comparability}

The goal of this section is to prove Lemma \ref{L:ep-comparability}. The proof is similar to \S\ref{S:higher-regularity}, although achieving the $H_{\mathbf{x}}^1$ bound below required a little bit more care -- there is no direct analogue in \S\ref{S:higher-regularity}, since in that section we start with the assumption of an $H_{\mathbf{x}}^1$ bound.  Here, we do have assumption \eqref{E:comp1} (right estimate) but we have to account for the $B^{-1}$ penalty when using this assumed bound.  Thus, we devised the strategy of first proving Lemma \ref{L:comp1}, which does not have $P_N$, and thus, allows for clean integration by parts in the term $\int (x-a_1) \, (\zeta^2)_x \, \zeta \, d\mathbf{x}$, to obtain the preliminary estimate \eqref{E:comp3}.  We then use \eqref{E:comp3} in the $P_N$ calculation in Lemma \ref{L:comp2}.  This is the main new idea in comparison to what is already in \S\ref{S:higher-regularity}.  

Before we begin, let us state and prove an elementary computational lemma.  In the statement, $P_N \, q \, P_M$ means the composition of operators $P_N \circ q \circ P_M$, where $q$ is the operator of multiplication by $q$.

\begin{lemma}
\label{L:weight-est}
Let  $q\in \mathcal{S}(\mathbb{R}^3)$ and $\omega>0$ arbitrary. Then for any $M,N \geq 1$, 
$$
\| P_N \, q \, P_M \|_{L^2\to L^2} \lesssim  \min \left(\frac{M}{N}, \frac{N}{M}\right)^\omega
$$
and
$$
\| \la \mathbf{x}  \ra \, P_N \, q \, P_M \|_{L^2\to L^2} \lesssim \min \left(\frac{M}{N}, \frac{N}{M}\right)^\omega
$$
with constants depending on $q$ and $\omega$.
\end{lemma}
\begin{proof}
By the Plancherel theorem, it suffices to prove the $L^2\to L^2$ estimates on the operators with kernels:
$$
K_1(\boldsymbol{\xi},\boldsymbol{\xi}') = \chi(N^{-1} \boldsymbol{\xi}) \, \hat q( \boldsymbol{\xi} - \boldsymbol{\xi}') \, \chi(M^{-1} \boldsymbol{\xi}') \quad \mbox{and}
$$
$$
K_2(\boldsymbol{\xi},\boldsymbol{\xi}') = \nabla_{\boldsymbol{\xi}}  [\chi(N^{-1} \boldsymbol{\xi}) \, \hat q( \boldsymbol{\xi} - \boldsymbol{\xi}') \, \chi(M^{-1} \boldsymbol{\xi}')].
$$
It suffices to examine $K_1$, since the $\nabla_{\boldsymbol{\xi}}$ operator in $K_2$, when distributed into the product, does not produce harmful factors.  

If $N\sim M$, then we just use that each component in the composition is an $L^2\to L^2$ operator with norm $\lesssim 1$ to obtain a bound of $\lesssim 1$ for the composition.

If $N\gg M$, then $|\boldsymbol{\xi} - \boldsymbol{\xi}'| \sim N$, so $|\hat q(\boldsymbol{\xi}-\boldsymbol{\xi}')| \lesssim N^{-\omega-3}$. Hence,
$$
\| K \|_{L_{\boldsymbol{\xi}}^\infty L_{\boldsymbol{\xi}'}^1}^{1/2} \| K \|_{L_{\boldsymbol{\xi}'}^\infty L_{\boldsymbol{\xi}}^1}^{1/2} \lesssim N^{-\omega-3} M^{3/2}N^{3/2} \lesssim N^{-\omega}.
$$

Similarly, if $M \gg N$, then $|\boldsymbol{\xi} - \boldsymbol{\xi}'| \sim M$, so $|\hat q(\boldsymbol{\xi}-\boldsymbol{\xi}')| \lesssim M^{-\omega-3}$. Hence,
$$
\| K \|_{L_{\boldsymbol{\xi}}^\infty L_{\boldsymbol{\xi}'}^1}^{1/2} \| K \|_{L_{\boldsymbol{\xi}'}^\infty L_{\boldsymbol{\xi}}^1}^{1/2} \lesssim M^{-\omega-3} M^{3/2}N^{3/2} \lesssim M^{-\omega}.
$$
The conclusion follows from these estimates and the Schur test.
\end{proof}

In this section, we prove Lemma \ref{L:ep-comparability}.    In the language of $\zeta$, we can phrase this in a way that does not reference the index $n$, but is instead a statement about obtaining bounds that are independent of the constant $0<B\ll 1$ in the equation for $\zeta$.  Let us recall from \S\ref{S:geom-decomp} the equation \eqref{E:zeta-1} for $\zeta$:
\begin{equation}
\label{E:zeta-1p}
\begin{aligned}[t]
\partial_t \zeta &= - \partial_x \Delta \zeta - 2 \partial_x ( Q_{c,\mathbf{a}} \zeta) + c^{-2} \la \zeta, f_{c,\mathbf{a}}\ra (\Lambda Q)_{c,\mathbf{a}} + c^{-2}\la \zeta, \mathbf{g}_{c,\mathbf{a}} \ra \cdot (\nabla Q)_{c,\mathbf{a}}
 \\
& \qquad - B \partial_x \zeta^2 + B \omega_c (\Lambda Q)_{c,\mathbf{a}} + B\boldsymbol{\omega}_{\mathbf{a}} \cdot (\nabla Q)_{c,\mathbf{a}},
\end{aligned}
\end{equation}
where by \eqref{E:eta-ODEs}, 
$$
|\omega_c| \lesssim 1 \quad \mbox{and} \quad |\boldsymbol{\omega}_{\mathbf{a}}| \lesssim 1.
$$

We can assume that for all $\theta>0$,
\begin{equation}
\label{E:comp1}
\| \la x-a_x\ra^{1/\theta} \zeta\|_{L_t^\infty L_{\mathbf{x}}^2} \lesssim_\theta 1 \quad \mbox{and} \quad
\| \zeta \|_{L_t^\infty H_{\mathbf{x}}^1} \lesssim \alpha B^{-1}
\end{equation}
with constant depending on $\theta$ but independent of $B$ and global in time, and we can assume that for all $s\geq 2$ and all finite length time intervals $I$, 
\begin{equation}
\label{E:comp2}
\| \zeta \|_{L_I^\infty H_{\mathbf{x}}^s} < \infty,
\end{equation}
where the bound is finite but can depend on anything, like the time interval or the constant $B$.  With these assumptions, we \emph{aim to prove} that for all $s \geq 1$, 
\begin{equation}
\label{E:comp2b}
\| \zeta \|_{L_t^\infty H_{\mathbf{x}}^s} \lesssim_s 1,
\end{equation}
where the constant  depends on $s$ but is independent of $B$ and global in time.  The assertion \eqref{E:comp2b} in the case $s=1$ will be established in Lemma \ref{L:comp4} below.  The argument is broken in steps with
$$
\text{Lemma }\ref{L:comp1} \implies \text{Lemma }\ref{L:comp2} \implies  \text{Lemma }\ref{L:comp3}  \implies \text{Lemma }\ref{L:comp4}.
$$
Higher values of $s$ are then addressed recursively by applying Lemma \ref{L:comp5} and Lemma \ref{L:comp6}, starting with $s=\frac32$, then proceeding by half-integer steps upward.

\begin{lemma}
\label{L:comp1}
Suppose \eqref{E:comp1} holds, and \eqref{E:comp2} holds for $s=1$.  Then, provided $|I| \ll 1$, 
\begin{equation}
\label{E:comp3}
\| \zeta \|_{L_I^2 H_{\mathbf{x}}^1} \lesssim 1
\end{equation}
with constant independent of $B$ and $I$.
\end{lemma}
\begin{proof}
By plugging in \eqref{E:zeta-1p}, we obtain
\begin{align*}
\partial_t \int (x-a_1) ( \zeta)^2 \, d\mathbf{x} 
&= -2 \int (x-a_1) \,  \zeta \, \partial_x \Delta  \zeta \, d\mathbf{x} \\
& \qquad - 4 \int (x-a_1) \,  \zeta \, \partial_x  ( Q_{c,\mathbf{a}} \zeta) \, d\mathbf{x} \\
& \qquad - 2 B \int (x-a_1) \,  \zeta \, \partial_x  ( \zeta^2) \, d \mathbf{x} + G(t),
\end{align*}
where
\begin{align*}
G(t) &= -2 c^{-2} \la \zeta, f_{c,\mathbf{a}}\ra  \int (x-a_1) \,  \zeta \,  (\Lambda Q)_{c,\mathbf{a}} \, d\mathbf{x} \\
& \qquad -2 c^{-2} \la \zeta, \mathbf{g}_{c,\mathbf{a}} \ra  \cdot  \int (x-a_1) \,  \zeta \,  (\nabla Q)_{c,\mathbf{a}} \, d\mathbf{x}  \\
& \qquad + 2B \omega_c \int (x-a_1) \, \zeta \, ( \Lambda Q)_{c,\mathbf{a}} \, d\mathbf{x} \\
& \qquad + 2B \omega_{\mathbf{a}} \cdot \int (x-a_1) \,  \zeta \, (\nabla Q)_{c,\mathbf{a}} \, d \mathbf{x}.
\end{align*}
Simplifying the term $-2 \int (x-a_1) \,  \zeta \, \partial_x \Delta  \zeta \, d\mathbf{x}$ (the first term on the right) via integration by parts, moving it over to the left, and integrating in time over $I=[t_-,t_+]$, we obtain
\begin{equation}
\label{E:comp17}
\| \zeta \|_{L_I^2 \dot H_{\mathbf{x}}^1}^2 \lesssim \int_I\int_{\mathbf{x}} [3 (\partial_x  \zeta)^2 + (\partial_y \zeta)^2]  \, d\mathbf{x} \, dt = H_1+H_2+H_3 + \int_{t_-}^{t_+}G(t) \,dt,
\end{equation}
where
\begin{align*}
H_1 &\defeq \int_{\mathbf{x}} (x-a_1) ( \zeta)^2 \, d\mathbf{x}  \Big|^{t=t_+}_{t=t_-},  \\
H_2 &\defeq - 4 \int_I \int_{\mathbf{x}} (x-a_1) \,  \zeta \, \partial_x  ( Q_{c,\mathbf{a}} \zeta) \, d\mathbf{x} \, dt, \\
H_3 &\defeq - 2 B \int_I \int_{\mathbf{x}} (x-a_1) \,  \zeta \, \partial_x  ( \zeta^2) \, d \mathbf{x} \, dt.
\end{align*}
First, we address $H_3$.  By integration by parts,
$$
\int_{\mathbf{x}} (x-a_1) \zeta (\zeta^2)_x \, d\mathbf{x} = - \frac{2}{3}\int_{\mathbf{x}} \zeta^3 \, d\mathbf{x},
$$
and hence,
$$
\left| \int_{\mathbf{x}} (x-a_1) \, \zeta (\zeta^2)_x \, d\mathbf{x} \right| \lesssim \| \zeta\|_{L_{\mathbf{x}}^3}^3 \lesssim \| \zeta \|_{L^2_{\mathbf{x}}}^{3/2}  \| \zeta \|_{\dot H^1_{\mathbf{x}}}^{3/2} \lesssim \|\zeta\|_{L_{\mathbf{x}}^2}^6 + \|\zeta \|_{\dot H_{\mathbf{x}}^1}^2.
$$
Adding the time integration, we obtain
$$
|H_3| \lesssim B |I| \, \|\zeta\|_{L_I^\infty L_{\mathbf{x}}^2}^6
+ B \, \|\zeta \|_{L_I^2 \dot H_{\mathbf{x}}^1}^2 \lesssim 1
+ B \, \|\zeta \|_{L_I^2 \dot H_{\mathbf{x}}^1}^2.
$$
Owing to the $B$ coefficient, the second term is easily absorbed on the left in \eqref{E:comp17}.   Next, we address $H_2$:  
\begin{align*}
\indentalign \int_{\mathbf{x}} (x-a_1) \,  \zeta \, \partial_x  ( Q_{c,\mathbf{a}} \zeta) \, d\mathbf{x}  = \int_{\mathbf{x}} (x-a_1) \,  (\partial_xQ_{c,\mathbf{a}}) \zeta^2 \, d\mathbf{x} + \int_{\mathbf{x}} (x-a_1) \,   Q_{c,\mathbf{a}} \, \zeta \zeta_x \, d\mathbf{x}\\
& = \int_{\mathbf{x}} (x-a_1) \,  (\partial_xQ_{c,\mathbf{a}}) \zeta^2 \, d\mathbf{x} -\frac12 \int_{\mathbf{x}} \partial_x[(x-a_1) \,   Q_{c,\mathbf{a}}] \, \zeta^2 \, d\mathbf{x}.
\end{align*}
Thus,
$$
|H_2| \lesssim |I| \|\zeta\|_{L_I^\infty L_{\mathbf{x}}^2}^2 \lesssim 1.
$$
Next,
$$
|H_1| \lesssim \|\la x-a_1 \ra \zeta \|_{L_I^\infty L_{\mathbf{x}}^2} \| \zeta \|_{L_I^\infty L_{\mathbf{x}}^2} \lesssim 1.
$$
The terms in $G$ are straightforwardly bounded by 
$$
\|G\|_{L_I^1} \lesssim |I| \|\zeta\|_{L_I^\infty} \lesssim 1.
$$
With all of these estimates, the bound follows from \eqref{E:comp17} .
\end{proof}

\begin{lemma}
\label{L:comp2}
Suppose $|I| \ll 1$, \eqref{E:comp1} holds, and \eqref{E:comp2} holds for $s=1$, so that \eqref{E:comp3} holds as well.  Then for all $N\geq 1$ and $0\leq \omega \leq \frac18$,
\begin{equation}
\label{E:comp4}
\| P_N \zeta \|_{L_I^2 H_{\mathbf{x}}^1} \lesssim  N^{-\omega} B^{-\omega}
\end{equation}
with constant independent of $N$, $B$ and $I$.  (Notice that $B^{-\omega}$ is a penalty but $N^{-\omega}$ is a gain.)
\end{lemma}
Therefore, we can obtain a gain in $N$ at the expense of a penalty in $B$.

\begin{proof}
By plugging in \eqref{E:zeta-1p}, we obtain
\begin{align*}
\partial_t \int (x-a_1) (P_N \zeta)^2 \, d\mathbf{x} 
&= -2 \int (x-a_1) \, P_N \zeta \, \partial_x \Delta P_N \zeta \, d\mathbf{x} \\
& \qquad - 4 \int (x-a_1) \, P_N \zeta \, \partial_x P_N ( Q_{c,\mathbf{a}} \zeta) \, d\mathbf{x} \\
& \qquad - 2 B \int (x-a_1) \, P_N \zeta \, \partial_x P_N ( \zeta^2) \, d \mathbf{x} + G(t),
\end{align*}
where
\begin{align*}
G(t) &= -2 c^{-2} \la \zeta, f_{c,\mathbf{a}}\ra  \int (x-a_1) \, P_N \zeta \, P_N (\Lambda Q)_{c,\mathbf{a}} \, d\mathbf{x} \\
& \qquad -2 c^{-2} \la \zeta, \mathbf{g}_{c,\mathbf{a}} \ra  \cdot  \int (x-a_1) \, P_N \zeta \, P_N (\nabla Q)_{c,\mathbf{a}} \, d\mathbf{x}  \\
& \qquad + 2B \omega_c \int (x-a_1) \, P_N\zeta \, ( \Lambda Q)_{c,\mathbf{a}} \, d\mathbf{x} \\
& \qquad + 2B \omega_{\mathbf{a}} \cdot \int (x-a_1) \, P_N \zeta \, (\nabla Q)_{c,\mathbf{a}} \, d \mathbf{x}.
\end{align*}
Simplifying the term $-2 \int (x-a_1) \, P_N \zeta \, \partial_x \Delta P_N \zeta \, d\mathbf{x}$ (the first term on the right) via integration by parts, moving it over to the left, and integrating in time over $I=[t_-,t_+]$, we obtain
\begin{equation}
\label{E:stuff1}
N^2 \, \| P_N \zeta \|_{L_I^2 L_{\mathbf{x}}^2}^2 \lesssim \int_I\int_{\mathbf{x}} [3 (\partial_x P_N \zeta)^2 + (\partial_y P_N\zeta)^2]  \, d\mathbf{x} \, dt = H_1+H_2+H_3 + \int_{t_-}^{t_+}G(t) \,dt,
\end{equation}
where, similarly as in the previous lemma, we define
\begin{align*}
H_1 &\defeq \int_{\mathbf{x}} (x-a_1) (P_N \zeta)^2 \, d\mathbf{x}  \Big|^{t=t_+}_{t=t_-}  \\
H_2 &\defeq - 4 \int_I \int_{\mathbf{x}} (x-a_1) \, P_N \zeta \, \partial_x P_N ( Q_{c,\mathbf{a}} \zeta) \, d\mathbf{x} \, dt \\
H_3 &\defeq - 2 B \int_I \int_{\mathbf{x}} (x-a_1) \, P_N \zeta \, \partial_x P_N ( \zeta^2) \, d \mathbf{x} \, dt.
\end{align*}

The terms in $G$ are easily bounded.  Note that in estimating $H_3$, we can use \eqref{E:comp3} as follows
\begin{align*}
|H_3| &\leq B \,\| \partial_x P_N \zeta^2 \|_{L_I^1L_{\mathbf{x}}^{3/2}} \| (x-a_1) P_N \zeta \|_{L_I^\infty L_{\mathbf{x}}^3} 
\leq B \, \| \zeta \zeta_x \|_{L_I^1L_{\mathbf{x}}^{3/2}} \| (x-a_1) P_N \zeta \|_{L_I^\infty L_{\mathbf{x}}^3} \\
&\leq B \,\| \zeta\|_{L_I^2 L_{\mathbf{x}}^6} \| \zeta_x \|_{L_I^2L_{\mathbf{x}}^2} \| (x-a_1) P_N \zeta \|_{L_I^\infty L_{\mathbf{x}}^3} 
\lesssim B\, \| (x-a_1) P_N \zeta \|_{L_I^\infty L_{\mathbf{x}}^3}
\end{align*}
by Sobolev embedding and \eqref{E:comp3}.  Following through with estimate \eqref{E:HR2}, $\theta=\frac12$, we obtain
$$
|H_3| \lesssim B\, \| |x-a_1|^2 P_N\zeta \|_{L_I^\infty L_\mathbf{x}^2}^{1/2} \| P_N \zeta \|_{L_I^\infty L_{\mathbf{x}}^6}^{1/2}.
$$
By \eqref{E:HR1},
$$
|H_3| \lesssim B \| |x-a_1|^{2/\theta} P_N\zeta \|_{L_I^\infty L_\mathbf{x}^2}^{\theta/2}  \|  P_N\zeta \|_{L_I^\infty L_\mathbf{x}^2}^{(1-\theta)/2} \| P_N \zeta \|_{L_I^\infty L_{\mathbf{x}}^6}^{1/2}.
$$
Finally, by \eqref{E:HR3} and Sobolev embedding,
$$
|H_3| \lesssim B\, \| |x-a_1|^{2/\theta} \zeta \|_{L_I^\infty L_\mathbf{x}^2}^{\theta/2}  \|  P_N\zeta \|_{L_I^\infty L_\mathbf{x}^2}^{(1-\theta)/2} \| P_N \zeta \|_{L_I^\infty H_{\mathbf{x}}^1}^{1/2}.
$$
By \eqref{E:comp1}, 
$$
|H_3| \lesssim B^{\theta/2} N^{-\frac{1-\theta}{2}},
$$
which suffices for \eqref{E:comp4}.   For $H_2$, we estimate as
$$
|H_2| \lesssim N \, \| P_N \zeta \|_{L_I^2 L_{\mathbf{x}}^2} \| (x-a_1) P_N ( Q_{c,\mathbf{a}} \zeta) \|_{L_I^2 L_{\mathbf{x}}^2}.
$$
Expanding $\zeta = \sum_{M\geq 1} P_M \zeta$,
$$
|H_2| \lesssim N \, \| P_N \zeta \|_{L_I^2 L_{\mathbf{x}}^2} \sum_{M \geq 1} \| (x-a_1) P_N ( Q_{c,\mathbf{a}} P_M \zeta) \|_{L_I^2 L_{\mathbf{x}}^2}.
$$
Applying Lemma \ref{L:weight-est}, we obtain
$$
|H_2| \lesssim N \, \| P_N \zeta \|_{L_I^2 L_{\mathbf{x}}^2} \, \| \zeta \|_{L_I^2 H_\mathbf{x}^1}  \sum_{M \geq 1} \min (NM^{-1}, MN^{-1}) M^{-1}.
$$
The sum carries out to $N^{-1}$.  By \eqref{E:comp3}, we can bound
$$
|H_2| \lesssim \| P_N \zeta \|_{L_I^2 L_{\mathbf{x}}^2} \lesssim \epsilon N^2\|P_N \zeta\|_{L_I^2 L_{\mathbf{x}}^2}^2 + \epsilon^{-1} N^{-2}.
$$
The first term can be absorbed into the main term \eqref{E:stuff1}, while the second term is an acceptable contribution to the upper bound in \eqref{E:comp4}.  For $H_1$, we estimate as
$$
|H_1| \lesssim \| |x-a_1| P_N \zeta \|_{L_I^\infty L_{\mathbf{x}}^2} \| P_N \zeta \|_{L_I^\infty L_{\mathbf{x}}^2}. 
$$
From \eqref{E:comp1}, it follows that $|H_1| \lesssim 1$.  On the other hand, we can also estimate as
$$
H_1 \lesssim  \| |x-a_1| \zeta \|_{L_I^\infty L_{\mathbf{x}}^2}  N^{-1} \| P_N \zeta \|_{L_I^\infty H_{\mathbf{x}}^1},
$$
and by applying \eqref{E:comp1}, obtain $|H_1| \lesssim B^{-1}N^{-1}$.  Interpolating, we obtain a bound of $B^{-\omega} N^{-\omega}$ for any $0\leq \omega \leq 1$.
\end{proof}

\begin{lemma}
\label{L:comp3}
Assume \eqref{E:comp1} and suppose  \eqref{E:comp2} holds for $s=1$. Suppose $I$ is an interval of length $0<\delta \ll 1$.  Then \eqref{E:comp3} and \eqref{E:comp4} hold, and in addition for $N\geq 1$,
\begin{equation}
\label{E:comp5}
\|P_N \zeta \|_{L_x^2 L_{yz I}^\infty} \lesssim (\ln^+ N)^4 \delta^{-1/2}
\end{equation}
with constant independent of $B$ and $I$.
\end{lemma}
\begin{proof}
Let $t_0\in I$ be such that
\begin{equation}
\label{E:comp10}
\|\zeta(t_0)\|_{H_{\mathbf{x}}^1} = \min_{t\in I} \|\zeta(t)\|_{H_{\mathbf{x}}^1} \leq \delta^{-1/2} \| \zeta \|_{L_I^2H_{\mathbf{x}}^1} \lesssim \delta^{-1/2}.
\end{equation}
Then we estimate
\begin{equation}
\label{E:comp8}
\gamma_N \defeq \| P_N \zeta(t) \|_{L_x^2 L_{y z I}^\infty}
\end{equation}
as follows.  Note that
\begin{equation}
\label{E:comp9}
P_N \zeta(t) = P_N  U(t-t_0) \zeta(t_0) + \int_{t_0}^t U(t-t') P_N F(t') \,d t',
\end{equation}
where $F = \sum_j F_j$ and the $F_j$ are terms in \eqref{E:zeta-1p}, specifically,
\begin{equation}
\label{E:comp9b}
\begin{aligned}
& F_1 = - 2 \partial_x ( Q_{c,\mathbf{a}} \zeta) \,,
&& F_2 =  c^{-2} \la \zeta, f_{c,\mathbf{a}}\ra (\Lambda Q)_{c,\mathbf{a}} \,,
&& F_3 =  c^{-2}\la \zeta, \mathbf{g}_{c,\mathbf{a}} \ra \cdot (\nabla Q)_{c,\mathbf{a}} \,, \\
&F_4 = - B \partial_x \zeta^2 \,,
&&F_5 = B \omega_c (\Lambda Q)_{c,\mathbf{a}} \,,
&&F_6 =  B\boldsymbol{\omega}_{\mathbf{a}} \cdot (\nabla Q)_{c,\mathbf{a}} \,,
\end{aligned}
\end{equation}
and the estimate of \eqref{E:comp8} via Lemma \ref{L:RV1} applied to \eqref{E:comp9} corresponding to $F_j$ will be denoted by $\gamma_{N,j}$, so that we have 
$$\gamma_N \leq \sum_j \gamma_{N,j}.$$
By \eqref{E:HR20},
\begin{equation}
\label{E:comp11}
\|P_N U(t-t_0) \zeta (t_0) \|_{L_x^2 L_{y z I}^\infty} \lesssim (\ln^+ N)^2 \| \zeta(t_0) \|_{H_{\mathbf{x}}^1} \lesssim \delta^{-1/2} (\ln^+ N)^2,
\end{equation}
where in the last step, we used \eqref{E:comp10}.
Now we consider the term $F_4$.  By \eqref{E:HR21},
$$
\gamma_{N,4} \lesssim B (\ln^+ N)^2 N \|P_N(\zeta^2) \|_{L_x^1 L_{yzI}^2}.
$$
Using the decomposition
\begin{equation}
\label{E:comp15}
P_N( \zeta^2) \sim P_N(P_{\lesssim N} \zeta  \cdot P_N \zeta) + \sum_{N'\geq N} P_N( P_{N'}\zeta \cdot P_{N'}\zeta) 
\end{equation}
and H\"older, we obtain
$$
\gamma_{N,4} \lesssim B (\ln^+ N)^2 N ( \|P_{\lesssim N} \zeta \|_{L_x^2 L_{yzI}^\infty} \|P_N \zeta \|_{L_x^2 L_{yzI}^2} + \sum_{N'\geq N} \|P_{N'} \zeta \|_{L_x^2 L_{yzI}^\infty} \|P_{N'} \zeta \|_{L_x^2 L_{yzI}^2} ).
$$
By \eqref{E:comp4} in Lemma \ref{L:comp2},
$$
\gamma_{N,4} \lesssim B^{1-\omega} N^{-\omega} (\ln^+ N)^2 \left( \|P_{\lesssim N} \zeta \|_{L_x^2 L_{yzI}^\infty}  + \sum_{N'\geq N} \frac{N^{1+\omega}}{(N')^{1+\omega}} \|P_{N'} \zeta \|_{L_x^2 L_{yzI}^\infty} \right),
$$
and thus,
\begin{equation}
\label{E:comp12}
\gamma_{N,4}\lesssim B^{1-\omega} N^{-\omega/2} \sum_{N' \geq 1} \min\left( 1 , \frac{N^{1+\omega}}{(N')^{1+\omega}}\right) \|P_{N'} \zeta \|_{L_x^2 L_{y z I}^\infty}.
\end{equation}
By \eqref{E:HR21},
\begin{align*}
\gamma_{N,1} &\lesssim (\ln^+ N)^2  \|P_N\partial_x ( Q_{c,\mathbf{a}} \zeta) \|_{L_x^1 L_{y z I}^2}\\
&\lesssim (\ln^+ N)^2  \|\partial_x ( Q_{c,\mathbf{a}} \zeta) \|_{L_x^1 L_{y z I}^2} \\
&\lesssim (\ln^+ N)^2 ( \|Q_{c,\mathbf{a}}\|_{L_x^2 L_{y z I}^\infty}+\|(\partial_x Q)_{c,\mathbf{a}}\|_{L_x^2 L_{y z I}^\infty}) \| \zeta \|_{L_I^2 H_{\mathbf{x}}^1}.
\end{align*}
By \eqref{E:comp3} in Lemma \ref{L:comp1},
\begin{equation}
\label{E:comp13}
\gamma_{N,1} \lesssim (\ln^+ N)^2.
\end{equation}
Since for each $\theta>0$, we have
$$
\|P_N (\Lambda Q)_{c,\mathbf{a}} \|_{L_x^2 L_{y z I}^\infty} \lesssim N^{-\theta} \quad \mbox{and} \quad 
\|P_N (\nabla Q)_{c,\mathbf{a}} \|_{L_x^2 L_{y z I}^\infty} \lesssim N^{-\theta},
$$
the remaining terms are more straightforward to estimate and we have
\begin{equation}
\label{E:comp14}
\gamma_{N,2} + \gamma_{N,3} + \gamma_{N,5} + \gamma_{N,6} \lesssim 1.
\end{equation}
By \eqref{E:comp9}, \eqref{E:comp11}, \eqref{E:comp12}, \eqref{E:comp13}, and \eqref{E:comp14}, we have
$$
\gamma_N \lesssim \delta^{-1/2}(\ln^+ N)^2 +  B^{1-\omega} N^{-\omega/2} \sum_{N' \geq 1} \min( 1 , \frac{N^{1+\omega}}{(N')^{1+\omega}}) \gamma_{N'}.
$$
Multiply by $(\ln^+ N)^{-4}$ and sum over dyadic $N \geq 1$ to obtain
$$
\sum_{N \geq 1} (\ln^+ N)^{-4} \gamma_N \lesssim \delta^{-1/2} +  B^{1-\omega} \sum_{N \geq 1} N^{-\omega/2} \sum_{N' \geq 1} \min( 1 , \frac{N^{1+\omega}}{(N')^{1+\omega}}) \gamma_{N'}.
$$
Interchanging the order of $N$ and $N'$ summation, we obtain
$$
\sum_{N \geq 1} (\ln^+ N)^{-4} \gamma_N \lesssim \delta^{-1/2} +  B^{1-\omega} \sum_{N' \geq 1} (N')^{-\omega/2} \gamma_{N'}.
$$
Since $B^{1-\omega} \ll 1$ and $(N')^{-\omega/2} \lesssim (\ln^+ N')^{-4}$, it follows that
$$
\sum_{N \geq 1} (\ln^+ N)^{-4} \gamma_N \lesssim \delta^{-1/2},
$$ 
and, in particular, \eqref{E:comp5} holds.
\end{proof}

\begin{lemma}
\label{L:comp4}
Suppose  \eqref{E:comp1} and \eqref{E:comp2} hold for $s=1$.  Suppose $I$ is an interval of length $0<\delta \ll 1$.   Then \eqref{E:comp3}, \eqref{E:comp4} and \eqref{E:comp5} hold, and moreover,
$$
\| \zeta \|_{L_I^\infty H_{\mathbf{x}}^1} \lesssim \delta^{-1/2}
$$
with constant independent of $B$ and $I$.
\end{lemma}
\begin{proof}
We start by writing the Duhamel formula
$$
\zeta(t) = U(t-t_0) \zeta(t_0) + \sum_{j=1}^6 \int_{t_0}^t U(t-t_0) F_j(t') \, dt'
$$
with $F_j$ defined by \eqref{E:comp9b}.   By the standard estimate for the linear flow,
$$
\| \zeta \|_{L_I^\infty H_{\mathbf{x}}^1} \lesssim \delta^{-1/2} + \sum_{j=1}^6 \mu_j,
$$
where
$$
\mu_j = \left\| \int_{t_0}^t U(t-t_0) F_j(t') \, dt' \right\|_{L_I^\infty H_{\mathbf{x}}^1}.
$$
By \eqref{E:HR21b},
$$
\mu_4 \lesssim B\| \nabla( \zeta^2) \|_{L_x^1 L_{y z I}^2}.
$$
Using, as usual, \eqref{E:comp15},
$$
\mu_4 \lesssim B \sum_{N \geq 1} N \left( \| P_{\lesssim N} \zeta \|_{L_x^2 L_{y z I}^\infty} \|P_N \zeta \|_{L_I^2 L_\mathbf{x}^2} + \sum_{N' \geq N}  \| P_{\lesssim N'} \zeta \|_{L_x^2 L_{y z I}^\infty} \|P_{N'} \zeta \|_{L_I^2 L_\mathbf{x}^2} \right).
$$
By  \eqref{E:comp4} and \eqref{E:comp5},
$$
\mu_4 \lesssim B^{1-\omega} \sum_{N\geq 1} \left(  (\ln^+ N)^5 \delta^{-1/2} N^{-\omega} + \sum_{N'\geq N} \frac{N}{N'} (\ln^+ N')^4 (N')^{-\omega} \delta^{-1/2} \right) \lesssim \delta^{-1/2}.
$$
By \eqref{E:HR21b} and \eqref{E:comp3},
$$
\mu_1 \lesssim \| \nabla (Q_{c,\mathbf{a}} \zeta) \|_{L_x^1 L_{y z I}^2} \lesssim (\|Q_{c,\mathbf{a}} \|_{L_x^2 L_{y z I}^\infty} +\|\nabla Q_{c,\mathbf{a}} \|_{L_x^2 L_{y z I}^\infty}) \|\zeta \|_{L_I^2 H_{\mathbf{x}}^1}  \lesssim 1.
$$
The estimates for $\mu_2$, $\mu_3$, $\mu_5$, and $\mu_6$ are more straightforward, since the terms $(\Lambda Q)_{c,\mathbf{a}}$ and  $(\nabla Q)_{c,\mathbf{a}}$ absorb derivatives.
\end{proof}

Thus, we have established that \eqref{E:comp2b} holds for $s=1$.    From here, the argument is similar but a bit easier, and we increment by half-derivatives recursively with Lemmas \ref{L:comp5}-\ref{L:comp6} below.

\begin{lemma}
\label{L:comp5}
Suppose  \eqref{E:comp1} holds, and \eqref{E:comp2b} holds for some $s\geq 1$.   Then for $|I|\leq 1$, 
\begin{equation}
\label{E:comp7}
\| \zeta \|_{L_I^2 H_{\mathbf{x}}^{s+\frac34}} \lesssim 1.
\end{equation}
\end{lemma}

\begin{proof}
We  know that 
$$
\sum_{N\geq 1} N^{2s+\frac32} \| P_N \zeta \|_{L_I^2 L_{\mathbf{x}}^2}  < \infty,
$$
so we just have to prove that it is bounded independently of $B$ and $I$, which is a key difference from the analysis here and that in \S \ref{S:higher-regularity}, and allows us to give a simpler argument here.  
By plugging in \eqref{E:zeta-1p}, we obtain
\begin{align*}
\partial_t \int (x-a_1) (P_N \zeta)^2 \, d\mathbf{x} 
&= -2 \int (x-a_1) \, P_N \zeta \, \partial_x \Delta P_N \zeta \, d\mathbf{x} \\
& \qquad - 4 \int (x-a_1) \, P_N \zeta \, \partial_x P_N ( Q_{c,\mathbf{a}} \zeta) \, d\mathbf{x} \\
& \qquad - 2 B \int (x-a_1) \, P_N \zeta \, \partial_x P_N ( \zeta^2) \, d \mathbf{x} + G(t),
\end{align*}
where
\begin{align*}
G(t) &= -2 c^{-2} \la \zeta, f_{c,\mathbf{a}}\ra  \int (x-a_1) \, P_N \zeta \, P_N (\Lambda Q)_{c,\mathbf{a}} \, d\mathbf{x} \\
& \qquad -2 c^{-2} \la \zeta, \mathbf{g}_{c,\mathbf{a}} \ra  \cdot  \int (x-a_1) \, P_N \zeta \, P_N (\nabla Q)_{c,\mathbf{a}} \, d\mathbf{x}  \\
& \qquad + 2B \omega_c \int (x-a_1) \, P_N\zeta \, ( \Lambda Q)_{c,\mathbf{a}} \, d\mathbf{x} \\
& \qquad + 2B \omega_{\mathbf{a}} \cdot \int (x-a_1) \, P_N \zeta \, (\nabla Q)_{c,\mathbf{a}} \, d \mathbf{x}.
\end{align*}
Simplifying the term $-2 \int (x-a_1) \, P_N \zeta \, \partial_x \Delta P_N \zeta \, d\mathbf{x}$ (the first term on the right) via integration by parts, moving it over to the left, and integrating in time over $I=[t_-,t_+]$, we obtain
$$
N^2 \| P_N \zeta \|_{L_I^2 L_{\mathbf{x}}^2}^2 \lesssim \int_I\int_{\mathbf{x}} [3 (\partial_x P_N \zeta)^2 + (\partial_y P_N\zeta)^2]  \, d\mathbf{x} \, dt = H_1+H_2+H_3 + \int_{t_-}^{t_+}G(t) \,dt,
$$
where
\begin{align*}
H_1 &\defeq \int_{\mathbf{x}} (x-a_1) (P_N \zeta)^2 \, d\mathbf{x}  \Big|^{t=t_+}_{t=t_-},  \\
H_2 &\defeq - 4 \int_I \int_{\mathbf{x}} (x-a_1) \, P_N \zeta \, \partial_x P_N ( Q_{c,\mathbf{a}} \zeta) \, d\mathbf{x} \, dt, \\
H_3 &\defeq - 2 B \int_I \int_{\mathbf{x}} (x-a_1) \, P_N \zeta \, \partial_x P_N ( \zeta^2) \, d \mathbf{x} \, dt.
\end{align*}
Multiply by $N^{2s-\frac12}$ and sum over dyadic $N\geq 1$, to obtain
\begin{equation}
\label{E:comp16}
\sum_{N\geq 1} N^{2s+\frac32} \|P_N \zeta\|_{L_I^2 L_{\mathbf{x}}^2}^2 \lesssim \sum_{j=1}^3 \sum_{N\geq 1} N^{2s-\frac12} H_j + \sum_{N\geq 1} N^{2s-\frac12} \|G\|_{L_I^1}.
\end{equation}
Let us first focus on the term 
$$
H_3 =  2 B \int_I \int_{\mathbf{x}}  \, P_N \zeta \, \partial_x P_N ( \zeta^2) \, d \mathbf{x} \, dt + 2 B \int_I \int_{\mathbf{x}} (x-a_1) \, \partial_x P_N \zeta \,  P_N ( \zeta^2) \, d \mathbf{x} \, dt.
$$
Using \eqref{E:comp15},
\begin{align*}
H_3 &\lesssim B \| \la x-a_1\ra P_N ( P_{\lesssim N} \zeta \, P_N \zeta) \, \partial_x P_N \zeta \|_{L_I^1 L_{\mathbf{x}}^1} \\
& \qquad + B \sum_{N'\geq N}\| \la x-a_1\ra P_N ( P_{N'} \zeta \, P_{N'} \zeta) \, \partial_x P_N \zeta \|_{L_I^1 L_{\mathbf{x}}^1}.
\end{align*}
Consider the first term on the right side of the above estimate.  Since $\la x \ra P_N \la x \ra^{-1}$ is an $L^2_{\mathbf{x}} \to L^2_{\mathbf{x}}$ bounded operator with operator norm $\lesssim 1$ (independent of $N\geq 1$),
\begin{align*}
H_{31} &\lesssim B\| \la x-a_1 \ra P_{\lesssim N} \zeta \, \, P_N \zeta \|_{L_I^2 L_{\mathbf{x}}^2} \, \| \partial_x P_N \zeta\|_{L_I^2 L_{\mathbf{x}}^2}\\
&\lesssim B \| \la x-a_1\ra P_{\lesssim N} \zeta \|_{L_I^\infty L_{\mathbf{x}}^3} \, \|P_N \zeta \|_{L_I^2 L_{\mathbf{x}}^6} \, \| \partial_x P_N \zeta \|_{L_I^2 L_\mathbf{x}^2}.
\end{align*}
By the Bernstein inequality and the fact that $\| \la x-a_1\ra P_{\lesssim N} \zeta \|_{L_I^\infty L_{\mathbf{x}}^3} \lesssim 1$ by the hypotheses (since $s>\frac12$), we get
$$
H_{31} \lesssim B N^2 \|P_N \zeta \|_{L_I^2 L_\mathbf{x}^2}^2.
$$
A similar analysis of the other term gives
$$
H_{32} \lesssim B \sum_{N'\geq N} (N')^2 \|P_{N'} \zeta \|_{L_I^2 L_\mathbf{x}^2}^2.
$$
Thus,
$$ 
H_3 \lesssim  B \sum_{N' \gtrsim N} (N'^2)  \| P_{N'} \zeta \|_{L_I^2 L_{\mathbf{x}}^2}^2. 
$$
By reversing the order of the double sum (sum in $N$ and sum in $N'$), we obtain
$$
\sum_{N \geq 1} N^{2s-\frac12} H_3 \lesssim B\sum_{N' \geq 1} (N')^{2s+\frac32} \|P_{N'}\zeta\|_{L_I^2L_{\mathbf{x}}^2}^2.
$$
Since $B\ll 1$, this term can be absorbed back on the left in \eqref{E:comp16}.   

The term $H_2$ is handled as in Lemma \ref{L:comp2}.  
$$
|H_2| \lesssim N \|P_N \zeta \|_{L_I^2L_{\mathbf{x}}^2} \| (x-a_1) \tilde P_N ( Q_{c,\mathbf{a}} \zeta) \|_{L_I^2 L_{\mathbf{x}}^2},
$$
where $\tilde P_N$ is a Littlewood-Paley multiplier different from $P_N$.
By Lemma \ref{L:weight-est},
$$
\| (x-a_1) P_N( Q_{c,\mathbf{a}} P_M \zeta) \|_{L_{\mathbf{x}}^2} \lesssim \min\left( \frac{M}{N}, \frac{N}{M}\right)^{s+1} \|P_M \zeta \|_{L_{\mathbf{x}}^2}.
$$
Thus, upon expanding $\zeta = \sum_{M \geq 1} P_M \zeta$, we obtain
$$
|H_2| \lesssim N \| P_N \zeta \|_{L_{\mathbf{x}}^2}  \sum_{M\geq 1} \min\left( \frac{M}{N}, \frac{N}{M}\right)^{s+1}\| P_M \zeta \|_{L_{\mathbf{x}}^2}.
$$
By Cauchy-Schwarz and the discrete Schur test applied to the kernel 
$$
K(M,N) = N^{s+\frac14} \min(MN^{-1}, NM^{-1})^{s+1} M^{-s-\frac14} \leq \min( N^{-1}M, N^{2s+1}M^{-(2s+1)}),
$$
we obtain
$$
\sum_{N\geq 1} N^{2s-\frac12}|H_2|  \lesssim \sum_{N\geq 1} N^{2s+\frac12} \|P_N \zeta\|_{L_{\mathbf{x}}^2}^2.
$$
This term is easy to absorb for $N \gg 1$, but for $N \lesssim 1$, it is trivially bounded.  Specifically, for $0< \delta \ll 1$ small but independent of $N$,
\begin{align*}
\sum_{N\geq 1} N^{2s-\frac12} H_2 &\lesssim \sum_{1\leq N\leq \log_2 \delta^{-1}} N^{2s+\frac12}  \|P_N \zeta \|_{L_I^2 L_{\mathbf{x}}^2}^2 + \sum_{N\geq \log_2 \delta^{-1}} N^{2s+\frac12}  \|P_N \zeta \|_{L_I^2 L_{\mathbf{x}}^2}^2 \\
&\lesssim \delta^{-1/2}|I|^{1/2}\sum_{1\leq N\leq \log_2 \delta^{-1}} N^{2s}  \|P_N \zeta \|_{L_I^\infty L_{\mathbf{x}}^2}^2 + \delta \sum_{N\geq \log_2 \delta^{-1}} N^{2s+\frac32}  \|P_N \zeta \|_{L_I^2 L_{\mathbf{x}}^2}^2.
\end{align*}
For $\delta$ sufficiently small, the second term can be absorbed on the left in \eqref{E:comp16}.

For $H_1$ we use
\begin{align*}
H_1 &\leq \|\la x-a_1 \ra P_N \zeta \|_{L_I^\infty L_{\mathbf{x}}^2} \|P_N \zeta \|_{L_I^\infty L_{\mathbf{x}}^2} \\
& \lesssim \| \la x-a_1 \ra^{1/\theta} P_N \zeta \|_{L_I^\infty L_{\mathbf{x}}^2}^\theta \|P_N \zeta \|_{L_I^\infty L_\mathbf{x}^2}^{2-\theta} \\
& \lesssim  N^{-s(2-\theta)} ( N^s \|P_N \zeta \|_{L_I^\infty L_{\mathbf{x}}^2})^{2-\theta},
\end{align*}
and therefore,
$$
\sum_{N \geq 1} N^{2s-\frac12} H_1  \lesssim \sum_{N \geq 1} N^{-\frac12 + s\theta} ( N^s \|P_N \zeta \|_{L_{\mathbf{x}}^2} )^{2-\theta}.
$$
Since by hypothesis $N^s \|P_N \zeta \|_{L_{\mathbf{x}}^2} \leq 1$, the above sum evaluates to $\lesssim 1$, provided we take $\theta< \frac{1}{2s}$. Thus, this contributes a constant term to the right side of \eqref{E:comp16}.  Finally, the terms in $G$ are straightforward to bound in \eqref{E:comp16}, using 
$$
\| \la x-a_1 \ra P_N q \|_{L_{\mathbf{x}}^2} \lesssim \|\tilde P_N \la x -a_1 \ra q \|_{L_{\mathbf{x}}^2},
$$
where $\tilde P_N$ is a new Littlewood-Paley multiplier, and that for all $\omega>0$,
$$
\| \tilde P_N [ \la x- a_1 \ra(\Lambda Q)_{x,\mathbf{a}}] \|_{L_{\mathbf{x}}^2} \lesssim_\omega N^{-\omega} \,, \qquad \| \tilde P_N [\la x- a_1 \ra (\nabla Q)_{x,\mathbf{a}}] \|_{L_{\mathbf{x}}^2} \lesssim_\omega N^{-\omega}.
$$ 
\end{proof}

\begin{lemma}
\label{L:comp6}
Suppose  $|I|\leq 1$,  \eqref{E:comp1} holds, \eqref{E:comp2b} holds for some $s\geq 1$, and thus, \eqref{E:comp7} holds.  Then
$$
\| \zeta \|_{L_I^\infty H_{\mathbf{x}}^{s+\frac12}} \lesssim 1,
$$
i.e., \eqref{E:comp2b} holds for $s \mapsto s+\frac12$.
\end{lemma}
\begin{proof}
This is pretty quickly done using the result of Lemma \ref{L:comp5} together with Lemma \ref{L:comp3} and the Ribaud-Vento well-posedness estimates.
\end{proof}

\section{Convergence of $w_n = \tilde \epsilon_n/B_n$ to $w$}
\label{S:convergence}

In this section, we prove Lemma \ref{L:convergence}.  Recall the setup from \S\ref{S:introduction}.  
Associated to $\tilde u_n$ are the parameters $\tilde c_n(t)$, $\tilde{\mathbf{a}}_n(t)$,  remainder $\tilde{\epsilon}_n(\mathbf{x},t)$,  and 
$$
b_n(t) = \|\tilde{\epsilon}_n(\mathbf{x},t) \|_{L_{\mathbf{x}}^2} \,, \qquad B_n = \|b_n(t) \|_{L_t^\infty}.
$$
The sequence has been shifted in time to arrange that $b_n(0) \geq \frac12 B_n$, and scaled and shifted in space to arrange that
$$\tilde c_n(0)=1 \quad \mbox{and} \quad \tilde{\mathbf{a}}_n(0) = 0.$$
As in \S\ref{S:geom-decomp}, we denote
\begin{equation}
\label{E:tilzeta}
\tilde{\eta}_n(\mathbf{x},t) = \tilde c_n^{\,-2} \,\tilde\epsilon_n ( \tilde c_n^{\,-1}\mathbf{x}(x-\tilde{\mathbf{a}}_n),t) \,, \qquad \tilde{\zeta}_n= B_n \tilde{\eta}_n.
\end{equation}
Note that
\begin{equation}
\label{E:con4}
\| \tilde \zeta_n(0) \|_{L_{\mathbf{x}}^2} = \frac{\| \tilde \epsilon_n(0) \|_{L_{\mathbf{x}}^2}}{B_n} = \frac{b_n(0)}{B_n} \geq \frac12.
\end{equation}
By Lemma \ref{L:ep-decay},
\begin{equation}
\label{E:con2}
\| \tilde \zeta_n(0) \|_{L_{\mathbf{x}}^2(|\mathbf{x}|\geq r)} \leq e^{-\delta r},
\end{equation}
and by Lemma \ref{L:ep-comparability}, for all $k \geq 0$,
\begin{equation}
\label{E:con3}
\| \tilde \zeta_n \|_{L_t^\infty H_{\mathbf{x}}^k} \lesssim_k 1.
\end{equation}
By \eqref{E:con2}, \eqref{E:con3} and the Rellich-Kondrachov theorem, we can pass to a subsequence (still indexed by $n$) so that
$$\tilde{\zeta}_n(0) \to \zeta_\infty(0)$$
strongly in $H_{\mathbf{x}}^k$, for every $k\geq 0$ (this is the \emph{definition} of $\zeta_\infty(0)$).  By \eqref{E:con4}, we have
$$\|\zeta_\infty(0)\|_{L_{\mathbf{x}}^2} \geq \frac12.$$
From \S \ref{S:geom-decomp}, \eqref{E:zeta-1} and \eqref{E:omegas}, we have
\begin{equation}
\label{E:con6}
\begin{aligned}[t]
\partial_t \tilde{\zeta}_n &= - \partial_x \Delta \tilde{\zeta}_n - 2 \partial_x ( Q_{\tilde c_n,\tilde{\mathbf{a}}_n} \tilde{\zeta}_n) + \tilde c_n^{-2} \la \tilde{\zeta}_n, f_{\tilde c_n,\tilde{\mathbf{a}}_n}\ra (\Lambda Q)_{\tilde c_n,\tilde{\mathbf{a}}_n} \\
& \qquad + \tilde c_n^{-2}\la \tilde{\zeta}_n, \mathbf{g}_{\tilde c_n,\tilde{\mathbf{a}}_n} \ra \cdot (\nabla Q)_{\tilde c_n,\tilde{\mathbf{a}}_n} - B_n \partial_x \tilde{\zeta}_n^2 + B_n \omega_{\tilde c_n} (\Lambda Q)_{\tilde c_n,\tilde{\mathbf{a}}_n} \\
& \qquad + B_n\boldsymbol{\omega}_{\tilde{\mathbf{a}}_n} \cdot (\nabla Q)_{{\tilde c_n},\tilde{\mathbf{a}}_n}.
\end{aligned}
\end{equation}

\begin{lemma}
\label{L:param-conv}
On $[-T,T]$, we have
$$
|\tilde c_n -1 | \lesssim TB_n \quad \mbox{and} \quad |\tilde{\mathbf{a}}_n - t \mathbf{i}| \lesssim \la T\ra^2 B_n.
$$
Consequently, if $F(\mathbf{x})$ is smooth, for any $k\geq 0$
$$
\| F_{\tilde c_n, \tilde{\mathbf{a}}_n} - F_{1,t \mathbf{i}} \|_{L_T^\infty H_{\mathbf{x}}^k} \lesssim_k \la T \ra^2 B_n.
$$
\end{lemma}
\begin{proof}
This follows from Lemma \ref{L:ODE-bounds}.
\end{proof}

By making the formal substitutions
$$\tilde{c}_n \to 1 \,, \quad \tilde{\mathbf{a}}_n \to t \mathbf{i}\,, \quad  \tilde{\zeta}_n \to \zeta_\infty \,, \quad F_{\tilde c_n, \tilde{\mathbf{a}}_n} \to F_{1,t\mathbf{i}}\,, \quad B_n \to 0,$$
where $F$ takes the place of $\Lambda Q$, $\nabla Q$, $Q$, $f$, or $\mathbf{g}$, we obtain that the expected limit $\zeta_\infty(t)$ of $\tilde{\zeta}_n(t)$ should solve
\begin{equation}
\label{E:con1}
\partial_t \zeta_\infty = - \partial_x \Delta \zeta_\infty - 2 \partial_x ( Q_{1,t \,\mathbf{i}} \zeta_\infty) +  \la \zeta_\infty , f_{1,t \,\mathbf{i}} \ra (\Lambda Q)_{1,t \mathbf{i}} + \la \zeta_\infty, \mathbf{g}_{1,t\,\mathbf{i}} \ra \cdot (\nabla Q)_{1,t \,\mathbf{i}} .
\end{equation}

Let $\zeta_\infty$ solve \eqref{E:con1} with initial condition $\zeta_\infty(0)$. [The well-posedness of \eqref{E:con1} can be proved in $C([-T,T];H_{\mathbf{x}}^k)$ using the Ribaud \& Vento estimates.]  We prove that, for each $T>0$ and each $k\geq 0$, 
\begin{equation}
\label{E:con8}
\tilde{\zeta}_n \to \zeta_\infty \text{ in } C([-T,T]; H_{\mathbf{x}}^k)
\end{equation}
as follows.  Let 
$$
\hat \zeta_n \defeq \tilde{\zeta}_n -  \zeta_\infty \quad \mbox{and} \quad \hat F_n = F_{\tilde{c}_n, \tilde{\mathbf{a}}_n} - F_{1,t\mathbf{i}},
$$
where $F$ takes the place of $\Lambda Q$, $\nabla Q$, $Q$, $f$, and $\mathbf{g}$.  In \eqref{E:con6}, for all terms without a $B_n$ coefficient, start by substituting 
$$F_{\tilde{c}_n, \tilde{\mathbf{a}}_n} = \hat F_n + F_{1,t\mathbf{i}}$$
to obtain
\begin{equation}
\label{E:con20}
\partial_t \tilde \zeta_n = - \partial_x \Delta \tilde \zeta_n - 2 \partial_x ( Q_{1,t\mathbf{i}} \tilde \zeta_n) + \la \tilde \zeta_n, f_{1,t\mathbf{i}} \ra (\Lambda Q)_{1,t\mathbf{i}} +  \la \tilde \zeta_n, \mathbf{g}_{1,t\mathbf{i}} \ra \cdot (\nabla Q)_{1, t \mathbf{i}} + G_n,
\end{equation}
where 
\begin{align*}
G_n = &-2\partial_x ( \hat Q_n \tilde \zeta_n) \\
&\hspace{5mm} +( \tilde c_n^{-2}-1) \la \tilde \zeta_n, f_{\tilde c_n, \tilde{\mathbf{a}}_n} \ra (\Lambda Q)_{\tilde c_n, \tilde{\mathbf{a}}_n} + \la \tilde \zeta_n, \hat f_n \ra (\Lambda Q)_{\tilde c_n, \tilde{\mathbf{a}}_n} + \la \tilde \zeta_n, f_{1,t \,\mathbf{i}} \ra \widehat{\Lambda Q}_n \\
&\hspace{5mm} +( \tilde c_n^{-2}-1) \la \tilde \zeta_n, \mathbf{g}_{\tilde c_n, \tilde{\mathbf{a}}_n} \ra \cdot (\nabla Q)_{\tilde c_n, \tilde{\mathbf{a}}_n} + \la \tilde \zeta_n, \hat{\mathbf{g}}_n \ra \cdot (\nabla Q)_{\tilde c_n, \tilde{\mathbf{a}}_n} + \la \tilde \zeta_n, \mathbf{g}_{1,t\,\mathbf{i}} \ra \cdot \widehat{\nabla Q}_n  \\
&\hspace{5mm} - B_n \partial_x \tilde{\zeta}_n^2 + B_n \omega_{\tilde c_n} (\Lambda Q)_{\tilde c_n,\tilde{\mathbf{a}}_n}  + B_n\boldsymbol{\omega}_{\tilde{\mathbf{a}}_n} \cdot (\nabla Q)_{{\tilde c_n},\tilde{\mathbf{a}}_n}.
\end{align*}
Since each term involves either $\tilde c_n -1$, $\hat F_n$, or a $B_n$ coefficient, Lemma \ref{L:param-conv} and \eqref{E:con3} implies
$$\|G_n \|_{H^k} \lesssim_k \la T \ra^2 B_n $$
for all $k\in \mathbb{N}$.    Taking the difference between \eqref{E:con20} and \eqref{E:con1}, we get
\begin{equation}
\label{E:con21}
\partial_t  \hat \zeta_n = - \partial_x \Delta  \hat \zeta_n - 2 \partial_x ( Q_{1,t\mathbf{i}}  \hat \zeta_n) + \la  \hat \zeta_n, f_{1,t\mathbf{i}} \ra (\Lambda Q)_{1,t\mathbf{i}} +  \la  \hat \zeta_n, \mathbf{g}_{1,t\mathbf{i}} \ra \cdot (\nabla Q)_{1, t \mathbf{i}} + G_n.
\end{equation}
We then compute
$$
\partial_t \| \nabla^k \hat \zeta_n \|_{L^2_{\mathbf{x}}}^2,
$$
then simplify with integration by parts, and apply Gronwall's inequality, to obtain
$$
\| \nabla^k \hat \zeta_n \|_{L_{[-T,T]}^\infty L^2_{\mathbf{x}}}^2 \lesssim e^{CT} (\| \nabla^k \hat \zeta_n(0) \|_{L_{[-T,T]}^\infty L^2_{\mathbf{x}}}^2+ B_n).
$$
Consequently, \eqref{E:con8} holds.  By \eqref{E:con3}, it follows that
\begin{equation}
\label{E:con10}
\| \zeta_\infty \|_{L_t^\infty H_{\mathbf{x}}^k} \lesssim_k 1.
\end{equation}
Note that
$$
w_n(\mathbf{x},t) = \frac{\tilde \epsilon_n(\mathbf{x},t)}{B_n}  = \tilde c_n^2 \tilde \zeta_n( \tilde c_n\mathbf{x}+ \tilde{\mathbf{a}}_n, t).
$$
Let
$$w(\mathbf{x},t) \defeq \zeta_\infty(\mathbf{x}+t \mathbf{i},t).$$
Then \eqref{E:con8} implies 
\begin{equation}
\label{E:con9}
w_n \to w \text{ in } C([-T,T]; H_{\mathbf{x}}^k)
\end{equation}
and \eqref{E:con10} implies
\begin{equation}
\label{E:con11}
\|w \|_{L_t^\infty H_{\mathbf{x}}^k} \lesssim_k 1.
\end{equation}
By Lemma \ref{L:ep-decay}, we have
$$
\| w_n \|_{L_{\mathbf{x}}^2(|\mathbf{x}| \geq r)} \lesssim e^{-\delta r}.
$$
By \eqref{E:con9}, we obtain
$$
\| w \|_{L_{\mathbf{x}}^2(|\mathbf{x}| \geq r)} \lesssim e^{-\delta r}.
$$
The equation \eqref{E:con1} converts to the equation for $w$ in the statement of Lemma \ref{L:convergence}.  Moreover, since $\tilde \epsilon_n$ satisfies the orthogonality conditions for each $n$, $w_n$ also satisfies them, and hence, the limit $w$ does as well.  This completes the proof of Lemma \ref{L:convergence}.  

\section{Proof of the linear Liouville lemma assuming the viral estimate}
\label{S:linear-Liouville}

In this section, we prove Lemma \ref{L:linear-Liouville}, the linear Liouville theorem.  We first note that 
\begin{equation}
\label{E:w1}
\partial_t \left[ \la \mathcal{L}w,w \ra + \frac{2}{\la \Lambda Q, Q \ra} \la w, Q \ra^2 \right] = 0,
\end{equation}
which follows from a straightforward computation substituting the equation \eqref{E:w-eq} for $w$ and applying the orthogonality conditions \eqref{E:extra-orth}.    This of course means that the expression $\ds \la \mathcal{L}w, w \ra + \frac{2}{\la \Lambda Q, Q \ra} \la w, Q \ra^2$ is constant in time. 

We observe that from the definition of $\mathcal{L}$ and integration by parts 
\begin{equation}
\label{E:w2}
\int_{t=-\infty}^{+\infty} \left( \la \mathcal{L}w,w \ra + \frac{2}{\la \Lambda Q, Q \ra} \la w, Q \ra^2 \right) \, dt 
\lesssim \| w\|_{L^2_t H^1_{\bf x}}^2 \,.
\end{equation}

Lemma \ref{L:v-virial} (proved in the next section) shows that for the dual problem $v = \mathcal{L}w$ we have the estimate
$$
\| v \|_{L_t^2 H_{\bf x}^1} \lesssim \| \la x \ra^{1/2} v \|_{L_t^\infty L_{\bf x}^2},
$$ 
which by Lemma \ref{L:conversion} implies the following bound for $w$: 
$$
\|  w \|_{L_t^2 H_{\bf x}^1} \leq \|w\|_{L_t^2 H_{\bf x}^3} \lesssim \| \la x \ra^{1/2} w \|_{L_t^\infty H_{\bf x}^2},
$$ 
which is finite by \eqref{E:w-dec}. Thus, the last term in \eqref{E:w2} is bounded, and hence, the integrand in the left-hand side of \eqref{E:w2} given by $\ds \la \mathcal{L}w, w \ra + \frac{2}{\la \Lambda Q, Q \ra} \la w, Q \ra^2$, which is constant in time, must be zero.   
Since $\la \Lambda Q, Q \ra =\frac12 \|Q\|^2_{L^2} > 0$ (subcritical case), the quantity is positive definite, and we conclude that both
$$
\la \mathcal{L}w, w \ra = 0 \quad \mbox{and} \quad \la  w, Q \ra = 0.
$$
By the orthogonality conditions, $\mathcal{L}$ is strictly positive definite, which implies that $w\equiv 0$.

\section{Proof of the viral estimate}
\label{S:virial}

In this section we prove Lemma \ref{L:virial}, which is just a combination of  Lemma \ref{L:conversion} and Lemma \ref{L:v-virial} below.  Lemma \ref{L:conversion} reduces the inequality to a statement about a dual function $v=\mathcal{L}w$, and Lemma \ref{L:v-virial} achieves the inequality for the dual function $v$ by invoking the results from the numerical verification in \S\ref{S:numerics} and by applying an ``angle lemma'' (Lemma \ref{L:angle}).  

We will start with the conversion lemma:

\begin{lemma}[conversion]\label{L:conversion}
Suppose that $w$ satisfies  $\la w, \nabla Q \ra =0$
and $v=\mathcal{L}w$.  If $v$ satisfies the global-in-time estimate
$$
\|v \|_{L_t^2 H_\mathbf{x}^1} \lesssim \| \la x \ra^{1/2} v\|_{L_t^\infty L_\mathbf{x}^2},
$$
then it follows that $w$ satisfies the global-in-time estimate
$$
\| w\|_{L_t^2 H_{\mathbf{x}}^3} \lesssim \| \la x \ra^{1/2} w \|_{L_t^\infty H_\mathbf{x}^2}.
$$
\end{lemma}
\begin{proof}
Since $\mathcal{L}$ is a self-adjoint Schr\" odinger operator with smooth rapidly decaying potential, its spectrum consists of $[1,+\infty)$ plus a finite number of eigenvalues.  It follows that the spectrum of $\mathcal{L}^2$ is $[1,+\infty)$ plus the square of the eigenvalues of $\mathcal{L}$.   Since $\ker L = \spn \{ \nabla Q \}$, $\ker \mathcal{L}^2 = \spn \{ \nabla Q \}$, and there is a positive gap to the next eigenvalue of $\mathcal{L}^2$.  Consequently, $\mathcal{L}^2$ is strictly positive on the orthocomplement of $\nabla Q$: there exists $\delta>0$ such that 
\begin{equation}
\label{E:conv1}
 \delta \|w \|_{L^2}^2 \leq  \la \mathcal{L}^2 w, w \ra = \| \mathcal{L}w \|_{L^2}^2 = \| v \|_{L^2}^2.
\end{equation}
It is straightforward that, for some $\kappa>0$,
\begin{equation}
\label{E:conv2}
\| w \|_{H^3}^2 \leq  \| \mathcal{L}w \|_{H^1}^2 + \kappa \| w\|_{L^2}^2 =   \| v \|_{H^1}^2 + \kappa \| w\|_{L^2}^2.
\end{equation}
Combining \eqref{E:conv1} and \eqref{E:conv2}, we obtain
$$
\| w \|_{H^3} \lesssim \| v \|_{H^1}.
$$
It is also straightforward that
$$
\|\la x \ra^{1/2} v \|_{L^2} = \| \la x \ra^{1/2} \mathcal{L} w \|_{L^2}  \lesssim \| \la x \ra^{1/2} w \|_{H^2}.
$$
\end{proof}


We provide here a statement of the elementary angle lemma--for proof see \cite{FHRY}.

\begin{lemma}[angle lemma]
\label{L:angle}
Suppose that $A$ is a self-adjoint operator on a Hilbert space $H$ with eigenvalue $\lambda_1$ and corresponding eigenspace spanned by a  function $e_1$ with $\|e_1\|_{L^2}=1$.  Let $P_1f = \la f,e_1\ra e_1$ be the corresponding orthogonal projection.   Assume that $(I-P_1)A$ has spectrum bounded below by $\lambda_\perp$, with $\lambda_\perp>\lambda_1$.  Suppose that $f$ is some other function such that $\|f\|_{L^2}=1$ and $0 \leq \beta \leq \pi$ is defined by $\cos \beta = \la f, e_1\ra$.   Then if $v$ satisfies $\la v, f \ra =0$, we have
$$
\la Av,v \ra \geq (\lambda_\perp - (\lambda_\perp - \lambda_1)\sin^2\beta)\|v\|_{H}^2.
$$
\end{lemma}


We are now ready to prove the virial estimate for $v$.

\begin{lemma}[linearized virial  estimate for $v$]
\label{L:v-virial}
Suppose that $v \in C^0(\mathbb{R}_t; H_{\bf{x}}^1) \cap C^1(\mathbb{R}_t; H_{\bf{x}}^{-2})$  solves
$$
\partial_t v =  \mathcal{L}\partial_x v  -2 \alpha Q
$$
for some time dependent coefficient $\alpha$, and moreover, $v$ satisfies the orthogonality conditions 
$$
\la v, Q \ra =0\quad \mbox{and} \quad \la v, \nabla Q \ra =0. 
$$
Then 
\begin{equation}
\label{E:v-virial-1}
\| v\|_{L_t^2H_{\bf x}^1} \lesssim \| \la x \ra^{1/2} v \|_{L_t^\infty L_{\bf x}^2},
\end{equation}
where $t$ is carried out over all time $-\infty < t< \infty$.
\end{lemma}
\begin{proof}
Using the orthogonality condition $\la v,Q \ra =0$, we compute
$$
0= \partial_t \la v, Q \ra = \la \mathcal{L}\partial_x v, Q \ra -2 \alpha \la Q, Q \ra.
$$
This yields 
$$
\alpha = \frac{\la v,  Q \, Q_x \ra}{\la Q, Q \ra}
$$
so that
\begin{equation}\label{E:P}
\partial_t v = \mathcal{L} \partial_x v - 2 \frac{\la v, Q \, Q_x \ra}{\la Q, Q\ra}\, Q.
\end{equation}
Now compute
\begin{equation}
\label{E:v-virial-2}
-\frac12 \partial_t \int x v^2 = \la Bv,v\ra + \la P v,v\ra,
\end{equation}
where
$$
B = \frac12 - \frac32\partial_x^2 - \frac12\partial_y^2 - \frac12\partial_z^2 - (x \,Q)_x
$$
and from \eqref{E:P} $P$ can be taken as the rank $2$ self-adjoint operator
$$
Pv = \frac{ Q \, Q_x}{\la Q, Q\ra} \la v, xQ\ra + \frac{xQ}{\la Q, Q\ra} \la v, Q \, Q_x\ra.
$$

The continuous spectrum of $A=B+P$ is $[\frac12,+\infty)$.  Via a numerical solver we find the eigenvalues and corresponding eigenfunctions below $\frac12$ (the details are given in \S \ref{S:numerics} below).

We obtain two simple eigenvalues below $\frac12$, namely,
$$
\lambda_1=-0.0294 \text{ and } \lambda_2=-0.4688.
$$  
Denoting the corresponding normalized eigenfunctions by $f_1$ and $f_2$,  and $g_1 = \frac{Q}{\|Q\|}$ and $g_2 = \frac{Q_x}{\|Q_x\|}$, we find
$$\la f_1, g_1 \ra =0.9946 \,, \qquad \la f_1,g_2 \ra = 0,$$
$$\la f_2, g_1 \ra = 0  \,, \qquad \la f_2, g_2 \ra = -0.7922.$$

Following the $L^2$ decomposition as in \cite[Lemma 14.2]{FHRY}, we consider the closed subspace $H_o$ of $L^2(\mathbb{R}^3)$ given by functions that are odd in $x$ (no constraint in $y$ or $z$), and the closed subspace $H_e$ of $L^2(\mathbb{R}^3)$ given by functions that are even in $x$ (no constraint in $y$ or $z$). Note that $L^2(\mathbb{R}^3) = H_o \oplus H_e$ is an orthogonal decomposition.  Observe that $f_1$ and $g_2$ belong to $H_o$, while $f_2$ and $g_1$ belong to $H_e$.   Thus, $A\big|_{H_o}$ has spectrum $\{ \lambda_1 \} \cup [\frac12, +\infty)$ with $f_1$ being the eigenfunction corresponding to $\lambda_1$.  
Applying the angle lemma (Lemma \ref{L:angle} or \cite[Lemma 14.3]{FHRY}) with $H=H_o$ and $\lambda_\perp = \frac12$, and noting that
$$
(\lambda_\perp - \lambda_1) \sin^2\beta = (0.5 + 0.0294)*(1-0.9946^2) = 0.0057,
$$
we find that
$$
\la AP_o v, P_o v \ra \geq (0.5000- 0.0057) \la P_ov, P_o v \ra = 0.4943 \, \la P_ov, P_o v \ra.
$$
Also, $A\big|_{H_e}$ has spectrum $\{ \lambda_2 \} \cup [\frac12, +\infty)$ with the eigenfunction $f_2$ corresponding to $\lambda_2$.  
Applying the angle lemma with $H=H_e$ and $\lambda_\perp = \frac12$, we get
$$
(\lambda_\perp - \lambda_2) \sin^2\beta = (0.5000-0.4688)*(1-0.7922^2)=0.0116,
$$ 
and 
$$
\la AP_e v, P_e v \ra \geq (0.5000-0.0116)\la P_e v, P_e v\ra=0.4884\,\la P_e v,P_e v\ra.
$$

Thus $A=B+P$ is positive (assuming $v$ satisfies the two orthogonality conditions). Integrating \eqref{E:v-virial-2} in time and using elliptic regularity, we obtain \eqref{E:v-virial-1}.
\end{proof}

\section{Verification of spectral property} 
\label{S:numerics}


\subsection{Set up}
Here, we discuss how we find the eigenvalues and eigenfunctions of the operator $2(B+P)$ in 3d (for computational convenience, we doubled the operator; thus, the continuous spectrum will start from 1):
\begin{align}
2(B+P) \defeq-3\partial_{xx}-\partial_{yy} -\partial_{zz} + 1 - 2(x\, Q)_x + 2P,
\end{align}
where $P$ is defined as 
\begin{align}
Pv= \frac{Q \, Q_x}{\|Q\|_2^2}\langle v, xQ \rangle + \frac{xQ}{\|Q\|_2^2} \langle v, Q\, Q_x \rangle.
\end{align}

We follow our approach from \cite[Section 16]{FHRY} and investigate the spectrum of the operator $2(B+P)$. Similar to the 2d case we use the collocation method, however, due to the computational limitations in 3d, we can only apply a few collocation points for each axis ($x$, $y$ and $z$). In this computations $N=36$ in each dimension is the maximum number that we could reach, though we show that even with that many points, the results are robust and truthful. 
To arrange the Chebyshev collocation points to be more concentrated at the center we need a specific mapping, we use a similar approach as in the 2d case:
\begin{align}\label{D: x mapping}
x(\xi)=L\frac{e^{a\xi}-e^{-a\xi}}{e^{a}-e^{-a}},
\end{align}
with $\xi \in [-1,1]$ and $a$ is the parameter that we can chose (in our computation we take $a=4$ or $a=5$). By the chain rule, the partial derivatives $\partial_x$, $\partial_{xx}$ are
\begin{align}\label{E:D1}
\partial_x=\frac{1}{x_{\xi}}\partial_{\xi},
\end{align}
and
\begin{align}\label{E:D2}
\partial_{xx}=\frac{1}{x_{\xi}^2}+\left(\partial_{\xi}(\frac{1}{x_{\xi}})\cdot\frac{1}{x_{\xi}}\right)\partial_{\xi}.
\end{align}
We apply similar mapping and calculation to the $y$-direction as well as the $z$-direction.

Now, we need to discretize the operator $2(B+P)$ with the mapped-Chebyshev collocation points.
The discretization of the operator $B$ as well as imposing the homogeneous Dirichlet boundary conditions are quite standard, for example, we follow the same approach as in \cite[Chapters 6, 9, 12]{T2001}. It follows similar steps as we had in the 2D case \cite{FHRY} (and we described  a general formula for discretizing the projection operator), for completeness, we outline the process here. 

First, we consider the 1D case. Then the extension to the cases $d\geq 2$ is done by standard numerical integration technique for multi-dimensions, e.g., see \cite[Chapters 6, 12]{T2001}.
We denote by $f_i$ 
the discretized form of the function $f(x)$ at the point $x_i$, and we write the vector $\vec{f}$ for $\vec{f}=(f_0,f_1,\cdots,f_N)^T$. We denote the operation ``$.*$" to be the pointwise multiplication of the vectors or matrices with the same dimension, i.e., $\vec{a}.*\vec{b}=(a_0b_0,\cdots,a_Nb_N)^{\mathrm{T}}$; the notation ``$*$" stands for the regular vector or matrix multiplication.

Let $w(x)$ to be the weights for a given quadrature. For example, if we consider the composite trapezoid rule with step-size $h$, we have
\begin{align*}
\vec{w}=(w_0,w_1,\cdots ,x_N)^T=\frac{h}{2}(1,2,\cdots,2,1)^T,
\end{align*}
since the composite trapezoid rule can be written as
\begin{align*}
\int_a^b f(x) dx \approx \sum_{i=0}^{N} f_i w_i=\vec{f}^{\,\, T} *\vec{w}.
\end{align*} 
To evaluate a Chebyshev Gauss-Lobatto quadrature, which we need for this work, we write 
$$
\int_{-1}^1 f(x) dx \approx \sum_{i=0}^N w_i f(x_i)=\vec{f}^{\,\, T} * \vec{w},
$$
where $w_i=\frac{\pi}{N}\sqrt{1-x_i^2}$ for $i=1,2,\cdots, N-1$, and $w_0=\frac{\pi}{2N}\sqrt{1-x_0^2}$, $w_N=\frac{\pi}{2N}\sqrt{1-x_N^2}$, are the weights together with the weighted functions.
We have
\begin{align*}
Pu=&\langle u,f \rangle g= (\sum_{i=0}^N w_i \,f_i\, u_i) \, \vec{g} 
= \left[\begin{matrix}
g_0\\
g_1\\
\vdots\\
g_N
\end{matrix}\right]  \big(\sum_{i=0}^N w_i f_i u_i \big) = \left[\begin{matrix}
g_0\\
g_1\\
\vdots\\
g_N
\end{matrix}\right]  (\vec{w}^{\, T}.*\vec{f}^{\,\, T}) * \vec{u} 
:= \mathbf{P} \vec{u},
\end{align*}
with the matrix
\begin{align}\label{E: P term}
\mathbf{P}=\vec{g}*(\vec{w}^{\,\, T}.*\vec{f}^{\,\, T})
\end{align}
to be the discretized approximation form of the projection operator $P$. 
Denote by $\mathbf{D_x^{(2)}}$, $\mathbf{D_y^{(2)}}$, $\mathbf{D_z^{(2)}}$ the second order mapped-Chebyshev differential matrices coming from the equation \eqref{E:D2} (see also \cite{T2001}), the $x$-derivative of $Q$ as $\vec{Q}_x=\mathbf{D_x^{(1)}}\vec{Q}$, and the matrix $\mathbf{M}=2(\mathbf{B}+\mathbf{P})$. Then we obtain
\begin{align}\label{E: M-matrix}
\mathbf{M}=-3\mathbf{{D}_x^{(2)}}-\mathbf{{D}_y^{(2)}}-\mathbf{{D}_z^{(2)}}+diag(\vec{1}-3*\vec{Q^2}-6*\vec{x}.*\vec{Q}.*\vec{Q}_x)+\mathbf{P},
\end{align}
where $\mathbf{P}$ is the matrix form for the projection term discretized from \eqref{E: P term}, and $\vec{1}=(1,\cdots,1)^T$ is the vector with the same size of other variables (such as $\vec{Q}$). Before we proceed with spectral properties, we explain how we obtain the ground state $Q$. 

\subsection{Calculation of the ground state $Q$}
While we can calculate the ground state directly in the 3D space, the computational cost is very expensive. Applying the radial symmetry, we only need to compute the ground state in 1D radial case and interpolate it into the 3D space. The 1D radial equation for the ground state is as follows
\begin{align}\label{E: GS_radial}
-R_{rr}-\frac{2}{r}R_r+R-|R|^{p-1}R=0, \quad R_r(0)=0, \quad R(2 L)=0.
\end{align}
We choose the computational domain to be $r \in [0,2L)$ since $r=\sqrt{x^2+y^2+z^2}$, where each $x,y,z \in [-L,L]$. Therefore, the computational domain for $r$ has to be greater or equal to $\sqrt{3}L$ to avoid the extrapolation in the upcoming interpolation process.

Next, the equation \eqref{E: GS_radial} can be solved by using the renormalization method \cite[Chapter 24]{F2016}. For that we use the shape preserving cubic spline to interpolate the solution into the full three dimensional data. Suppose $\vec{r}=(r_0,r_1,\cdots, r_{N_r})^T$ to be the $N_r$ collocation points we used to compute the equation \eqref{E: GS_radial}, and $\vec{R}$ is the discretized solution of \eqref{E: GS_radial} from $\vec{r}$. Let $\vec{x}=(x_0,x_1,\cdots,x_N)^T $ with $x_0=-L$ and $x_N=L$ be the mapped Chebyshev collocation points we discussed previously. We generate the 3D tensor data by using the matlab command \texttt{meshgrid}
$$
[\mathbf{X},\mathbf{Y},\mathbf{Z}]=\operatorname{meshgrid}(\vec{x}).
$$
Then, the tensor data for $\mathbf{Q}$ (the 3D ground state $Q$), is obtained via the shape-preserving cubic spline interpolation with the matlab function \texttt{interp1} by
\begin{align*}
\mathbf{Q}=\operatorname{interp1}(\vec{r},\vec{R},\sqrt{\mathbf{X}^2 +\mathbf{Y}^2 +\mathbf{Z}^2},'pchip').
\end{align*}

\subsection{Spectrum} 
Let $N$ be the number of collocation points assigned for each dimensions (this will result in a $N^3 \times N^3$ matrix of $\mathbf{M}$). 
Let $M[R]$ be the mass of $Q$ computed from the radial solution $R$ by the composite trapezoid rule, and $M[Q]$ be the mass of $Q$ computed in full 3D 
by evaluating the Chebyshev-Gauss quadrature. We track a possible error generated by the interpolation via $\mathcal{E}=\| M[Q]-M[R] \|_{\infty}$. 

The matlab command ``\texttt{eigs}" produces the eigenvalues, and we consider only those, which are less than $1$. 
Taking a different number of collocation points $N$ for each direction ($x$, $y$ and $z$), and normalizing the $L^2$ norm of the corresponding eigenfunctions to $1$, we obtain the following: 

$\bullet$ \underline{\bf $N=16$}: $\mathcal{E}=0.17778$. 

The eigenvalues are
\begin{align}
\lambda_{1,2}=   -0.04938, \qquad 0.93316.
\end{align}
The angles with the eigenfunctions (and normalized $Q$ and $Q_x$) are 
\begin{align}
\left[ \begin{matrix}
\langle Q,  \phi_1 \rangle & \langle Q, \phi_2 \rangle \\
\langle Q_x, \phi_1 \rangle & \langle Q_x, \phi_2 \rangle \\
\end{matrix} \right] = \left[ \begin{matrix}
  -0.9952 &  -0.0000 \\
  0.0000 &  -0.7940 \\
\end{matrix} \right].
\end{align}

$\bullet$ \underline{\bf $N=21$}: $\mathcal{E}=0.0024339$. 

The eigenvalues are
\begin{align}
\lambda_{1,2}=   -0.052992, \qquad 0.9382.
\end{align}
The angles with the eigenfunctions are
\begin{align}
\left[ \begin{matrix}
\langle Q,  \phi_1 \rangle & \langle Q, \phi_2 \rangle \\
\langle Q_x, \phi_1 \rangle & \langle Q_x, \phi_2 \rangle \\
\end{matrix} \right] = \left[ \begin{matrix}
  0.9947 &  -0.0000 \\
  0.0000 &  -0.7918 \\
\end{matrix} \right].
\end{align}

$\bullet$ \underline{\bf $N=32$}: $\mathcal{E}=6.9879e-06$. 

The eigenvalues are
\begin{align}
\lambda_{1,2}=   -0.058808, \qquad 0.93757.
\end{align}
The angles with the eigenfunctions are
\begin{align}
\left[ \begin{matrix}
\langle Q,  \phi_1 \rangle & \langle Q, \phi_2 \rangle \\
\langle Q_x, \phi_1 \rangle & \langle Q_x, \phi_2 \rangle \\
\end{matrix} \right] = \left[ \begin{matrix}
  0.9946 &  -0.0000 \\
  0.0000 &  -0.7922 \\
\end{matrix} \right].
\end{align}

$\bullet$ \underline{\bf $N=36$}: $\mathcal{E}=1.6117e-06$. 

The eigenvalues are
\begin{align}
ss=   -0.058812, \qquad 0.93757.
\end{align}
The angles with the eigenfunctions are obtained as
\begin{align}
\left[ \begin{matrix}
\langle Q,  \phi_1 \rangle & \langle Q, \phi_2 \rangle \\
\langle Q_x, \phi_1 \rangle & \langle Q_x, \phi_2 \rangle \\
\end{matrix} \right] = \left[ \begin{matrix}
  0.9946 &  -0.0000 \\
  0.0000 &  -0.7922 \\
\end{matrix} \right].
\end{align}

Finally, we conclude that the eigenfunction $\phi_1$, corresponding to 
$\lambda_1$, the negative eigenvalue, is (almost) orthogonal to $Q$, and the second eigenfunction $\phi_2$ is (almost) orthogonal to $Q_x$. 
We also note that while we do not use a large number of points, our numerical findings become consistent with an increasing $N$ (see the consistency for $N=32$ and $N=36$).

\end{document}